\documentclass{amsart}
\pdfoutput=1
\usepackage{amsfonts,amssymb,amsthm,amsmath,euscript}
\usepackage{graphicx}
\usepackage{tikz}
\usetikzlibrary{calc}
\usepackage{pgfplots}
\usepackage{pinlabel}
\usepackage[colorlinks,citecolor=blue,linkcolor=red]{hyperref}
\input{xy}
\xyoption{all}

\theoremstyle{plain}
\newtheorem{theorem}{Theorem}[section]
\newtheorem{lemma}[theorem]{Lemma}
\newtheorem{corollary}[theorem]{Corollary}

\newtheorem{proposition}[theorem]{Proposition}
\newtheorem{claim}[theorem]{Claim}
\newtheorem*{theorem:tnorm_main}{Theorem \ref{T:folded basics}}
\newtheorem*{theorem:iwip_main}{Theorem \ref{T:new iwip main}}
\newtheorem*{theorem:fried_main}{Theorem \ref{T:new entropy main}}
\newtheorem{theor}{Theorem} 

\newtheorem*{theorem:folded basics}{Theorem \ref{T:folded basics}}
\newtheorem*{theorem:new iwip main}{Theorem \ref{T:new iwip main}}
\newtheorem*{theorem:new entropy main}{Theorem \ref{T:new entropy main}}
\newtheorem*{theorem:Topological graph to graph}{Theorem \ref{T:topological graph to graph}}
\newtheorem*{proposition:h(f)}{Proposition \ref{prop:h(f)}}
\newtheorem*{theorem:clean}{Theorem \ref{T:clean}}

\theoremstyle{definition}
\newtheorem{defn}[theorem]{Definition}
\newtheorem{remark}[theorem]{Remark}
\newtheorem{conv}[theorem]{Convention}
\newtheorem{example}[theorem]{Example}
\newtheorem*{running-example}{Example~\ref{Ex:perpetuating_example} cont}

\numberwithin{equation}{section}

\newcommand{\Out}{\mathrm{Out}}
\newcommand{\Aut}{\mathrm{Aut}}

\newcommand{\vol}{\mathrm{vol}}

\newcommand{\MM}{{\mathbb M}}
\newcommand{\M}{{\mathcal M}}
\newcommand{\R}{{\mathbb R}}
\newcommand{\Z}{{\mathbb Z}}
\newcommand{\C}{\EuScript{C}}
\newcommand{\A}{{\mathcal A}}

\newcommand{\wh}{\mathrm {Wh}}
\newcommand{\Hom}{\mathrm {Hom}}
\newcommand{\f}{f} 
\newcommand{\fee}{\varphi} 
\newcommand{\fib}{\eta}    
\newcommand{\rk}{\mathrm{rk}}
\newcommand{\link}{\mathrm{Lk}}
\newcommand{\turns}{\mbox{Turns}}
\newcommand{\flow}{\psi}   
\newcommand{\Flow}{\Psi} 
\newcommand{\norm}{{\mathfrak n}} 

\newcommand{\ma}{\mathfrak m}  

\newcommand{\len}{\mbox{len}}
\newcommand{\val}{\mbox{val}}

\begin{document}

\title{Dynamics on free-by-cyclic groups}
\author{Spencer Dowdall, Ilya Kapovich, and Christopher J. Leininger}

\address{\tt Department of Mathematics, University of Illinois at
  Urbana-Champaign, 1409 West Green Street, Urbana, IL 61801, U.S.A.;  http://www.math.uiuc.edu/\~{}dowdall/} \email{\tt dowdall@illinois.edu}

\address{\tt Department of Mathematics, University of Illinois at
  Urbana-Champaign, 1409 West Green Street, Urbana, IL 61801, U.S.A.;
  http://www.math.uiuc.edu/\~{}kapovich/} \email{\tt kapovich@math.uiuc.edu}

\address{\tt Department of Mathematics, University of Illinois at
  Urbana-Champaign, 1409 West Green Street, Urbana, IL 61801, U.S.A.;  http://www.math.uiuc.edu/\~{}clein/} \email{\tt clein@math.uiuc.edu}

\begin{abstract}
Given a free-by-cyclic group $G = F_N \rtimes_\fee \Z$ determined by any outer automorphism $\fee \in \Out(F_N)$ which is represented by an expanding irreducible train-track map $f$, we construct a $K(G,1)$ $2$--complex $X$ called the {\em folded mapping torus} of $f$, and equip it with a semiflow.   We show that $X$ enjoys many similar properties to those proven by Thurston and Fried for the mapping torus of a pseudo-Anosov homeomorphism.  In particular, we construct an open, convex cone $\A \subset H^1(X;\R) = \Hom(G;\R)$ containing the homomorphism $u_0 \colon G \to \Z$ having $\ker(u_0) = F_N$, a homology class $\epsilon \in H_1(X;\R)$, and a continuous, convex, homogeneous of degree $-1$ function $\mathfrak H\colon\A \to \R$ with the following properties.  Given any primitive integral class $u \in \A$ there is a graph $\Theta_u \subset X$ such that: (1) the inclusion $\Theta_u \to X$ is $\pi_1$--injective and $\pi_1(\Theta_u) = \ker(u)$, (2)  $u(\epsilon) = \chi(\Theta_u)$, (3) $\Theta_u \subset X$ is a section of the semiflow and the first return map to $\Theta_u$ is an expanding irreducible train track map representing $\fee_u \in \Out(\ker(u))$ such that $G = \ker(u) \rtimes_{\fee_u} \Z$, (4) the logarithm of the stretch factor of $\fee_u$ is precisely $\mathfrak H(u)$, (5) if $\fee$ was further assumed to be hyperbolic and fully irreducible then for every primitive integral $u\in \A$ the automorphism $\fee_u$ of $\ker(u)$ is also hyperbolic and fully irreducible.
\end{abstract}


\thanks{The first author was partially supported by the NSF postdoctoral fellowship, NSF MSPRF No. 1204814. The second author was partially supported by the NSF grants DMS-0904200 and DMS-1405146 and by the Collaboration Grant no. 279836 from the Simons Foundation. The third author was partially supported by the NSF grant DMS-1207183}

\maketitle

\tableofcontents

\section{Introduction}
\label{sec:intro}
Given a group $K$ and $\fee \in \Out(K)$ represented by the automorphism $\Phi \colon K \to K$, we denote the {\em mapping torus group} of $\fee$ as the semidirect product
\[ G = G_\fee = K \rtimes_\fee \Z = \{ w, t \mid t^{-1} w t = \Phi(w) \text{ for all } w\in K \}.\]
The group $G$ and the {\em associated homomorphism} $u_0 \colon G \to
\Z$ given by $u_0(t^n w) = n$ depend only on $\fee$, and not on
$\Phi$.  Conversely, given a {\em primitive integral} $u \in
\Hom(G,\R) = H^1(G;\R)$ (a homomorphism with $u(G) = \Z$), there is an
{\em associated splitting} $G = \ker(u) \rtimes_{\fee_u} \Z$ with
$\fee_u \in \Out(\ker(u))$ defining the splitting and with associated
homomorphism $u$. Namely, if $t_u\in G$ is such that $u(t_u)=1\in\Z$,
then $\fee_u   \in \Out(\ker(u))$ is the outer automorphism class of
the automorphism $\Phi_u$ of $\ker(u)$ defined as
$\Phi_u(g)=t_u^{-1}gt_u$, where $g\in \ker(u)$.
Of particular interest is the case when $\ker(u)$ is finitely
generated.  According to a result of Neumann \cite{NeuNormal}, as well
as more general results of Bieri--Neumann--Strebel \cite{BNS} and Levitt~\cite{Levitt87}, for finitely generated $G$ there exists an open cone $\C_G \subseteq
H^1(G;\R)$ (a subset invariant by positive scalar multiplication), which we call the \emph{f.g.~cone of $G$}, such that a primitive integral element $u\in H^1(G;\R)$ has
finitely generated kernel if and only if $u \in \C_G$.

When $G = \pi_1(M)$, where $M$ is a closed oriented $3$--manifold, Stallings
\cite{StallFib} proved that every primitive integral element $u \in \C_G$  is induced by a fibration $M \to \mathbb{S}^1$.  The fiber $S_u
\subset M$ is then a surface with $\pi_1(S_u) = \ker(u)$, and $\fee_u$ is
induced by a homeomorphism $F_u \colon S_u \to S_u$ called the {\em
  monodromy} for $u$. Suppose that for some $u_0 \in \C_G$, the monodromy $F_{u_0}$ is {\em pseudo-Anosov} and let $\C_0 \subset \C_G$ denote the component containing $u_0$.  Through a detailed analysis of $M$, Thurston \cite{ThuN,FLP} and Fried \cite{FriedS,FriedD} proved that the monodromies for all primitive integral elements in  $\C_0$ are intimately related to one another in important topological, geometric and dynamical ways. For example, the assignments $u\mapsto \chi(S_u)$ and $u\mapsto h(F_u)$ for primitive integral $u\in \C_0$, where $h(F_u)$ denotes the topological entropy of $F_u$, both extend to convex functions on all of $\C_0$.

In this paper, we provide a model for, and a similar analysis of, the situation where the $3$--manifold group is replaced by a {\em free-by-cyclic group}, $G = F_N \rtimes_\fee \Z$, where $F_N$ is a free group of rank $2 \leq N < \infty$ and $\fee \in \Out(F_N)$. Free-by-cyclic groups share many features with fundamental groups of compact 3--manifolds fibering over the circle (for example, see \cite{AR,FH99,KL12}), although the behavior in the free-by-cyclic case is generally more varied and more complicated.

Let $\fee \in \Out(F_N)$ be an arbitrary element and let $f\colon\Gamma\to\Gamma$ be a {\em tame graph-map} which is a {\em topological representative} of $\fee$; see Definitions~\ref{D:top-graph-map}, \ref{D:reg-exp-graph-map}, \ref{D:tame}, and~\ref{D:TR} (such representatives always exist and are easy to build; see Remark~\ref{R:tame}).
Let $G=G_\fee$ be the mapping torus group of $\fee$  and let $u_0 \in H^1(G_\fee;\R)$ be the associated homomorphism.  From $f$ (and a folding sequence for $f$; see \S\ref{sect:folding}) we construct $(X,\flow,\epsilon,\A)$ where
\begin{itemize}
\item $X$ is a $K(G,1)$ polyhedral $2$--complex,
\item $\flow$ is a semiflow on $X$---that is, an action of the additive semigroup $\R_+$,
\item $\epsilon \in Z_1(X;\R)$ is a cellular cycle,
\item $\A = \A_X \subset H^1(G;\R)$ is an open, convex cone containing $u_0$,
\end{itemize}
A {\em section} of $\flow$ is a subset $\Theta \subset X$ such that every flow line meets $\Theta$ infinitely often at a discrete set of times.  Consequently any section has a {\em first return map} $\Theta \to \Theta$.  With this terminology and notation, our results can be summarized as follows.

\begin{theor} \label{T:main theorem}  Let $N \geq 2$, let $\fee \in
  \Out(F_N)$ be an arbitrary element represented by a tame graph-map $f\colon\Gamma\to\Gamma$ as above, and let $G = F_N
  \rtimes_\fee \Z$.  Then there exists $(X,\flow,\epsilon,\A)$ as above such that for any primitive integral $u \in \A$, there
  exists a compact connected topological graph $\Theta_u \subset X$ with the following properties:
\begin{enumerate}
\item The inclusion $\Theta_u \subset X$ is $\pi_1$-injective and, with
  the identification $\pi_1(X)=G$ we have $\pi_1(\Theta_u)=\ker(u)$,
\item $u(\epsilon) = \chi(\Theta_u) = 1 - \rk(\ker(u))$,
\item $\Theta_u$ is a section of the semiflow $\flow$ and the first return map $f_u \colon \Theta_u \to \Theta_u$ represents $\fee_u \in \Out(\ker(u)) = \Out(\pi_1(\Theta_u))$.
\end{enumerate}

\noindent
If we further assume that $f\colon\Gamma\to \Gamma$ is an expanding irreducible train-track map then there exists a continuous, convex, homogeneous of degree $-1$ function $\mathfrak H \colon \A \to (0,\infty)$ such that for every primitive integral $u\in \A$ we have:
\begin{enumerate}
\item[(4)] $f_u \colon \Theta_u \to \Theta_u$ is an expanding irreducible train-track map,
\item[(5)] $\mathfrak H(u)=\log(\lambda(\fee_u)) = h(f_u)$ where $\lambda(\fee_u)>1$ is the ``growth rate'' or ``stretch factor'' of $\fee_u$ and $h(f_u)$ is the topological entropy of $f_u$.
\item[(6)] If in addition we assume that $\fee\in \Out(F_N)$ is hyperbolic and fully irreducible then for every primitive integral $u\in \A$ the automorphism $\fee_u$ of $\ker(u)$ is hyperbolic and fully irreducible.
\end{enumerate}
\end{theor}

In the remainder of the introduction, we will discuss in more detail some generalizations and refinements of the various parts of this theorem (see \S\S\ref{S:intro fibration}--\ref{S:intro entropy} and Theorems \ref{T:folded basics}, \ref{T:new iwip main}, and \ref{T:new entropy main} below, which imply Theorem \ref{T:main theorem}), elaborating on the similarities and differences with the $3$--manifold setting.  Next we will describe in \S \ref{S:intro small dil} an application and the connection to work of Algom-Kfir and Rafi \cite{AR}.
In \S \ref{S:intro other construction} we compare our construction with related constructions of Gautero \cite{Gau1} and Wang \cite{Wang}.
Finally, in \S \ref{S:intro alexander} we briefly explain the relation between the class $\epsilon$ and the work of McMullen \cite{McAlex},  Dunfield \cite{Dunfield}, and Button \cite{Button} on the Alexander norm, further strengthening the analogy with the $3$--manifold setting.

\begin{remark}  In our follow-up paper \cite{DKL2}, our definition of a section of $\flow$ is more restrictive than what is given above.  It turns out that all sections we construct in this paper do satisfy the conditions of that definition, however.
\end{remark}

\subsection{Fibrations, semiflows and an Euler-like class.} \label{S:intro fibration}

In what follows, we assume that $M$ is a closed, oriented, aspherical
$3$--manifold and write $G_M = \pi_1(M)$.  We also assume that $M$
fibers over the circle $\fib_0 \colon M \to \mathbb{S}^1$ with fiber $S$ and
monodromy $F \colon S \to S$, so that $M = S \times [0,1]/(x,1) \sim
(F(x),0)$.  Write $u_0 \in H^1(M;\R)$ to denote the induced primitive
integral element $u_0 = (\fib_0)_* \colon G_M \to \Z$.

We sketch here a geometric proof of the existence of the open set $\C_{G_M}$, essentially due to Tischler \cite{Tisch}, that motivates our construction. The starting point of the analysis is the following two observations. Firstly, there is a natural suspension flow on $M$ determined by the local flow on $S\times[0,1]$, given by $(x,t)\mapsto(x,t+s)$ for $x\in S$ and $s\leq 1-t$. Secondly, via the de Rham isomorphism $H^1(M;\R) \cong H^1_{dR}(M;\R)$, one can represent $u_0$ by the closed, nowhere vanishing $1$--form $\nu_0 = \fib_0^*(vol_{\mathbb{S}^1})$, where $vol_{\mathbb{S}^1}$ is the unit volume form on $\mathbb{S}^1$. Any class $u \in H^1(M;\R)$ {\em projectively} close to $u_0$ (that is, close up to scaling by positive real number), is also represented by a closed nowhere vanishing $1$--form $\nu$ projectively close to $\nu_0$.  

This $1$--form $\nu$ provides a crucial tool for studying properties of $u$: If $u$ is a primitive integral class represented by such a $\nu$, then integrating $\nu$ defines a fibration $\fib_\nu \colon M \to \R/\Z = \mathbb{S}^1$ so that $\nu = \fib_\nu^*(vol_{\mathbb{S}^1})$ and $u = (\fib_\nu)_*\colon G_M\to \Z$. Thus every primitive integral $u\in H^1(M;\R)$ projectively close to $u_0$ is also induced by a fibration over $\mathbb{S}^1$, and hence has finitely generated kernel equal to the fundamental group $\pi_1(S_\nu)$ of the fiber $S_\nu$. Furthermore, the fiber $S_\nu$ is a section of the suspension flow on $M$, and the first return map to $S_\nu$ is precisely the monodromy $F_\nu$ of the fibration. The Euler characteristic $\chi(S_\nu)$ of any fiber $S_\nu$ may then be calculated by pairing the homology class of $S_\nu$ with the Euler class of the $2$--plane bundle $\ker(\nu)$---equivalently, one may pair $u$, the Poincar\'e dual of $S_\nu$, with the Poincar\'e dual $e_\nu$ of this Euler class. 
On the other hand, for any $\nu$ projectively close to $\nu_0$, $\ker(\nu)$ and
$\ker(\nu_0)$ are homotopic $2$--plane sub-bundles of $TM$ and so have equivalent Euler classes. Therefore, there is a single class $e = e_{\nu_0}$ such that $\chi(S_\nu)=u(e)$ for all $u$ projectively close to $u_0$; see \cite{ThuN}.\\

Now consider a free-by-cyclic group $G = G_\fee = F_N\rtimes_\fee \Z$ defined by $\fee \in \Out(F_N)$ with $2 \leq N < \infty$ and associated primitive integral $u_0\colon G\to \Z$ with $\ker(u_0)=F_N$.  We would like to obtain a similar geometric understanding of those $u\in H^1(G;\R)$ which are projectively near to $u_0$. Taking a topological representative $f\colon \Gamma\to\Gamma$ of $\fee$ (that is, a graph $\Gamma$ together with a topological graph map $f$ and an isomorphism $\ma\colon F_n\to\pi_1(\Gamma)$ for which $\ma^{-1}f_* \ma$ represents $\fee\in \Out(F_N)$; see \S\ref{subsect:mark}), one may build the corresponding mapping torus $Z_f = \Gamma\times[0,1] / (x,1)\sim(f(x),0)$. Analogous to the $3$--manifold above, $Z_f$ has fundamental group isomorphic to $G$ and admits a map to the circle $Z_f\to \mathbb{S}^1$ inducing $u_0\colon G\to \Z$. However, it is unclear how to ``perturb'' this map to the circle to obtain useful information about primitive integral elements $u \in H^1(G;\R)$ which are near $u_0$, up to scaling. 

To circumvent this problem, for any $\fee\in \Out(F_N)$, given a tame graph map $f\colon\Gamma\to\Gamma$ which is a topological representative of $\fee$ 
(which always exists and is easy to find) and a folding sequence for $f$, we introduce a new $2$--complex called the {\em
  folded mapping torus} $X_f$ for $\fee$. The complex $X_f$ is obtained as a quotient of the full
mapping torus $Z_f$ of $f$ by following the ``folding line'' from
$\Gamma$ to $\Gamma\fee$ in the Culler-Vogtmann outer space determined by the folding sequence.  See
\S\ref{sect:background} for background, \S\ref{sect:folding} for a description of the folding line, and \S\ref{sect:folded mapping torus and semiflows} for the construction of $X_f$.  

The local flow on $\Gamma \times [0,1]$ given by $(x,t) \mapsto (x,t+s)$ descends to a semiflow $\flow$ on $X_f$ (see \S\ref{S:semiflow construction}).  Moreover, $X_f$ is a $K(G,1)$--space (see Corollary \ref{C:K(G,1)}), and the homomorphism $u_0 \colon G \to \Z$ is induced by a map $\fib_0 \colon X_f \to \mathbb{S}^1$.

Although $X_f$ is not a manifold, and so there is no notion of smooth $1$--forms on $X_f$ in the classical sense, the combinatorial structure of $X_f$ nevertheless provides a way to perturb the maps $\fib_0$ and $u_0$. In \S\ref{sect:trapezoids}, we construct a {\em trapezoid cell structure} on $X_f$ (with a choice of orientation for the $1$--cells) so that the restriction of $\fib_0$ to every semiflow line and every $1$--cell is an orientation preserving local homeomorphism. We then define the \emph{positive cone} of $X_f$ to be the subset $\A=\A_{X_f}\subset H^1(X_f;\R)$ of cohomology classes represented by cellular $1$--cocycles that are positive on every positively oriented $1$--cell.  Moreover, this cell structure allows for the definition of a cellular $1$--chain $\epsilon$ on $X_f$ that serves as a ``combinatorial'' analogue of the Poincar\'e dual of the Euler class of the $2$--plane sub-bundle $\ker(\nu_0)$ of $TM$ considered above; see Remark \ref{R:Euler comparison} for more on this analogy. We summarize these basic properties of $X_f$ in the following.

\begin{proposition}
Let $\fee \in \Out(F_N)$ (where $N\ge 2$) be arbitrary,  let $f \colon \Gamma \to \Gamma$ be a tame  graph-map which is a topological representative of $\fee$, and let $X_f$ be a folded mapping torus for $f$,
with the trapezoid cell structure, and let $\flow$ be the associated semiflow. Let $u_0\in
H^1(G_\fee;\R)$ be the associated homomorphism for the splitting
$G_\fee=F_N\rtimes_\fee \Z$. Then:
\begin{itemize}
\item[a)] The 2--complex $X_f$ is a $K(G_\fee,1)$--space.
\item[b)] The set $\A_{X_f} \subset H^1(X_f;\R)$ is an open convex cone containing $u_0$.
\item[c)] The singular 1--chain $\epsilon$ is a $1$--cycle and thus defines an element also denoted $\epsilon\in H_1(X_f;\R)$, which we call the {\em $E_0$--class of $X_f$}.
\end{itemize}
\end{proposition}

\begin{remark}
The construction of $X$, $\flow$, and $\A$ all depend on the choice of a representative $\f$ and a corresponding folding line for $\f$. For any triple $(X', \flow',\A')$ obtained by making other choices, $\A'\cap \A$ is an open convex cone containing $u_0$ and $X'$ is homotopy equivalent to $X$.  The $E_0$--class $\epsilon$ is independent of these choices.
\end{remark}

Working in the framework of $X_f$, we obtain the following results that precisely mirror the phenomenon described above for the $3$--manifold $M$.

\begin{theor} \label{T:folded basics}
Let $G_\fee = F_N\rtimes_\fee \Z$ be the mapping torus group of $\fee \in \Out(F_N)$ (where $N\ge 2$), let $f \colon \Gamma \to \Gamma$ be a tame graph-map which is a topological representative of $\fee$, and let $X_f$ be a folded mapping torus for $f$ with the associated semiflow $\flow$, positive cone $\A_{X_f} \subset H^1(X_f;\R)$, and $E_0$--class $\epsilon \in H_1(X_f;\R)$.  Given a positive cellular $1$--cocycle $z \in Z^1(X_f;\R)$ representing a primitive integral element $u \in \A_{X_f} \subset H^1(X_f;\R) = H^1(G_\fee;\R)$, there exists a map
$\fib_z \colon X_f \to \mathbb{S}^1$ and a fiber $\Theta_z = \fib_z^{-1}(x_0)$ for some $x_0 \in \mathbb{S}^1$ so that
\begin{enumerate}
\item $\Theta_z$ is a finite connected topological graph such that the
  inclusion $\Theta_z\subset X_f$ is $\pi_1$--injective and such that
  $\pi_1(\Theta_z)=\ker(u)\le \pi_1(X_f)=G_\fee$.
\item $u(\epsilon) = \chi(\Theta_z) = 1 - \rk(\ker(u))$,
\item $\Theta_z$ is a section of $\flow$,
\item the first return map $f_z \colon \Theta_z \to \Theta_z$ of $\flow$ is a homotopy equivalence with $f(V\Theta_z)\subset V\Theta_z$ and $(f_z)_*$ represents $\fee_u \in \Out(\ker(u))$.
\end{enumerate}
\end{theor}

See \S\ref{S:intro other construction} and \S\ref{S:intro alexander} for discussion of related constructions and results.

\begin{remark}
The map $\fib_z$ depends on the choice of representative cocycle $z$ and the graph $\Theta_z$ further depends on the point $x_0$.  A different representative $z'$ gives rise to a homotopic map $\fib_{z'} \simeq \fib_z$.  It follows from the proof of Theorem~\ref{T:folded basics} in \S\ref{S:proof_of_theorem_B}  that for any $x_1 \in \mathbb S^1$, $\fib_{z'}^{-1}(x_1)$ can be ``flowed onto'' $\Theta_z$ producing a homotopy equivalence $\fib_{z'}^{-1}(x_1) \to \Theta_z$ homotopic to the inclusion $\fib_{z'}^{-1}(x_1) \hookrightarrow X$. 
\end{remark}

\subsection{Train tracks and semiflows} \label{S:intro train}

We return to the motivational setting of surfaces and $3$--manifolds.  Thurston \cite{ThuMCG,FLP} proved that a homeomorphism $F \colon S \to S$ is isotopic to a {\em pseudo-Anosov homeomorphism} if and only if $F_\ast \in \Out(\pi_1(S))$ is {\em atoroidal}, meaning that it has no nontrivial periodic conjugacy classes.  In fact, Thurston proved that this happens if and only if $F_\ast$ is {\em geometrically atoroidal}, meaning that there are no nontrivial periodic conjugacy classes of elements represented by {\em simple} closed curves.

Thurston's hyperbolization theorem \cite{Otal} moreover implies that the monodromy $F \colon S \to S$ of the fibration $\fib_0 \colon M \to \mathbb S^1$ is isotopic to a pseudo-Anosov homeomorphism if and only if $M$ is a hyperbolic manifold.  Consequently, $F$ is isotopic to a pseudo-Anosov homeomorphism if and only if the monodromy $F_\nu \colon S_\nu \to S_\nu$ of  {\em any} fibration $\eta_\nu$ is isotopic to a pseudo-Anosov homeomorphism.  If $F \colon S \to S$ is a pseudo-Anosov homeomorphism (not just isotopic to one), Fried \cite[Expos\'e 14]{FLP} proved a stronger statement: when the $1$--form $\nu$ is sufficiently close to $\nu_0$ (the $1$--form representing $u_0$) then the monodromy $F_\nu \colon S_\nu \to S_\nu$ of the fibration $\fib_\nu$, which is the first return map of the suspension flow of $F$, is already pseudo-Anosov (again, the isotopy is unnecessary).  Fried's proof does not appeal to the hyperbolic structure on $M$.\\

The equivalent algebraic conditions on $F_\ast$ described above have analogues for $\fee \in \Out(F_N)$.  However, these conditions on $\fee$ are not equivalent and have different implications. Firstly, we first say that $\fee$ is {\em atoroidal} if $\fee$ has no nontrivial periodic conjugacy classes in $F_N$.  Appealing to the work of Bestvina--Feighn~\cite{BF92}, Brinkmann~\cite{Brink} proved that $\fee\in \Out(F_N)$ is atoroidal if and only if it is {\em hyperbolic}, meaning that the $G = F_N \rtimes_\fee \Z$ is word-hyperbolic.  It follows that the monodromy $\fee_u$ of any other primitive integral $u \in \C_G$ is necessarily hyperbolic as well.

Simple closed curves on a surface $S$ provide a means of cutting $S$ into pieces, thus the $F_N$--analogue of a periodic conjugacy class represented by a {\em simple closed curve} on $S$ is a periodic free factor.  By analogy with the geometric atoroidal condition, one therefore says that an element $\fee \in \Out(F_N)$ is \emph{fully irreducible} or \emph{iwip} (for ``irreducible with irreducible powers'') if no positive power of $\fee$ preserves the conjugacy class of a proper free factor of $F_N$.  

The fully irreducible condition and the hyperbolic (or atoroidal) condition are independent, with neither one implying the other. Thus one is often interested in outer automorphisms of $F_N$ that are both hyperbolic and fully irreducible; such outer automorphisms provide the closest counterpart to (isotopy classes containing) pseudo-Anosov surface homeomorphisms and are also known to be generic in $\Out(F_N)$ (e.g., see Rivin \cite{Rivin08,Rivin10} and Calegari--Maher \cite{CM10}).

From a geometric/dynamical perspective, the analogue of a pseudo-Anosov homeomorphism is an {\em expanding irreducible train-track map} $f\colon \Gamma\to\Gamma$ representing $\fee \in \Out(F_N)$ (see Definitions~\ref{D:reg-exp-graph-map} and \ref{D:train-track}). Such representatives are automatically tame (see Lemma~\ref{L:tt-tame}). Train-track maps were introduced by Bestvina and Handel in \cite{BH92}, where it was also shown that an expanding irreducible train-track representative exists for every fully irreducible $\fee\in \Out(F_N)$, and more generally for every irreducible $\fee$  of infinite order.
 Our next result proves that if the folded mapping torus $X_f$ is constructed using such a $f$, then for every primitive integral $u \in \A$ represented by a positive cocycle $z$, the first return map $f_z \colon \Theta_z \to \Theta_z$ from Theorem~\ref{T:folded basics} is also an expanding irreducible train-track representative of $\fee_u\in \Out(\ker(u))$---this is analogous to Fried's result.

If $\fee$ is additionally assumed to be fully irreducible and hyperbolic, then in view of the hyperbolicity of $G\cong G_\fee$, the monodromy $\fee_u$ of any other primitive integral $u \in \C_G$ is also hyperbolic.
However, it is not at all clear whether $\fee_u$ must also be fully irreducible. In general, proving full irreducibility for free group automorphisms is a fairly hard task, see for example \cite{CP10}. Nevertheless, by using the structure of the folded mapping torus $X_f$, we prove that for any primitive integral $u \in \A$, the monodromy $\fee_u\in \Out(\ker(u))$ is indeed fully irreducible: 

\begin{theor} \label{T:new iwip main}
Let $\fee \in \Out(F_N)$ (where $N\ge 2$) be an outer automorphism represented by an expanding irreducible train track  map $f
\colon \Gamma \to \Gamma$, let
$X_f$ be a folded mapping torus, and $\flow$ the associated semiflow.
Let $u \in \A_{X_f} \subset H^1(G_\fee;\R) = H^1(X_f;\R)$ be a primitive integral element which is
represented by a positive $1$--cocycle $z \in Z^1(X_f;\R)$, and let $\Theta_z
\subset X_f$ be the topological graph from Theorem \ref{T:folded
  basics} so that the first return map $f_z \colon \Theta_z \to
\Theta_z$ is a homotopy equivalence representing
$\fee_u$. Then:
\begin{enumerate}
\item The first return map $f_z \colon \Theta_z \to\Theta_z$ is an expanding irreducible train-track map (with respect to an appropriately chosen linear graph structure on $\Theta_z$).

\item If we additionally assume that $\fee$ is hyperbolic and fully irreducible, then $\fee_u\in \Out(\ker(u))$ is also hyperbolic and fully irreducible.
\end{enumerate}
\end{theor}

We note that if in item (2) above $\fee$ is only assumed to be fully irreducible (and not hyperbolic), then $\fee_u$ need not be fully irreducible; see Remark~\ref{R:geometric-iwip}.  

Our definitions require that train-track maps are  piecewise linear on every edge
with respect to some linear structure on the graph; see
\S\S\ref{S:graph maps}--\ref{subject:tt} for precise statements.
While this property is part of the assumption on $f$, it becomes part of the conclusion for $f_z$ for any positive cocycle $z$.    The graph $\Theta_z$ does not come equipped with any natural linear structure, and thus part of the proof of Theorem \ref{T:new iwip main} includes the construction of a natural linear structure.  This structure is deduced as a consequence of a general dynamical result (which appears to be  of independent interest), Theorem~\ref{T:topological graph to graph}, about topological graph maps that are ``sufficiently mixing'' (a notion we call {\em expanding on all scales}).  

To prove that $\fee_u$ is fully irreducible, under the assumption that $\fee$ is hyperbolic and fully irreducible, we utilize a criterion of full irreducibility for hyperbolic automorphisms obtained by Kapovich in~\cite{K12}; see also Pfaff~\cite{Pf12,Pf13}.  That criterion says that if $f'\colon\Gamma'\to\Gamma'$ is a train-track map representative of a hyperbolic automorphism $\fee'$ then $\fee'$ is fully irreducible if and only if some power $A(f')^m$ of the transition matrix $A(f')$ is strictly positive and the Whitehead graphs of all vertices of $\Gamma'$ are connected. We first obtain a general result strengthening this criterion (Theorem~\ref{T:clean} below) and show that the condition $A(f')^m>0$ can be replaced by just requiring the transition matrix $A(f')$ to be irreducible. We then verify that $f_z\colon \Theta_z\to\Theta_z$ is a train-track map with an irreducible transition matrix and with connected Whitehead graphs by exploiting the properties of $f$ and of the semiflow $\flow$. This allows us to conclude that $\fee_u$ is fully irreducible. As another useful application of Theorem~\ref{T:clean}, we show in Corollary~\ref{C:hyp_iwip} that for hyperbolic elements in  $\Out(F_N)$ being irreducible is equivalent to being fully irreducible.

\subsection{Entropy and stretch factors} \label{S:intro entropy}

For a pseudo-Anosov homeomorphism $F \colon S \to S$, the {\em stretch factor} or {\em dilatation} $\lambda(F)$ records the asymptotic growth rate for lengths of curves under iteration.  That is,
\[ \log(\lambda(F)) = \lim_{n \to \infty} \frac{\log(\ell(F^n(\gamma)))}{n} \]
where $\ell(\cdot)$ is the length of the geodesic representative in any fixed metric on $S$ and $\gamma$ is any homotopically nontrivial curve on $S$.  From the work of Fathi and Shub \cite[Expos\'e 10]{FLP}, $\log(\lambda(F))$ is the topological entropy $h(F)$ of $F$.   
Returning to the $3$--manifold $M$, we continue to assume that some (hence every) monodromy $F \colon S \to S$ for a primitive integral $u_0 \in H^1(M;\R)$ is pseudo-Anosov, and we let $\C_0 \subset \C_G$ denote the component of the f.g.~cone containing $u_0$.  Fried \cite{FriedD} proved that the function which assigns to every $u \in \C_0$ the topological entropy $h(F_u)$ of the monodromy $F_u \colon S_u \to S_u$, extends to a continuous, homogeneous of degree $-1$ function
\[ \mathfrak h \colon \C_0 \to \R \]
for which the reciprocal $1/{\mathfrak h}$ is concave.  See
\cite{LO,Oertel,Mc} for other proofs and refinements of this result.\\

Any $\fee\in \Out(F_N)$ also has a \emph{stretch factor} $\lambda(\fee)\ge 1$ defined as
\[
\log\lambda(\fee)=\sup_{w\in F_N} \lim_{n\to\infty} \frac{\log ||\fee^n(w)||_A}{n},
\]
where $||w||_A$ denotes the cyclically reduced length of $w$ with respect to some fixed free basis $A$ of $F_N$. It is known that the above definition does not depend on the choice of $A$ and that the $\sup$ in the above formula is actually a $\max$.  

An element $\fee\in \Out(F_N)$ is called \emph{exponentially growing} if $\lambda(\fee)>1$. It is known that every fully irreducible $\fee$ is exponentially growing. More generally, if $\fee\in \Out(F_N)$ admits an expanding irreducible train-track representative $f\colon\Gamma\to\Gamma$, then $\fee$ is exponentially growing and $\lambda(\fee)=\lambda(f)$, where $\lambda(f)$ is the Perron-Frobenius eigenvalue of the transition matrix $A(f)$ of $f$ (see \S\ref{S:graph maps}).  We note that in this situation $\lambda(f)$ depends only on $\fee$ and not on the choice of $f$, and that it is alternatively calculated via topological entropy $h(f)$ as $h(f) = \log(\lambda(f))$ (see Proposition~\ref{prop:h(f)} below for a more general version of this statement). This discussion applies, for example, to all fully irreducible automorphisms and, more generally, to all irreducible elements of $\Out(F_N)$ of infinite order.

The final part of Theorem \ref{T:main theorem} provides the analogue of Fried's result on entropy.

\begin{theor} \label{T:new entropy main}
Let $\fee \in \Out(F_N)$ (where $N\ge 2)$ be an element represented by an expanding irreducible train track map $f \colon \Gamma \to \Gamma$. Let $X_f$ be a folded mapping torus, $\flow$ the associated semiflow, and $\A_{X_f} \subset H^1(X_\f;\R)$ the positive cone of $X_\f$.  Then there exists a continuous, homogeneous of degree $-1$ function $\mathfrak H \colon \A_{X_f} \to \R$
such that for every primitive integral $u \in \A_{X_f}$
\[ \log(\lambda(\fee_u)) = \mathfrak H(u).\]
Moreover, $1/\mathfrak H$ is concave, and hence $\mathfrak H$ is convex.
\end{theor}

For any positive cocycle $z\in Z^1(X_f;\R)$ representing a primitive integral $u\in \A$ as above, we have $h(f_z) = \log(\lambda(f_z)) = \log(\lambda(\fee_u))$ where $f_z\colon \Theta_z\to\Theta_z$ is the expanding irreducible train-track map provided by Theorem~\ref{T:new iwip main}.  The proof of Theorem~\ref{T:new entropy main} involves showing that $[z] \mapsto h(f_z)$ extends to a function $\mathfrak H$ which is continuous, homogeneous of degree $-1$, and for which $1/\mathfrak H$ is concave; this argument closely follows Fried \cite{FriedD},  with some important modifications.

The rough idea of the proof is as follows.
According to the variational principle,
$h(f_z)=\sup_\nu h_\nu(f_z)$ where $\nu$ runs over all
$f_z$--invariant Borel probability measures on $\Theta_z$ and where
$h_\nu(f_z)$ is the Kolmogorov-Sinai entropy with respect to $\nu$.
There is a natural bijection between $\flow$--invariant finite
measures on $X_f$ and $f_z$--invariant finite measures on $\Theta_z$,
which we write as $\mu \mapsto \bar \mu_0$ (with $\bar \mu$
representing the normalized probability measure).  Then, according to
a general result of Abramov, for any $\flow$--invariant probability
measure $\mu$, $h_{\bar \mu}(f_z)=h_\mu(\flow)/ \bar \mu_0(\Theta_z)$
where $h_\mu(\flow)$ is the entropy of the flow with respect to $\mu$
(which depends on $\mu$, but not on the class $u \in \A$ represented
by the cocycle $z = z_u$). Then one gets $\frac{1}{h(f_z)}=\inf_\mu
\bar{\mu}_0(\Theta_z)/h_\mu(\flow)$. A key step in Fried's
argument is to then show that for any $\mu$ the function $u\mapsto
\bar \mu_0(\Theta_{z_u})$ extends to a \emph{linear} continuous
function $\mu_\ast \colon \A \to \R$. The function $W\colon  u\mapsto
\inf_\mu \mu_\ast(u)/h_\mu(\flow)$ is now given as the infimum of a
family of linear continuous functions. Therefore $W$ is concave on
$\A$ and hence $W$ is continuous. For every primitive integral $u\in
\A$ we have $W(u)=\frac{1}{h(f_{z_u})}$, and the conclusion of
Theorem~\ref{T:new entropy main} follows.  There are several
difficulties with carrying out this line of argument that we need to
overcome. First, Fried's and Abramov's results are for flows, while
$\flow$ is only a semiflow. Thus we need to immediately pass, via the
``natural extension'' construction in dynamics, from a semiflow
$\flow$ on $X_f$ to a related flow $\flow'$ on another space
$X_f'$. We then need to make sure that the above argument gets
correctly translated to this new setting.  Another difficulty is that
the geometric ``surgery'' argument Fried uses to prove linearity of
$\mu_\ast$ does not work well for $X_\f'$. Thus we deploy a homological argument for that purpose instead of attempting to mimic the argument of Fried.

\subsection{Connections to small dilatation free group automorphisms.} \label{S:intro small dil}

As a corollary of Theorem \ref{T:new entropy main}, we see that the
function $u \mapsto \log(\lambda(\fee_u)) \chi(\ker(u))$ defined on
primitive integral classes $u$ extends to a continuous function which
is constant on rays.  Thus from any $G = G_\fee$ with $b_1(G) \ge 2$
and $\fee$ represented by an expanding irreducible train track map,
one obtains sequences $\{\fee_k \colon F_{N_k} \to
F_{N_k}\}_{k=1}^\infty$ of automorphisms $\fee_k$ represented by
expanding irreducible train track maps for which $N_k \to \infty$ and
$N_k \log(\lambda(\fee_k))$ is uniformly bounded above and below by
positive constants. Moreover, by Theorem \ref{T:new iwip main} (2), if
$\fee$ is fully irreducible and hyperbolic so are all $\fee_k$; see
Corollary \ref{C:McMullen analogue}. This corollary mirrors a construction of pseudo-Anosov homeomorphisms with the analogous property using $3$--manifolds as described by McMullen \cite[Section 10]{Mc}.  Corollary \ref{C:McMullen analogue} also complements the work of Algom-Kfir and Rafi \cite{AR} who showed that any sequence $\{\fee_k \colon F_{N_k} \to F_{N_k}\}_{k=1}^\infty$ for which $N_k \log(\lambda(\fee_k))$ is bounded must come from this construction (up to a certain surgery procedure).  This was in turn a free-by-cyclic analogue of a result of Farb--Leininger--Margalit \cite{FLM} (see also \cite{Ag}) complementing McMullen's construction.  So in some sense, Corollary \ref{C:McMullen analogue} brings this picture full circle.

\subsection{Related results and constructions} \label{S:intro other construction}

As mentioned at the start of the introduction, it already follows from general facts about the BNS-invariant and related constructions (see, for example,  \cite{BNS,Levitt87}) that under the assumptions of Theorem~\ref{T:folded basics} there exists some open cone $\C_G \subseteq H^1(G_\fee)$ containing $u_0$, such that every primitive integral $u\in \C_G$ has finitely generated kernel $\ker(u)$.  Additional general cohomological considerations 
(e.g.~see Theorem~2.6 and Remark~2.7 in \cite{GMSW}) further imply that $\ker(u)$ is a free group.  Consequently, $u$ determines another splitting of $G_\fee$ as a (f.g.~free)-by-(infinite~cyclic) group.  We do not use these general facts in the proof of Theorem~\ref{T:folded basics} and instead directly establish for every primitive integral $u\in \mathcal A$ that $\ker(u)$ is a free group of finite rank.

The folded mapping torus $X_f$ and the semiflow $\flow$ are similar to Gautero's construction of a ``dynamical $2$--complex''~\cite{Gau1,Gau2,Gau3,Gau4} (see also the discussion in \cite[Section 5]{Gau1} for connections to other related notions) and the unpublished construction of Z.~Wang~\cite{Wang} of a ``decomposed mapping torus'' and various ``folded complexes''.  
Our folded mapping torus $X_f$ turns out the be homeomorphic to the complex constructed by Brinkmann--Schleimer in \cite{BStrian}, which treats only the case when the automorphism is induced by a homeomorphism of a surface with one boundary component.
However, \cite{BStrian} does not construct a semi-flow on the complex, but rather uses it to construct an ideal triangulation of the mapping torus of the homeomorphism.

Gautero builds a $2$--complex from a topological representative $f\colon\Gamma\to\Gamma$, where $\Gamma$ is a trivalent graph, and a path in outer space from $\Gamma$ to $\Gamma \fee$ using Whitehead moves (see also the variation in \cite{Gau4} using Stalling's folds instead of Whitehead moves, which is a modification of his original construction and is more closely related to the one here). 
Wang \cite{Wang} on the other hand constructs from a sequence of Stallings folds from $\Gamma$ to $\Gamma \fee$ first an iterated mapping cylinder which he calls a {\em decomposed mapping torus}.  This space is homotopy equivalent to the mapping torus and also admits a semiflow. Then for each cohomology class $u \in H^1(G_\fee;\R)$ satisfying certain conditions he builds what he calls a {\em folded complex} $Y_{f,u}$ which is a quotient of his decomposed mapping torus. 

In both of these constructions, the main goal is that of finding cross sections in analogy with Fried's work \cite{FriedS}.  In particular, these constructions are likely sufficient for proving analogues of Theorem~\ref{T:folded basics}.  However, it is unclear how to attack the remainder of Theorem \ref{T:main theorem} and Theorems \ref{T:new iwip main} and \ref{T:new entropy main}, with these constructions.  It is worth noting that in Wang's construction, he uses sections of the folded complexes to find sections of the original mapping torus (though the latter will likely not satisfy Theorem~\ref{T:folded basics}).
In any event, we try to mention the relevant connections to Gautero's and Wang's construction wherever possible throughout the paper.

\subsection{Alexander norm} \label{S:intro alexander}

In the $3$--manifold setting, the {\em Thurston norm} is a seminorm
\[ \norm_T \colon H^1(M;\R) = H^1(G_M;\R) \to \R\]
defined by Thurston \cite{ThuN} taking integral values on primitive
integral points (see also
\cite{Gabai,MosherD1,MosherD2,OertelHom,LO,Sch}).  He then observed that the unit ball $\mathbf B_T$ is a polyhedron.  The norm of a primitive integral class is a measure of the complexity of the class, and for any class $u$ represented by a fibration $M \to \mathbb{S}^1$, the norm is precisely $\norm_T(u) = -u(e)$, where $e$ is again the Poincar\'e dual of the associated Euler class.  In fact, in this setting $\C_{G_M}$ is a union of cones on the interiors of certain top dimensional faces of the polyhedron $\mathbf B_T$. In \cite{McAlex}, McMullen constructed a related {\em Alexander norm} $\norm_A$ on $H^1(G;\R)$ for any finitely presented group $G$, which again has a polyhedral unit ball $\mathbf B_A$ (and so as with the Thurston norm, $\norm_A$ is linear on cones on faces). In the case of the $3$--manifold group $G_M$ the Thurston and Alexander norms actually agree on $\C_{G_M}$ when $\dim(H^1(M;\R))>1$; however Dunfield \cite{Dunfield} proved that in general they do not agree on all of $H^1(G_M;\R)$.

In the case of our free-by-cyclic group $G_\fee$, since $G_\fee$ has a deficiency $1$ presentation, the results of McMullen \cite[Theorem 4.1]{McAlex} and Button \cite[Theorem 3.1]{Button} together imply that when $\dim(H^1(G_\fee;\R))>1$, the Alexander norm calculates $\norm_A(u) = b_1(\ker(u))-1$ for any primitive integral $u\in \C_{G_\fee}$, where $b_1(\ker(u))=\dim(H^1(\ker(u);\R))$. As mentioned above, general cohomological considerations 
(e.g.~see Theorem~2.6 and Remark~2.7 in \cite{GMSW}) additionally imply that $\ker(u)$ is free for any primitive integral $u\in H^1(G_\fee;\R)$ with $\ker(u)$ finitely generated; thus we have $b_1(\ker(u)) = \rk(\ker(u))$ in this case. Moreover, Dunfield's work \cite{Dunfield} implies that any primitive integral $u\in \C_{G_\fee}$ must lie over the interior of a top dimensional face of $\mathbf B_A$; therefore $\norm_A$ restricts to a linear function on an open cone neighborhood of any such $u$. Putting these results together, it follows that for every primitive integral $u_0\in \C_{G_\fee}$, the assignment $u\mapsto  1-\rk(\ker(u))$ for primitive integral $u$ extends to a linear function $u\mapsto -\norm_A(u)$ on some open cone containing $u_0$ (namely the intersection of $\C_{G_\fee}$ with the cone on a top dimensional face of $\mathbf B_A$). Thus Theorem~\ref{T:folded basics}(2) gives a realization of this linear function as the geometrically defined $E_0$--class $\epsilon\in H^1(G_\fee;\R)$. Moreover, it shows that pairing any element of $\A_{X_f}$ with the $E_0$--class gives precisely the (negative of the) Alexander norm and thus strengthens the analogy with McMullen's result for $3$--manifolds.

\bigskip

\begin{remark}
The results and techniques developed in this manuscript have served as the basis for a number of advances in the study of free-by-cyclic groups and $\Out(F_N)$ since it was made publicly available in January 2013.  For example, in \cite{DKL2}, we continue our analysis of the folded mapping torus $X_f$ and develop an analogy with McMullen's Teichm\"uller polynomial \cite{Mc}.  Thus \cite{DKL2} strengthens the connection with the BNS-invariant and provides a new proof of Theorem \ref{T:new entropy main}.  However, while the results in \cite{DKL2} rely on the main constructions of this paper, most of the new results there are orthogonal to those of the current paper.  Then in \cite{DKL2.5}, we draw on the results from both of these papers to provide examples of fully irreducible automorphisms of free groups whose stretch factors illustrate a new phenomenon, complementing the results of Handel-Mosher \cite{HM-expansion}. 

In addition, concurrent with \cite{DKL2}, Algom-Kfir, Hironaka, and Rafi \cite{AKHR} construct their own analogue of McMullen's Teichm\"uller polynomial.  Their results rely on the construction of the folded mapping torus $X_f$, properties of $\A_{X_f}$, and in particular Theorem \ref{T:new iwip main}.
\end{remark}

\bigskip

\noindent
{\bf Acknowledgements.}

We are grateful to Jayadev Athreya, Jack Button, Francesco Cellarosi, Nathan
Dunfield and Jerome Los for useful mathematical discussions, to Lee Mosher for bringing the unpublished thesis of Z.~Wang to our attention, to
Peter Brinkmann and Derek Holt for help with computer experiments, and to the anonymous referees for their careful reading and for several helpful suggestions.

\section{Background} \label{sect:background}

\subsection{Graphs} 

A \emph{topological graph} $\Gamma$ is a locally finite $1$-dimensional cell complex (where locally finite means that every $0$--cell is incident to only finite many $1$--cells). The $0$--cells of $\Gamma$ are called \emph{vertices} and the set of all vertices of $\Gamma$ is denoted by $V\Gamma$.
Open $1$--cells of $\Gamma$ are called \emph{topological edges} of $\Gamma$ and the set of these is denoted $E_{\text{top}}\Gamma$. Every topological edge of $\Gamma$ is homeomorphic to an open interval and thus admits exactly two orientations. An \emph{oriented edge} $e$ of $\Gamma$ is a topological edge together with a choice of an orientation on it. We denote the set of all oriented edges of $\Gamma$ by $E\Gamma$.  If $e$ is an oriented edge of $\Gamma$, then $e^{-1}$ denotes the same topological edge with the opposite orientation. Thus ${}^{-1}\colon E\Gamma\to E\Gamma$ is a fixed-point-free involution. For an oriented edge $e\in E\Gamma$ the attaching maps for the underlying $1$--cell canonically define the \emph{initial vertex} $o(e)\in V\Gamma$ and the \emph{terminal vertex} $t(e)\in V\Gamma$. By construction, for any $e\in E\Gamma$ we have $o(e^{-1})=t(e)$ and $t(e^{-1})=o(e)$.   For $e\in E\Gamma$ we denote by $\overline{e}$ the closure of $e$ in $\Gamma$.

An \emph{orientation} on a topological graph $\Gamma$ is a partition
$E\Gamma=E_+\Gamma\sqcup E_-\Gamma$ such that for every $e\in E\Gamma$
we have $e\in E_+\Gamma$ if and only if $e^{-1}\in E_-\Gamma$. 

Given a topological graph $\Gamma$ and $v\in V\Gamma$ the \emph{link} $\link_\Gamma(v)$ consists of all $e\in E\Gamma$ with $o(e)=v$.
A \emph{turn} in $\Gamma$ is an unordered pair $e,e'$ of (not necessarily distinct) elements of $\link_\Gamma(v)$.  A turn $e,e'$ is \emph{degenerate} if $e=e'$ and \emph{nondegenerate} otherwise.   We denote the set of all turns in $\Gamma$ by $\turns(\Gamma)$.

Given a finite set of points in a topological graph $\Gamma$, one can naturally construct a topological graph $\Gamma'$ as the {\em subdivision}, in which the vertex set contains the original vertex set, together with the finite set of points.

\subsection{Linear and metric structures}\label{subsect:structures}

Given a topological graph $\Gamma$, a {\em linear atlas of charts} for $\Gamma$ is a set
\[ \Lambda=\{ (e_\alpha,j_\alpha,[a_\alpha,b_\alpha]) \}_{\alpha \in J} \]
where $J$ is some index set,  $e_\alpha \in E\Gamma$ for every $\alpha$, and $j_\alpha\colon[a_\alpha,b_\alpha] \to \Gamma$ is a characteristic map: a continuous function for which the restriction to $(a_\alpha,b_\alpha)$ is an orientation preserving homeomorphism onto the interior of $e_\alpha$ (in particular, $j_\alpha(a_\alpha) = o(e_\alpha)$ and $j_\alpha(b_\alpha) = t(e_\alpha)$).  We further assume that (1) for every $e \in E\Gamma$, there exists at least one $\alpha \in J$ so that $e = e_\alpha$; (2) for any $\alpha, \beta\in J$ such that $e_\alpha=e_{\beta}$ then then the restriction $j_\alpha^{-1} j_\beta \colon(a_\beta,b_\beta) \to (a_\alpha,b_\alpha)$ is the unique orientation preserving affine map;  and (3) for every $\alpha$ there exists $\alpha'$ so that $e_{\alpha'} = e_\alpha^{-1}$ and $j_{\alpha'} = j_\alpha \circ \sigma$, where $\sigma\colon[a_\alpha,b_\alpha] \to [a_\alpha,b_\alpha]$ is the unique affine, orientation reversing involution.  A  \emph{linear structure} on a topological graph $\Gamma$ is a linear atlas $\Lambda = \{(e_\alpha,j_\alpha,[a_\alpha,b_\alpha]) \}_{\alpha \in J}$ which is maximal with respect to inclusion.

A \emph{metric atlas} on a topological graph $\Gamma$ is an linear atlas $\mathcal L=\{(e_\alpha,j_\alpha,[a_\alpha,b_\alpha]) \}_{\alpha \in J}$ on $\Gamma$ such that  whenever $\alpha,\beta\in J$ satisfy $e_\alpha=e_\beta$ or  $e_\alpha=e_\beta^{-1}$,  then $|b_\alpha-a_\alpha|=|b_\beta-a_\beta|$.  Note that this implies that whenever $e_\alpha=e_\beta$ then  the restriction $j_\alpha^{-1} j_\beta \colon(a_\beta,b_\beta) \to (a_\alpha,b_\alpha)$ is the unique orientation preserving isometry and whenever  $e_\alpha=e_\beta^{-1}$ then  the restriction $j_\alpha^{-1} j_\beta \colon(a_\beta,b_\beta) \to (a_\alpha,b_\alpha)$ is the unique orientation reversing isometry. A \emph{metric structure} on a topological graph $\Gamma$ is a metric atlas $\mathcal L$ on $\Gamma$ which is maximal with respect to inclusion. 

A linear or metric structure on $\Gamma$ uniquely determines one on any subdivision of $\Gamma$.  Furthermore, linear and metric structures naturally lift under covering maps. 

Given a metric structure $\mathcal L = \{ (e_\alpha,j_\alpha,[a_\alpha,b_\alpha]) \}_{\alpha \in J}$, on a graph $\Gamma$, we can define $\mathcal L(e_\alpha) = b_\alpha - a_\alpha > 0$ to be the $\mathcal L$--length of $e_\alpha$.  From the definition of metric structure, this procedure uniquely defines a function we also denote $\mathcal L\colon E\Gamma \to \R_+$ satisfying $\mathcal L(e) = \mathcal L(e^{-1})$, for every $e \in E\Gamma$.  Any such function on $E\Gamma$ is called a {\em length function for $\Gamma$}. The \emph{volume} of $(\Gamma,\mathcal{L})$ is then defined to be the (possibly infinite) quantity $\vol_{\mathcal L}(\Gamma):=\tfrac{1}{2}\sum_{e\in E\Gamma}\mathcal{L}(e)$. Since $\Gamma$ is locally finite, there is a unique geodesic metric $d_{\mathcal L}$ on each component of $\Gamma$ defining the original weak topology on the $1$--complex $\Gamma$ for which the characteristic maps of that  metric structure are all local isometries.  Conversely, given a geodesic metric $d$ on the components of $\Gamma$, there is a canonical metric structure $\mathcal L$ so that $d = d_{\mathcal L}$. 

Any metric structure $\mathcal L$ on $\Gamma$ extends to a canonical linear structure $\Lambda$, meaning that $\mathcal L \subset  \Lambda$.  Also, given a linear structure $\Lambda$ on $\Gamma$ and any length function $L\colon E\Gamma\to \R_+$, there exists a unique metric structure $\mathcal L \subset \Lambda$ for which $\mathcal L(e) = L(e)$ for all $e \in E\Gamma$. That is, when a linear structure $\Lambda$ on $\Gamma$ is fixed, we can uniquely specify a metric structure $\mathcal L \subset \Lambda$ on $\Gamma$ by simply specifying the lengths of edges.  

A \emph{graph} or \emph{linear graph} is a topological graph $\Gamma$ together with a linear structure $\Lambda$ on it.  A {\em metric graph} is similarly defined as a topological graph $\Gamma$ together with a metric structure $\mathcal L$.   Via the canonical linear structure determined by a metric structure, any metric graph is also naturally a linear graph.

\subsection{Paths}

If $\Gamma$ is a topological graph, for $n\ge 1$ a \emph{combinatorial edge-path} of \emph{simplicial length $n$} in
$\Gamma$ is a sequence $\gamma=e_1e_2\dots e_n$ of oriented edges of
$\Gamma$ such that for $i=1,\dots, n-1$ we have
$t(e_i)=o(e_{i+1})$. We denote $|\gamma|=n$ and put
$o(\gamma):=o(e_1)$ and $t(\gamma):=t(e_n)$. Also, for a vertex $v$ of
$\Gamma$ we view $\gamma=v$ as a combinatorial edge-path  of
simplicial length $|\gamma|=0$ with $o(\gamma)=t(\gamma)=v$. For a combinatorial edge-path $\gamma=e_1\dots e_n$ we denote $\gamma^{-1}=e_n^{-1}\dots e_1^{-1}$, so that $\gamma^{-1}$ is a combinatorial edge-path from $t(\gamma)$ to $o(\gamma)$ with $|\gamma|=|\gamma^{-1}|=n$. For $\gamma=v$ (where $v\in V\Gamma$) we put $\gamma^{-1}=\gamma=v$.

If $\gamma=e_1\dots e_n$ is a combinatorial edge-path and $\mathbf t$ is a turn, we say that this turn is \emph{contained} in $\gamma$ if there exists $2\le i\le n$ such that $\mathbf t= \{e_{i-1}^{-1}, e_i\}$ or $\mathbf t=\{ e_i^{-1}, e_{i-1}\}$.

A \emph{topological edge-path} of \emph{simplicial length} $n$ is a continuous map $g\colon[a,b]\to \Gamma$ such that there exists a combinatorial edge-path $\gamma_g=e_1\dots e_n$ in $\Gamma$ and a subdivision $a=c_0<c_1<\dots < c_n=b$ such that for each $i=1,\dotsc,n$ we have $g(c_{i-1})=o(e_i)$, $g(c_i)=t(e_i)$, and the restriction $g|_{(c_{i-1},c_i)}$ is an orientation preserving homeomorphism onto $e_i$.  We say that $\gamma_g$ is the combinatorial edge-path \emph{associated to} $g$.  Writing $o(g) :=o(e_1)$ and $t(g):=t(e_n)$, then $g$ is a path in $\Gamma$ from $o(g)$ to $t(g)$.  For any vertex $v$ of $\Gamma$ and $[a,b] \subset \R$, we also regard the constant map $g\colon[a,b] \to \Gamma$ given by $g(s) = v$ for all $s \in [a,b]$ as a topological edge-path of \emph{simplicial length} $0$ with $o(g) = t(g) = v$.  If $\Gamma$ is a linear graph, then a topological edge-path $g\colon[a,b] \to \Gamma$ is a \emph{PL edge-path} if it additionally satisfies the following condition. For every edge $e_i$ of $\gamma_g$ and any chart $(e_\alpha,j_\alpha,[a_\alpha,b_\alpha])$ from the linear atlas with $e_\alpha=e_i$, the restriction of $g$ to $[c_{i-1},c_i]$ is equal to $j_\alpha \circ \tau_{\alpha,i}$ where $\tau_{\alpha,i}$ is the unique orientation-preserving affine homeomorphism from $[c_{i-1},c_i]$ to $[a_\alpha,b_\alpha]$.

By dropping the requirement that $g(a)$ and $g(b)$ be vertices of $\Gamma$, the above definition is straightforwardly modified to get the notion of a \emph{PL path} in $\Gamma$.

If we have a combinatorial path $\gamma$ with $|\gamma|=n$ in a topological (respectively, linear) graph $\Gamma$, there always exists a topological (respectively, PL) edge-path $g\colon[a,b]\to\Gamma$ with $\gamma_g=\gamma$. Therefore we will often suppress the distinction between combinatorial, topological, and PL edge-paths and just call them edge-paths, unless we wish to emphasize the distinction.  We call edge-paths of positive simplicial length \emph{nondegenerate} and edge-paths of simplicial length $0$ \emph{degenerate}.

A combinatorial edge-path $\gamma$ in $\Gamma$ is \emph{tight} or \emph{reduced} if $\gamma$ does not contain a subpath of the form $ee^{-1}$ where $e\in E\Gamma$ (that is, if $\gamma$ contains no degenerate turns).  A topological or PL edge-path is \emph{tight} or \emph{reduced} if the associated combinatorial edge-path is so. More generally, a PL path $g\colon[a,b]\to \Gamma$ is \emph{tight} if $g$ is locally injective. 

If $(\Gamma,\mathcal L)$ is a metric graph, then the $d_{\mathcal{L}}$--length of a PL-path $g$ in $(\Gamma,\mathcal{L})$ is denoted by $|g|_{\mathcal L}$. Similarly, the length of a combinatorial edge-path $\gamma = e_1\dots e_n$ will be denoted by $|\gamma|_{\mathcal L} = \sum_{i=1}^n \mathcal{L}(e_i)$.

\subsection{Graph maps} \label{S:graph maps}

\begin{defn}[Graph-map]\label{D:top-graph-map}
If $\Gamma$ and $\Delta$ are topological graphs, a \emph{topological graph-map} $f\colon\Gamma\to
\Delta$ is a continuous map such that $f(V\Gamma)\subseteq
V\Delta$ and such that for every $e\in E\Gamma$ $f|_e$ is a topological-edge-path
(possibly degenerate) from $f(o(e))$ to $f(t(e))$ in $\Delta$.  More precisely, this means that for any characteristic map $j\colon[a,b] \to e$, the composition $\f \circ j\colon[a,b] \to \Delta$ is a topological edge-path.  
In this situation we will often denote the combinatorial edge-path in $\Delta$ associated to $f\circ j$ by $f(e)$. 
If $\Gamma$ and $\Delta$ are linear graphs and in addition we assume that for every $e \in E\Gamma$, $\f|_e$ is a PL edge-path from $\f(o(e))$ to $\f(t(e))$, then we call $\f$ a \emph{linear graph-map} or just a \emph{graph-map}. 
\end{defn}

We note that if $\Gamma$ and $\Delta$ are equipped with linear structures, then any topological graph-map $f\colon\Gamma\to\Delta$ is homotopic rel $f^{-1}(V\Delta)\supset V\Gamma$ to a unique (linear) graph-map $f'\colon \Gamma\to \Delta$.

\begin{example}[Running example] \label{Ex:perpetuating_example}
Here we introduce an example that we will develop throughout the paper
as a means of illustrating key ideas. Let $\Gamma$ denote the graph in
Figure \ref{F:train_tracks}.  This graph has four edges, oriented as shown and labeled $\{a,b,c,d\} = E_+\Gamma$. We consider a graph-map $f\colon \Gamma\to\Gamma$ under which the edges of $\Gamma$ map to the combinatorial edge-paths $f(a)=d$, $f(b)=a$, $f(c)=b^{-1}a$, and $f(d)=ba^{-1}db^{-1}ac$. 

\begin{figure} [htb]
\labellist
\small\hair 2pt
\pinlabel $a$ [b] at 37 58
\pinlabel $b$ [b] at 37 31
\pinlabel $d$ [b] at 37 5
\pinlabel $c$ [r] at 98 28
\pinlabel $\iota$ [b] at 118 52
\pinlabel $f'$ [b] at 118 8
\pinlabel $d$ [b] at 175 58
\pinlabel $a$ [b] at 175 31
\pinlabel $b$ [l] at 145 23
\pinlabel $a$ [bl] at 153 12
\pinlabel $d$ [b] at 167 6
\pinlabel $b$ [b] at 185 5.5
\pinlabel $a$ [br] at 199 12
\pinlabel $c$ [r] at 207 23
\pinlabel $b$ [b] at 226 43.5
\pinlabel $a$ [b] at 226 17
\pinlabel $\Gamma$ at 5 54
\pinlabel $\Delta$ at 212 54
\endlabellist
\begin{center}
\includegraphics[height=2.7cm]{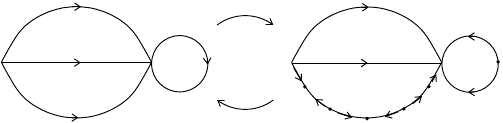} \caption{An example graph-map. Left: Original graph $\Gamma$. Right: Subdivided graph $\Delta$ with labels. Our example $\f \colon \Gamma \to \Gamma$ is the  composition of the ``identity'' $\iota\colon\Gamma\to\Delta$ with the map $f'\colon\Delta\to\Gamma$ that sends edges to edges preserving labels and orientations.}
\label{F:train_tracks}
\end{center}
\end{figure}
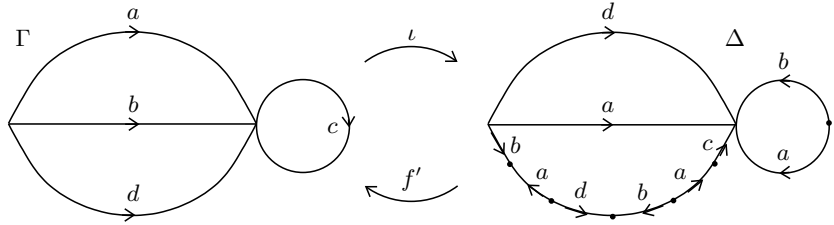

To describe this example precisely, give $\Gamma$ a linear structure and let $\Delta$ be the graph obtained by subdividing $\Gamma$ as pictured on the right of Figure~\ref{F:train_tracks}. 
The identity then gives a graph-map $\iota\colon\Gamma\to \Delta$. Labeling the edges of $\Delta$ as indicated, there is a unique linear graph-map $f'\colon \Delta\to \Gamma$ for which each edge of $\Delta$ maps by an affine map to the edge of the same label in $\Gamma$. Our example is defined to be the graph-map
\begin{equation*}
f := f'\circ \iota\colon \Gamma\to\Gamma.
\end{equation*}
\end{example}

\begin{defn}[Regular and expanding maps] \label{D:reg-exp-graph-map}
A topological graph-map $f\colon\Gamma\to\Delta$ is \emph{regular} if for any $e\in E\Gamma$ the combinatorial edge-path $f(e)$ is non-degenerate and tight (thus $|f(e)|>0$). 
We say that a topological graph map $f\colon\Gamma\to\Gamma$ is \emph{expanding} if for every edge $e\in E\Gamma$ we have $|f^n(e)|\to\infty$ as $n\to\infty$. 
\end{defn}

The next definition is convenient for carrying out the constructions in subsequent sections.

\begin{defn}[Combinatorial maps] \label{D:combinatorial-graph-map}
A {\em combinatorial graph map} is a regular, linear graph map $f \colon \Gamma \to \Gamma$ where $\Gamma$ is a metric graph with all edge lengths equal to $1$, and where the subdivision of any edge $e$ along $f^{-1}(V\Gamma)$ is into equal length parts.
\end{defn}

\begin{remark}\label{rem:make_combinatorial}
Given any regular topological graph map $f_0 \colon \Gamma \to \Gamma$, then it is easy to construct a metric structure on $\Gamma$ and a combinatorial graph map $f \colon \Gamma \to \Gamma$ that is homotopic to $f_0$ and for which the combinatorial edge-paths $f(e)$ and $f_0(e)$ agree for every edge $e$. In fact, under an appropriate expansion condition on $f_0$ (Definition~\ref{defn:expanding_all_scales}), one can use Corollary~\ref{cor:find_combinatorial_metric} from Appendix~\ref{App:Entropy_expansion} to arrange for $f = f_0$ (i.e., no homotopy is needed).
\end{remark}

Let  $f\colon\Gamma\to\Delta$ be a regular topological graph-map and let $v\in V\Gamma$. Define an equivalence relation $\sim_{f,v}$ on $\link_\Gamma(v)$ by saying that for $e,e'\in \link_\Gamma(v)$ we have $e\sim_{f,v} e'$ whenever the combinatorial edge-paths $f(e)$ and $f(e')$ start with the same (oriented) edge of $\Delta$.

\begin{defn}[Tame maps]\label{D:tame}
A topological graph-map $f\colon\Gamma\to\Delta$ is \emph{tame} if $f$ is regular and if for every vertex $v\in V\Gamma$ the relation $\sim_{f,v}$ partitions $\link_\Gamma(v)$ into at least two equivalence classes.
Note that if $f\colon\Gamma\to\Delta$ is tame then every vertex of $\Gamma$ has valence $\ge 2$.
\end{defn}

\begin{defn}[Taken turns and Whitehead graphs]
Let $f\colon\Gamma\to\Gamma$ be a regular graph-map.  A turn $e_1,e_2$ in $\Gamma$ is \emph{taken} by $f$ if there exist $e\in E\Gamma$ and $n\ge 1$ such that $e_1^{-1}e_2$ is a subpath of $f^n(e)$ (that is, if the turn $e_1,e_2$ is contained in $\f^n(e)$).   For any vertex $v\in V\Gamma$ we define the \emph{Whitehead graph}  of $v$,  denoted $\wh_\Gamma(f,v)$ or just $\wh(f,v)$ if $\Gamma$ is fixed, as follows.  The graph $\wh_\Gamma(f,v)$  is a simple graph whose vertex set is $\link_\Gamma(v)$. Thus the vertices of $\wh_\Gamma(f,v)$ are exactly all those $e\in E\Gamma$ such that $o(e)=v$. Two distinct elements $e_1,e_2\in \link_\Gamma(v)$ are adjacent as vertices in $\wh_\Gamma(f,v)$ whenever the turn $e_1,e_2$ is taken by $f$. (Note that since $e_1\ne e_2$, the turn $e_1,e_2$ is nondegenerate.)  
\end{defn}

\begin{remark}\label{R:WG}
For a regular graph-map $f\colon\Gamma\to\Gamma$, denote by $\turns_n(f)$ the set of all turns contained in the paths $\{f^k(e) :  e\in E\Gamma, 1\le k\le n\}$. Then  $\turns_n(f)\subseteq \turns_{n+1}(f)$ for all $n\ge 1$.  Put $\turns_\infty(f):=\cup_{n\ge 1} \turns_n(f)$. By definition, for a vertex $v$ of $\Gamma$ the Whitehead graph $\wh_\Gamma(f,v)$ records all the nondegenerate turns at $v$ that belong to $\turns_\infty(f)$. 
\end{remark}

Let $\Gamma$ be a finite topological graph with at least one edge and let $f\colon\Gamma\to\Gamma$ be a topological graph-map. Choose an orientation on $\Gamma$ and let $e_1,\dots, e_k$ be an ordering of $E_+\Gamma$.
The \emph{transition matrix} $A(f)=(a_{ij})_{i,j=1}^k$ is a $k\times k$ matrix where for $1\le i,j\le k$ the entry $a_{ij}$ is the total number of occurrences of $e_j^{\pm 1}$ in the path $f(e_i)$ (note that this definition gives the transpose of that used for the transition matrix in~\cite{BH92}).  Thus for every $n\ge 1$ we have $A(f^n)=A(f)^n$. 
Define $\lambda(f)$ to be the spectral radius of the transition matrix $A(f)$.

We say that $A(f)$ is \emph{positive}, denoted $A(f)>0$, if $a_{ij}>0$ for all $1\le i,j\le k$. We say that $A=A(f)$ is \emph{irreducible} if for every $1\le i,j\le k$ there exists $t=t(i,j) \ge 1$ such that $(A^t)_{ij}>0$. Thus if $A(f)>0$ then $A(f)$ is irreducible. Note that if $A(f)$ is irreducible then $\lambda(f)$ is exactly the Perron-Frobenius eigenvalue of $A(f)$. Furthermore, when $A(f)$ is irreducible, the property $\lambda(f) > 1$ is equivalent to $f$ being expanding, which is also equivalent to $A(f)$ not being a permutation matrix.

\begin{example}\label{Ex:whitehead_graphs}
Consider again our running example graph-map $f\colon\Gamma\to\Gamma$ defined in Example~\ref{Ex:perpetuating_example}. The four combinatorial edge-paths $f(a)=d$, $f(b)=a$, $f(c)=b^{-1}a$ and $f(d)=ba^{-1}db^{-1}ac$ are all reduced; therefore $f$ is a regular graph-map. Hence we also see that the links $\link_\Gamma(v)$ at the left and right vertices of $\Gamma$ are respectively partitioned into the equivalence classes
\[ \{a\}, \{b\},\{d\},\qquad\text{and}\qquad \{a^{-1}\}, \{b^{-1},c^{-1}\}, \{c\}, \{d^{-1}\},\]
by the relations $\sim_{f,v}$; thus $f$ is tame.
One may also calculate by hand the Whitehead graphs of $f$ for the vertices of $\Gamma$, which we have illustrated in Figure~\ref{F:whitehead}. 
Namely, in the notations of Remark~\ref{R:WG}, Figure~\ref{F:whitehead} shows all the nondegenerate turns that belong to $\turns_3(f)$ is this example. A direct check shows that  $\turns_3(f)=\turns_4(f)$ here and hence $\turns_\infty(f)=\turns_3(f)$. Therefore the graphs shown in Figure~\ref{F:whitehead} are exactly the Whitehead graphs of the two vertices of $\Gamma$. Note that both these Whitehead graphs are connected.
\begin{figure}[htb]

\begin{tikzpicture}[scale=.9]
\small
\path (0,0) coordinate (Lorig);
\path (Lorig) ++(120:1.5cm) node (a) {$a$};
\path (Lorig) ++(0:1.5cm) node (b)   {$b$};
\path (Lorig) ++(240:1.5cm) node (d) {$d$};

\path (8,0) coordinate (Rorig);
\path (Rorig) ++(90: 1.5cm) node (c) {$c$};
\path (Rorig) ++(162:1.5cm) node (A) {$a^{-1}$};
\path (Rorig) ++(234:1.5cm) node (B) {$b^{-1}$};
\path (Rorig) ++(306:1.5cm) node (D) {$d^{-1}$};
\path (Rorig) ++(18:1.5cm) node (C) {$c^{-1}$};

\tiny
\draw (a) -- node[sloped,above]{$f(d)$} 
      (b) -- node[sloped,below]{$f^3(c)$}
      (d) -- node[left]{$f(d)$}
      (a);

\draw (A) -- node[below,sloped]{$f(d)$} (B);
\draw (A) -- node[above,sloped]{$f^2(d)$}(C);
\draw (D) -- node[above,sloped]{$f^2(d)$}(A);
\draw (A) -- node[above,sloped]{$f(d)$} (c);
\draw (B) -- node[below,sloped]{$f(d)$} (D);
\draw (C) -- node[below,sloped]{$f^3(d)$} (D);
\end{tikzpicture}
\caption{The Whitehead graphs of the left and right vertices of $\Gamma$, respectively. We have labeled each adjacency by an edge-path that takes the corresponding turn in $\Gamma$.}
\label{F:whitehead}
\end{figure}
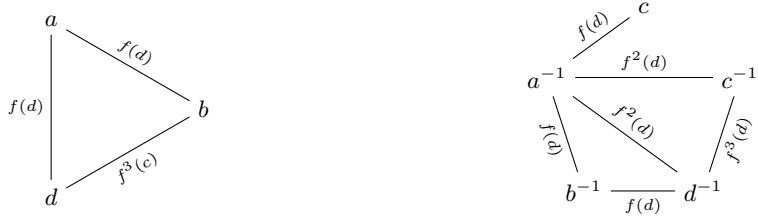

Finally, the transition matrix, Perron-Frobenius eigenvalue and corresponding unit eigenvector for this example are
\[ A(\f) = \begin{pmatrix}
0 & 0 & 0 & 1\\
1 & 0 & 0 & 0\\
1 & 1 & 0 & 0\\
2 & 2 & 1 & 1 
\end{pmatrix}, \qquad  \lambda \approx 2.4142, \qquad {\bf {v}} \approx \begin{pmatrix}
.2265 \\ .0939 \\ .1327 \\ .5469 \end{pmatrix}.\]
The characteristic polynomial of $A(f)$ is $x^4-x^3-2x^2-3x-1$. It is easy to check that $A(f)^3$ is positive, so $A(f)$ is irreducible. Since $\lambda(f) > 1$, we also see that $f$ is expanding.
\end{example}

\subsection{Entropy and expansion} \label{S:entropy_expansion_background}

The next two results are likely well known, but we could not find an explicit reference in the literature.  In Appendix \ref{App:Entropy_expansion}, we have included proofs of these for completeness, as well as some references for closely related results.

The first result tells us that under fairly mild assumptions on a topological graph map, the topological entropy can be calculated by the transition matrix.

\begin{proposition:h(f)}
Let $\Gamma$ be a finite, connected, topological graph with at least one edge and let $f\colon\Gamma\to\Gamma$ be regular topological graph-map.  Let $h(f)$ be the topological entropy of $f$. 
Then $h(f)=\log \lambda(f)$ (recall that $\lambda(f)$ is the spectral radius of $A(f)$).
\end{proposition:h(f)}

The next theorem gives a very strong rigidity result which, among other things, allows one to promote a topological graph map to a (linear) graph map, with respect to a carefully chosen linear structure, in certain situations.  To state the theorem, we first make the following

\begin{defn}\label{defn:expanding_all_scales}
A topological graph-map $\f\colon\Gamma \to \Gamma$ is {\em expanding on all scales} if for every $x \in \Gamma$, every neighborhood $U$ of $x$, and every edge $e$ of $\Gamma$, there exists a positive integer $m$ and an open interval $W \subset U$ so that $\f^m$ maps $W$ homeomorphically onto $e$.
\end{defn}

\begin{theorem:Topological graph to graph}
Suppose $\f\colon\Gamma \to \Gamma$ is a regular topological graph-map that is expanding on all scales.  Choose an orientation on $\Gamma$ and an ordering $E_+\Gamma = \{e_1,\ldots,e_n\}$. 

Then the topological graph-map $\f$ is expanding, the matrix $A(\f)$ is irreducible with the Perron-Frobenius eigenvalue  $\lambda=\lambda(f)>1$, and  there exists a unique volume-1 metric structure $\mathcal L$ on $\Gamma$ with the following properties:

\begin{enumerate}
\item $\f \colon \Gamma \to \Gamma$ is linear with respect to $\mathcal L$;
\item $\f$ is a local $\lambda$--homothety, meaning that the norm of the derivative of $\f$ at every point of $\Gamma \setminus \f^{-1} (V\Gamma)$ is $\lambda$; and
\item For every $i=1,\dots, n$ the $d_\mathcal L$--length of $e_i$ is equal to $v_i$ where ${\bf v}=(v_1,\dots, v_n)^T$ is the left Perron-Frobenius eigenvector (i.e.~for the left action of $A(\f)$ on column vectors) with $\sum_{i=1}^n v_i =1$.
\end{enumerate}
\end{theorem:Topological graph to graph}

As we show in Appendix \ref{App:Entropy_expansion}, expanding irreducible graph maps are expanding on all scales; see Lemma \ref{L:linear maps expanding on all scales}.  Consequently, as explained Corollary~\ref{cor:find_combinatorial_metric}, given any expanding, irreducible graph map $f \colon \Gamma \to \Gamma$, there is a choice of (possibly different) metric structure on $\Gamma$ making $f$ into a combinatorial graph map (Definition~\ref{D:combinatorial-graph-map}).

\subsection{Train-tracks}\label{subject:tt}
To obtain sufficient information to prove Theorems~\ref{T:new iwip main} and \ref{T:new entropy main} we will make heavy use of a certain class of graph-maps introduced in \cite{BH92} that have additional structure. This extra structure ensures that the map and all of its powers are combinatorially efficient. 

\begin{defn}[Train-track map]\
\label{D:train-track}
Let $\Gamma$ be a finite connected (linear) graph with at least one edge and without valence-1 vertices. A (linear) graph-map $f\colon\Gamma\to\Gamma$ is a \emph{train-track map} if:
\begin{enumerate}
\item The map $f$ is a homotopy equivalence (and hence $f$ is surjective).
\item For every edge $e\in E\Gamma$ and every $n\ge 1$ the map $f^n|_e$ is an immersion; that is, the edge-path $f^n(e)$ is tight and nondegenerate. (Note that this implies that for every $n\ge 1$ the map $f^n\colon \Gamma\to\Gamma$ is regular.)
\end{enumerate}
\end{defn}
The original definition of train-track map \cite{BH92} does not allow valence $2$ vertices.  The next lemma guarantees that no complications occur near such vertices.

\begin{lemma} \label{L:tt-tame}
Suppose that $\f \colon \Gamma \to \Gamma$ is a train-track map.   Then $f$ is tame.  Furthermore, if $v$ is a vertex of valence $2$ in $\Gamma$ then for every $n\ge 1$ the map $f^n$ is an immersion near $v$; that is, for every $n\ge 1$ there exists an open neighborhood $J$ of $v$ in $\Gamma$ such that $f^n|_{J}$ is injective.
\end{lemma}
\begin{proof}  By condition (2) of Definition~\ref{D:train-track}, both conclusions of the lemma will follow if we can prove that there is a power $m_0 > 0$ so that for every vertex $v \in V\Gamma$, either $\f^{m_0}(v) = v$ and $\f^{m_0}$ is injective on a neighborhood of $v$ or else there is a turn of $v$ taken by $\f^{m_0}$.  

Let $\mathcal V \subset V\Gamma$ be the set of {\em $\f$--stable} vertices: that is, those vertices which are eventually permuted by $\f$.  Let $m >0$ be such that $\f^m\vert_{\mathcal V} = id_{\mathcal V}$, and so that $\f^m(V\Gamma) = \mathcal V$.   First observe that if $v \in V\Gamma \setminus \mathcal V$, then since $\f^m$ is surjective by (1), there must be a turn of $v$ taken by $\f^m$.  It follows that any $v \in \f^m(V\Gamma \setminus \mathcal V) \subset \mathcal V$ has a turn taken by $\f^{2m}$ (the image of a turn taken by $\f^m$).  

Finally, if $v \in \mathcal V \setminus \f^m(V\Gamma \setminus \mathcal V)$, then we claim that $\f^m$ is locally injective on some neighborhood of $v$ or else there is a turn of $v$ taken by $\f^m$.  To prove this claim, suppose that $\f^m$ is not locally injective on a neighborhood of $v$.  Then the map $\link_\Gamma(v) \to \link_\Gamma(v)$ induced by $\f^m$ (defined since $\f^m(v) = v$)  would fail to be injective, and hence must also fail to be surjective, missing some $e \in \link_\Gamma(v)$.  Surjectivity of $\f^m$ implies that there must be some edge $e'$ so that $e$ is an edge of $\f^m(e')$.  A point of $\bar e'$ sent to $o(e) = v$ cannot be a vertex (by assumption on $v$), and hence must be an interior point.   But this means that $v$ has a turn that is taken by $\f^m$ as required.  Setting $m_0 = 2m$ completes the proof.
\end{proof}

An \emph{irreducible train track map} is a train-track map $f\colon\Gamma\to\Gamma$ such that the transition matrix $A(f)$ is irreducible. 
In this case, the Perron-Frobenius theory tells us that $\lambda(f)\ge 1$ is an eigenvalue of $A(f)$, that there exists a unique (up to a scalar multiple) left eigenvector $\mathbf v$ for $A(f)$ with $A(f){\mathbf v}=\lambda(f) {\mathbf v}$, and that all the coordinates of $\mathbf v$ are positive. We normalize $\mathbf v$ so that the sum of its coordinates is equal to $1$. The coordinates of $\mathbf v$ then define a length function $E\Gamma\to \R_+$ which together with the linear structure induces a metric structure $\mathcal L$ on $\Gamma$ so that $\vol_{\mathcal L}(\Gamma)=1$ and such that for every edge $e\in E\Gamma$ we have $|f(e)|_\mathcal L=\lambda(f)\mathcal L(e)=\lambda(f)|e|_\mathcal L$.  We call this $\mathcal L$ the \emph{train-track metric structure} on $\Gamma$ and the geodesic metric $d_\mathcal L$ the \emph{train track-metric} on $\Gamma$. 

\begin{remark}\label{R:tt-metric}
An irreducible train-track map $f\colon \Gamma\to\Gamma$ typically need not be a local $\lambda(f)$--homothety with respect to the train-track metric $d_\mathcal L$. 
The reason is that $\Gamma$ comes equipped with a linear structure $\Lambda$ such that for each $e\in E\Gamma$ there exists a unique chart map $j_e\colon [0, \mathcal L(e)]\to \Gamma$ in $\Lambda$ mapping $(0,\mathcal L(e))$ homeomorphically to $e$. If $f(e)=e_1\dots e_m$ then there is a unique subdivision $0=t_0<t_1<\dots < t_m=\mathcal L(e)$ such that $f\circ j_e$ maps $(t_{i-1},t_i)$ homeomorphically to $e_i$. Since the relative lengths of the subintervals $(t_{i-1},t_i)$ are specified by the linear structure, there is in general no reason to assume that we will have $\lambda(f) |t_i-t_{i-1}|=\mathcal L(e_i)$ for $i=1,\dots, m$. Thus, even though $\mathcal L(f(e))=\lambda(f)\mathcal L(e)$, the map $f$ may stretch the metric $d_\mathcal L$ on the subsegments $j_e([t_{i-1},t_i])$ of $e$ by different amounts. 
However, Corollary \ref{cor:eigenmetric} shows that one may \emph{change} the linear structure on $\Gamma$ in such a way that $f$ remains a linear graph-map and that, with respect to the new train track metric, $f$ \emph{is} locally a $\lambda(f)$--homothety on edges.
\end{remark}

If $f\colon\Gamma\to\Gamma$ is a regular topological graph-map, we have a natural \emph{derivative map} $D_f\colon E\Gamma\to E\Gamma$, where for $e\in E\Gamma$ we define $D_f(e)$ to be the initial edge of the edge-path $f(e)$. The map $D_f$ naturally extends to $D_f\colon \turns(\Gamma)\to \turns(\Gamma)$ by setting $D_f(\{e,e'\}):=\{D_f(e), D_f(e')\}$, where $e,e'$ is a turn in $\Gamma$.  A turn $\mathbf t\in \turns(\Gamma)$ is then called \emph{legal} for $f$ if $(D_f)^n(\mathbf t)$ is non-degenerate for every $n\ge 0$. 

\begin{remark}\label{R:train-track}
Consider a finite graph $\Gamma$ without valence-$1$ vertices and a homotopy equivalence regular (linear) graph-map $f\colon\Gamma\to\Gamma$. In this situation it is easy to see that $f$ is a train-track map if and only if the unique non-degenerate turn at each valence-$2$ vertex is legal and every turn contained in an edge-path of the form $f(e)$ for $e\in E\Gamma$ is legal. Moreover, it is easy to check whether a given turn is legal because the set $\turns(\Gamma)$ is finite. Therefore, it is straightforward and algorithmic to determine whether the given map $f$ is a train track-map.
\end{remark}

\begin{example}\label{Ex:train-track}
The graph-map $f$ introduced in Example~\ref{Ex:perpetuating_example}
is actually a train-track map. This fact can be easily checked by hand using the algorithm explained in Remark~\ref{R:train-track} above, and we have also verified that $f$ is a train-track map using Peter Brinkmann's software package {\rm xtrain}\footnote{Available at http://gitorious.org/xtrain}.
\end{example}

\subsection{Outer automorphisms}\label{subsect:outer}

For an integer $N\ge 2$ let $F_N=F(x_1,\dotsc, x_N)$ be the free group of rank $N$. Recall that  $\Out(F_N)$ is the group of all automorphisms of $F_N$ modulo the subgroup of inner automorphisms

An element $\fee\in \Out(F_N)$ is said to be \emph{reducible} if there exists a free product decomposition $F_N=A_1\ast \dots \ast A_k \ast C$ with $k\ge 1$ and each $A_i\ne \{1\}$ (and where it is possible that $C=\{1\}$ only in the case $k \ge 2$) such that $\fee$ permutes the conjugacy classes $[A_1],\dots, [A_k]$. 

An element $\fee\in \Out(F_N)$ is \emph{irreducible} if it is not reducible. 
Moreover, $\fee\in \Out(F_N)$ is said to be \emph{fully irreducible} or \emph{iwip} (for ``irreducible with irreducible powers'') if $\fee^n$ is irreducible for every $n\ge 1$. Thus $\fee$ is fully irreducible if and only if there do not exist $n\ge 1$ and a proper free factor $A$ of $F_N$ such that $\fee^n([A])=[A]$.  Accordingly, $\Phi\in \Aut(F_N)$ is said to be \emph{irreducible} (respectively, \emph{fully irreducible}) if the outer automorphism class  $\fee\in\Out(F_N)$ of $\Phi$ is irreducible (respectively, fully irreducible). 

Finally, $\fee\in \Out(F_N)$ (or an element of $\Aut(F_N)$ representing $\fee$) is called \emph{atoroidal} or \emph{hyperbolic} if there do not exist $n\ge 1$ and $w\in F_N, w\ne 1$, such that $\fee^n$ preserves the conjugacy class $[w]$ of $w$ in $F_N$.
A result of Brinkmann~\cite{Brink} (utilizing the Bestvina--Feighn Combination Theorem~\cite{BF92}) says that $\fee\in \Aut(F_N)$ is hyperbolic if and only if the mapping torus group $G_\fee=F_N\rtimes_\fee \mathbb Z$ of $\fee$ is word-hyperbolic.

\subsection{Topological representatives}\label{subsect:mark}

Given a finite connected graph $\Gamma$ and a homotopy equivalence topological graph-map $\f\colon\Gamma\to\Gamma$, let $v\in V\Gamma$ be a vertex and let $\beta$ be an edge-path from $v$ to $f(v)$. We denote by $\f_\ast = \f^\beta_\ast\colon \pi_1(\Gamma,v)\to \pi_1(\Gamma,v)$ the map given by $\f_\ast([\gamma])=[\beta (\f\circ \gamma) \beta^{-1}]$, where $\gamma$ is an edge-path from $v$ to $v$ in $\Gamma$.  We omit mention of $\beta$ in the notation for $f_\ast$ as we will be interested in the outer class of $f_\ast$ which is independent of the choice of path $\beta$.

\begin{defn}[Topological representative]\label{D:TR}
Let $\fee\in \Out(F_N)$, where $N\ge 2$.  A \emph{topological representative} of $\fee$ is a pair $(\ma, \f)$ where $\f\colon\Gamma\to\Gamma$ is a homotopy equivalence topological graph-map of a finite connected topological graph $\Gamma$ to itself, and where $\ma\colon F_N\to \pi_1(\Gamma,v)$ is an isomorphism such that $\ma^{-1}\circ \f_\ast\circ \ma\colon F_N\to F_N$ is an automorphism of $F_N$ whose outer automorphism class is $\fee$ (this condition is independent of the choice of edge-path $\beta$ from $v$ to $f(v)$).  The isomorphism $\ma$ is called a \emph{marking}. By convention, we will usually omit the explicit mention of marking and just say that $\f\colon\Gamma\to\Gamma$ is a topological representative of $\fee$.
\end{defn}

If \textbf{P} is a type of topological graph-map (such as ``regular'' or ``linear''), we say that a topological representative $f\colon\Gamma\to\Gamma$ \emph{is \textbf{P}} if the topological graph-map $f$ is of type \textbf{P}. Alternately we may say that $f\colon\Gamma\to\Gamma$ is a \emph{\textbf{P} representative} of $\fee$. Thus we may discuss regular representatives, tame representatives, irreducible train-track representatives, and so on. We emphasize that if $f$ is a train-track representative of $\fee$, then this automatically implies that $f$ is a regular topological representative of $\fee$ and that $f^n\colon\Gamma\to\Gamma$ is a train-track representative of $\fee^n$ for all $n\ge 1$.

Given any $\fee \in \Out(F_N)$, one can easily construct a regular topological representative for $\fee$ with $\Gamma$ being a wedge of circles. Constructing representatives with more structure is in general harder and may depend on the outer automorphism $\fee$. An important general result of Bestvina and Handel~\cite{BH92} says that for every irreducible $\fee\in \Out(F_N)$ (where $N\ge 2$) there exists an irreducible train-track topological representative of $\fee$.  The converse is, in general, not true: it is possible for an irreducible train-track map to be a topological representative of a reducible $\fee\in \Out(F_N)$. However, if $\fee\in \Out(F_N)$ is fully irreducible then every train-track representative of $\fee$ has irreducible transition matrix and thus is an irreducible train-track map. 
We discuss the relationship between fully irreducible hyperbolic outer automorphisms and the properties of their train-track representatives in greater detail in the next subsection and in Appendix \ref{sect:clean}.

\begin{example}\label{Ex:hyperbolic_auto}
We now return to the graph map $f\colon \Gamma\to\Gamma$ introduced in Example~\ref{Ex:perpetuating_example}. We have $\pi_1(\Gamma)\cong F_3$ and it is not hard to check directly that $f$ is a homotopy equivalence; thus $f$ is a topological representative of an element $\fee\in \Out(F_3)$.  Specifically, use $v_0=o(c)$ as a  base-vertex in $\Gamma$ and set $\mathfrak m \colon F_3 = F(x_1,x_2,x_3) \to \pi_1(\Gamma,v_0)$ by $\mathfrak m(x_1)=b^{-1}a$, $\mathfrak m(x_2)=a^{-1}d$ and $\mathfrak m(x_3)=c$.
Then the outer class $\fee$ is represented by $\fee = \mathfrak m^{-1} \circ f_\ast \circ \mathfrak m$ where $f_\ast \colon \pi_1(\Gamma,v_0) \to \pi_1(\Gamma,v_0)$ (where $\beta$ in the definition of $f_\ast$ is the constant path from $v_0$ to itself).  This automorphism of $F_3$ is given by $\phi(x_1)=x_2$, $\phi(x_2)=x_2^{-1}x_1^{-1}x_2x_1x_3$ and $\phi(x_3)=x_1$. We have moreover verified, with the help of Derek Holt and the computational software package {\rm kbmag}\footnote{Available at  http://www.gap-system.org/Packages/kbmag.html}, that the group $G_\fee$ is word-hyperbolic and therefore that the automorphism $\fee$ is hyperbolic. In Example~\ref{Ex:iwip} we will see that $\fee$ is also fully irreducible.
\end{example}

\subsection{Criterion for full irreducibility}

We end this section with a criterion for a hyperbolic automorphism to be fully irreducible.  Stating this criterion requires the notion of ``clean train-tracks'' from \cite{K12}:

\begin{defn}[Clean train-track]
\label{defn:clean_track}
We say that a train-track map $\f\colon \Gamma\to\Gamma$ is  \emph{clean} if the Whitehead graph $\wh(\f,v)$ is connected for each vertex $v\in V\Gamma$, and if there exists $m>0$ such that $A(\f^m)>0$. (This implies $\f$ is expanding, that $A(\f^n)=A(\f)^n$ is irreducible for all $n\ge 1$, and that $A(\f^n)>0$ for all $n\ge m$.) We say that $\f$ is \emph{weakly clean} if $\f$ is expanding, if the Whitehead graph $\wh(\f,v)$ is connected for each vertex $v\in V\Gamma$ and if $A(f)$ is irreducible.
\end{defn}

Our next theorem provides a criterion for full irreducibility in terms of the existence of weakly clean train-track representatives; this is a strengthening of the criterion proved in \cite{K12}. 

\begin{theorem:clean}
Let $N\ge 3$ and let $\fee\in \Out(F_N)$ be a hyperbolic element. Then the following conditions are equivalent:
\begin{enumerate}
\item The automorphism $\fee$ is fully irreducible.

\item If $\f\colon\Gamma\to\Gamma$ is a train-track representative of $\fee$ then $\f$ is clean.

\item There exists a clean train-track representative of $\fee$.

\item There exists a weakly clean train-track representative of $\fee$.

\end{enumerate}
\end{theorem:clean}

The proof of Theorem~\ref{T:clean} is given in Appendix \ref{sect:clean}. This result seems to be of independent interest and also has some interesting consequences (e.g.,~Corollary~\ref{C:hyp_iwip}). However, as these considerations are tangential to the main results of the paper, we have relegated the proofs and discussion to an appendix.

\begin{example}\label{Ex:iwip}
We have already seen that $\f\colon\Gamma\to\Gamma$ in
Example~\ref{Ex:perpetuating_example} is an expanding irreducible
train-track map representing an automorphism $\fee\in \Out(F_3)$. We
also know that $\fee$ is hyperbolic -- this fact was verified (see Example~\ref{Ex:hyperbolic_auto})  by checking that the group $G_\fee$ is word-hyperbolic.
We have additionally verified in Example~\ref{Ex:whitehead_graphs} that the Whitehead graphs of both vertices of $\Gamma$ with respect to $\f$ are connected. Thus $f$ is a weakly clean train-track map. Therefore, by Theorem~\ref{T:clean}, $\fee$ is fully irreducible.
\end{example}

\section{Folding paths} \label{sect:folding}

We will need a continuous version of Stallings' folding line adapted to our setting
(see~\cite{St83,KM02} for the classic version of Stallings folds). 
Our version of the folding line is more elementary and less canonical than what Francaviglia and Martino call ``fast folding line''
(see Definition~5.6 in~\cite{FM11}) and what Bestvina and Feighn call
``folding all illegal turns at speed 1'' (see the definition
immediately prior to Proposition~2.2 in \cite{BF11}). While we could have used the ``fast folding line'' when defining the ``folded mapping torus'' $X_\f$ in \S\ref{sect:folded mapping torus and semiflows} below, the more basic Stallings folds construction is sufficient for our purposes and simpler to analyze.

\begin{conv} \label{Conv:simplicial for construction}
In this section let $\fee\in \Out(F_N)$, let $\Gamma$ be a finite
connected linear graph and let $\f\colon\Gamma\to\Gamma$ be a linear
graph-map that is a tame topological
representative of $\fee$ (see Definitions~\ref{D:top-graph-map}, \ref{D:reg-exp-graph-map}, \ref{D:tame}, and \ref{D:TR}). Additionally let $\mathcal L$ be a metric structure on $\Gamma$ compatible with the given linear graph structure.

Note that the tameness assumption implies every vertex of $\Gamma$ has valence $\ge 2$ and that $\f(e)$ is a tight nondegenerate edge-path in $\Gamma$ for every edge $e\in E\Gamma$.
\end{conv}

\begin{remark}\label{R:tame}
Given any $\fee\in \Out(F_N)$ it is easy to construct a tame topological representative of $\fee$.
This can be done, for example, via replacing an arbitrary regular topological representative $f'\colon\Gamma\to\Gamma$ of $\fee$ by an ``optimal'' map $f\colon\Gamma\to\Gamma$ that is freely homotopic to $f$, as explained in~\cite{FM11}. Alternatively, given $\fee\in \Out(F_N)$, one can always find, via an explicit combinatorial construction, a tame representative $f\colon\Gamma\to\Gamma$ of $\fee$ with $\Gamma$ being a wedge of $N$ circles, as explained in~\cite{KR12}.
(Also, Lemma~\ref{L:tt-tame} implies that if $f\colon\Gamma\to\Gamma$ is an expanding irreducible train-track map, then $f$ is tame.) As indicated after Definition~\ref{D:top-graph-map}, such a topological representative may then be adjusted by a homotopy rel $V\Gamma$ to yield a linear graph-map (as required by Convention~\ref{Conv:simplicial for construction}) with respect to any given linear structure on $\Gamma$
\end{remark}

While our constructions of folding paths (\S\ref{sect:folding}) and associated folded mapping tori (\S\ref{sect:folded mapping torus and semiflows}) apply to any map $f$ as in Convention~\ref{Conv:simplicial for construction}, working in this level of generality introduces technicalities that significantly complicate the notation. To enhance readability, we therefore only describe the constructions in the following setting:

\begin{conv}\label{conv:fold path readability}
For expository purposes only, we additionally suppose that $f\colon \Gamma\to\Gamma$ is a combinatorial graph-map with respect to the given metric structure $\mathcal L$ (see Definition~\ref{D:combinatorial-graph-map}).
\end{conv}

\begin{remark}\label{R:restricting to combinatorial}
We emphasize that this choice to work with combinatorial graph-maps only serves to simplify the notation and that it is straightforward to adjust the constructions in \S\S\ref{sect:folding}--\ref{sect:folded mapping torus and semiflows} to handle the general setting of Convention~\ref{Conv:simplicial for construction}. In fact, the curious reader may find these details in \S4 and \S6 of a previous version of this paper (available at \url{http://arxiv.org/abs/1301.7739v5}) where we work in the more general situation.

Moreover, this additional requirement that $f\colon\Gamma\to\Gamma$ be a combinatorial graph-map is not a serious constraint: Remark~\ref{rem:make_combinatorial} explains that any regular topological graph-map is homotopic to a combinatorial graph-map (with respect to an appropriate metric structure), and Corollary~\ref{cor:find_combinatorial_metric} shows that oftentimes the homotopy is unnecessary.
\end{remark}

Now we form a new labeled metric graph $\Delta$ as follows. We have $\Delta=\Gamma$ as a set and as a topological space. 
The vertex set of $\Delta$ is $V\Delta:=f^{-1}(V\Gamma)\cup V\Gamma$ and the $1$--cells of $\Delta$ are the connected components of $\Delta- V\Delta$. 
Thus, combinatorially, $\Delta$ is obtained from $\Gamma$ by subdividing each edge $e$ of $\Gamma$ into $|\f(e)|$ edges.  The graph $\Delta$ inherits a natural linear graph structure from $\Gamma$ and we still have the graph-map $\f\colon \Delta\to\Gamma$ which is now an affine homeomorphism on every open $1$--cell.  We endow $\Delta$ with a metric structure compatible with the linear structure in which every edge has length $1$. 
Note that this metric structure is not the one that the subdivided graph $\Delta$ naturally inherits from $\Gamma$ (recall that Definition~\ref{D:combinatorial-graph-map} requires $\Gamma$ to be a metric graph).
 For every edge $e\in E\Delta$ and every vertex $v\in V\Delta$ we think of $\f(e)\in E\Gamma$ as the \emph{label} of $e$ and of $\f(v)\in V\Gamma$ as the \emph{label} of $v$. This labeling satisfies the property that $\f(e^{-1})=(\f(e))^{-1}$ and that $o(\f(e))=\f(o(e))$, $t(\f(e))=\f(t(e))$ for all $e\in E\Delta$. In general, we call a (linear) graph $\Psi$ with a labeling $\tau\colon V\Psi\to V\Gamma$, $\tau\colon E\Psi\to E\Gamma$ of vertices and edges, satisfying these properties, a \emph{$\Gamma$--graph}. Any $\Gamma$--graph $\Psi$ comes equipped with a canonical graph-map $\tau\colon\Psi\to \Gamma$, given by the labeling, which is an affine homeomorphism on open edges, and a metric compatible with the linear structure giving each edge length $1$.  

Since $\tau\colon\Psi\to \Gamma$ sends an edge to an edge, $\tau$ is a regular graph-map. A $\Gamma$--graph $\Psi$ is \emph{folded} if whenever $e,e'\in E\Psi$ are such that $o(e)=o(e')$ and $\tau(e)=\tau(e')$ then $e=e'$. It is easy to see that for a $\Gamma$--graph $\Psi$ the map $\tau\colon\Psi\to\Gamma$ is locally injective if and only if $\Psi$ is folded. If a $\Gamma$--graph $\Psi$ is not folded and $e\ne e'$ are elements of $E\Psi$ such that $o(e)=o(e')$ and $\tau(e)=\tau(e')$, then identifying $e$ and $e'$ into a single edge with label $\tau(e)=\tau(e')$ produces a new $\Gamma$--graph $\Psi'$. We say that $\Psi'$ is obtained from $\Psi$ by a \emph{combinatorial Stallings fold}.  

We say that a $\Gamma$--graph $\Psi$ is \emph{tame} if the  labeling graph-map $\tau\colon\Psi\to\Gamma$ is tame.
Note that for a $\Gamma$--graph $\Psi$ and a vertex $v$ of $\Psi$, the number of equivalence classes in the partition of $\link_\Gamma(v)$ by the relation $\sim_{\tau,v}$ is equal to  the number of distinct labels of edges of $\Psi$ with origin $v$. Thus a $\Gamma$--graph $\Psi$ is tame  if and only if for every vertex $v$ of $\Psi$ there exist some two edges $e,e'\in E\Psi$ with $o(e)=o(e')=v$ and $\tau(e)\ne \tau(e')$. Hence if $\Psi$ is a tame $\Gamma$--graph and  $\Psi'$ is obtained from $\Psi$ by a combinatorial Stallings fold then $\Psi'$ is again tame.

We now return to the construction of a folding path.
According to \cite{St83}, we can obtain $\Gamma$ from $\Delta$ by a finite sequence of combinatorial Stallings folds 
\begin{equation}\label{E:foldings}
\Delta=\Delta_0, \Delta_1,\dots, \Delta_m=\Gamma 
\end{equation}
where $m=\#E_{top}\Delta-\#E_{top}\Gamma$ and where each $\Gamma$--graph $\Delta_{i+1}$ is obtained from $\Delta_{i}$ by a single combinatorial Stallings fold identifying the edges $e_i,e_i'$ of $\Delta_i$.  We endow each $\Delta_i$ with its natural $\Gamma$--graph metric and linear structure as described above.

Note that since by assumption, the map $f\colon\Gamma\to \Gamma$ is tame, the $\Gamma$--graph $\Delta=\Delta_0$ is tame. Since Stallings folds preserve tameness, it follows that each $\Delta_i$ is tame. In particular, this means that every vertex of $\Delta_i$ has valence $\ge 2$.

We now define a $1$--parameter family of metric graphs $\{\Gamma_t\}_{t \in [0,1]}$ as follows.  For each $j = 0,\ldots,m$, set $\Gamma_{j/m} = \Delta_j$.  For any $j = 0,\ldots,m-1$ and arbitrary $t \in [j/m,(j+1)/m]$, let $\Gamma_t$ be obtained from $\Delta_j = \Gamma_{j/m}$ by isometrically identifying initial segments of $e_j$ and $e_j'$ of length $tm-j \in [0,1]$.  We may endow $\Gamma_t$ with a metric and associated linear structure so that the quotient map $\Gamma_{j/m} \to \Gamma_t$ is a local isometry on each edge, and we do so.

Note that for each $j = 0,\ldots,m-1$, the two definitions of $\Gamma_{(j+1)/m}$ as $\Delta_{j+1}$ and as a quotient of $\Gamma_{j/m} = \Delta_j$ agree by definition of combinatorial Stallings fold.  Furthermore, we note that $\Gamma_0$ is obtained from $\Gamma$ by subdividing and changing the metric, but as topological spaces, $\Gamma_0 = \Gamma = \Gamma_1$.

\begin{remark}\label{R:deg1}
We have noted above that each $\Delta_i$ is tame and thus has no valence-1 vertices. Therefore for every $0\le t\le 1$ the graph $\Gamma_t$ has no valence-1 vertices either, since the new vertices that appear in the middle of performing a single Stallings fold have valence-$3$.  
\end{remark}

Each $\Gamma_t$ comes equipped with two natural homotopy equivalence graph maps
\[ r_{t,0} \colon \Gamma_0 \to \Gamma_t \mbox{ and } r_{1,t} \colon \Gamma_t \to \Gamma_1 \]
which are locally isometric on edges and such that
\begin{equation} \label{Eqn:folding composition gives f}
f = r_{1,t} \circ r_{t,0} = r_{1,0} \colon \Gamma_0 \to \Gamma_1,
\end{equation} where we have used the identification of $\Gamma_0 = \Gamma = \Gamma_1$ as topological spaces to make sense of this equality.
In fact, for any $0 \le t \le t' \le 1$, there is a canonical homotopy equivalence map $r_{t',t} \colon \Gamma_t \to \Gamma_{t'}$ which is a local isometry on edges and so that for all $0 \le t \le t' \le t'' \le 1$ we have
\begin{equation} \label{Eqn:semiflow for folding maps}
r_{t'',t'} \circ r_{t',t} = r_{t'',t}.
\end{equation}

We refer to the family $\{\Gamma_t\}_{t\in [0,1]}$, together with the maps $\{r_{t',t}\}_{0 \le t \le t' \le 1}$ defined above, as the \emph{topological folding line} corresponding to the topological representative $f\colon\Gamma\to \Gamma$ of $\fee$.  We write $\{ (\Gamma_t, r_{t',t})\}_{ \{0 \leq t \leq t' \leq1 \} }$ to denote the topological folding line.

This folding line also depends on the choice of a sequence of combinatorial Stallings folds~(\ref{E:foldings}) taking $\Delta$ to $\Gamma$, but this dependence will not affect the proofs of our main results (though the objects themselves do depend on these choices, for notational simplicity we have suppressed this dependence from the notation in the following).

\begin{example}\label{Ex:Stallings_folds}
Let $f\colon\Gamma\to\Gamma$  be the graph-map from Example~\ref{Ex:perpetuating_example}.  Figure~\ref{F:folds} illustrates the first two combinatorial Stallings folds in the folding sequence taking the $\Gamma$--graph $\Delta$ to $\Gamma$ for the graph map $f\colon\Gamma\to\Gamma$.

\begin{figure} [htb]
\begin{center}
\includegraphics[height=7cm]{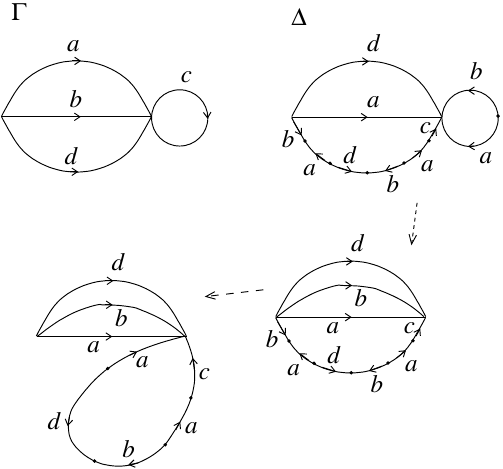} \caption{Two combinatorial Stallings folds used to construct a topological folding line.}
\label{F:folds}
\end{center}
	\end{figure}
\end{example}

\begin{remark}
Technically, in the standard Outer space terminology, the ``folding path'' $\Gamma_t$ is a family of ``marked metric graphs'' which connects the ``marked metric graph'' $(\ma, \Gamma_0)$ (where $\ma\colon F_N\to \pi_1(\Gamma_0)$ is the marking given by the fact that $\f$ is a topological representative of $\fee$ and that $\Gamma_0$ is a subdivision of $\Gamma$) to the ``marked metric graph'' $(\f_\ast\circ \ma, \Gamma)$. However,  the markings on the graphs $\Gamma_t$ are not important for the specific applications used in this paper and therefore we did not record them in our definition.

For the purposes of constructing the ``folded mapping torus'' of $\f$ in \S\ref{sect:folded mapping torus and semiflows}  below, we will mainly forget about the metric structures $\mathcal L_t$ on the graphs $\Gamma_t$ and only look at the underlying $1$--complexes $\Gamma_t$ and the topological maps $\{r_{t',t}\}$ provided by the above construction. The metric structures on $\Gamma_t$ will occasionally be useful, and, in particular, they will make it  simpler to describe the cell structure on the ``folded mapping torus'' of $\f$ defined in the next section. 
\end{remark}

\section{The folded mapping torus and semiflows.} \label{sect:folded mapping torus and semiflows}

\begin{conv}\label{conv:nice}
Let $\fee \in \Out(F_N)$ and set $G = G_{\fee}$.  Fix a linear graph-map $\f\colon\Gamma \to \Gamma$ that is a tame topological representative $\fee$, as in Convention~\ref{Conv:simplicial for construction}, and a topological folding line $\{ \Gamma_t,r_{t',t})\}_{ \{0 \leq t \leq t' \leq1 \} }$.  We also continue to identify $\Gamma_0 = \Gamma = \Gamma_1$ as topological spaces (though $\Gamma_0$ is a subdivision of $\Gamma = \Gamma_1$ as a topological graph). For expository purposes (see Remark~\ref{R:restricting to combinatorial}), we continue to assume that $\f$ is a combinatorial graph-map as in Convention~\ref{conv:fold path readability}.
\end{conv}

In this section, given the above data,  we describe a cell complex we call the {\em folded mapping torus} of $\f$, denoted $X_\f$, which is a particular $K(G,1)$ built from the data of the folding line.  We also construct a semiflow on $X_\f$ that encodes the folding line, and prove a few basic properties of $X_\f$.  We end by describing the cell structure in a little more detail and defining the $E_0$--class $\epsilon \in H_1(X_\f;\Z)$ to be used in Theorem \ref{T:folded basics}.

\begin{remark}
The construction of another $K(G,1)$ $2$--complex from $\fee \in \Out(F_N)$---Gautero's ``dynamical $2$--complex''---and a semiflow is carried out by Gautero in \cite[Proposition 3.9, Proposition 4.3]{Gau1} using a sequence of Whitehead moves in place of the folding path;  see also the modification in \cite{Gau4}.  Similarly, Wang's construction \cite{Wang} of the decomposed mapping torus is similar to our construction via ``building blocks" as explained in the proof of Proposition \ref{P:X' cell structure} below. 
\end{remark}

\subsection{The folded mapping torus.} \label{sect: folded mapping torus}

Starting with the folding line $\{(\Gamma_t,r_{t',t})\}_{ \{0 \leq t \leq t' \leq1 \} }$, we define an equivalence relation on $\Gamma \times [0,1]$ by declaring $(x,t)$ to be equivalent to $(x',t')$ if and only if $t = t'$ and $r_{t,0}(x) = r_{t,0}(x')$.  We denote the associated quotient space by
\[ \Pi' \colon \Gamma \times [0,1] \to X'.\]
By construction, we can identify $\Pi'(\Gamma \times \{t \})$ with $r_{t,0}(\Gamma_0) = \Gamma_t$ as {\em sets} so that
\begin{equation} \label{Eqn:Pi' property}
\Pi'(x,t) = r_{t,0}(x),
\end{equation} and we do so.  In fact, this identification is much more than simply a set-wise identification.

\begin{proposition} \label{P:X' cell structure}
The space $X'$ is homeomorphic to a cell-complex.  The identification of the subspace $\Pi'(\Gamma \times \{t\})$ with $\Gamma_t$ is a homeomorphism, where $\Pi'(\Gamma \times \{ t\})$ is given the subspace topology.
\end{proposition}

\begin{proof}
Recall that for each $j = 0,\ldots,m-1$, $\Gamma_{(j+1)/m} = \Delta_{j+1}$ is obtained from $\Gamma_{j/m} = \Delta_j$ by a single combinatorial Stallings fold identifying the edges $e_j,e_j'$ from $\Gamma_{j/m}$.
We will use this to reconstruct $X'$ as a cell complex obtained by gluing together very simple cell complexes.

For each $j = 0,\ldots,m-1$, begin with the product cell complex $\Gamma_{j/m} \times [j/m,(j+1)/m]$.  The (closed) $2$--cells are all thus ``rectangles'' obtained as products of (closed) $1$--cells in $\Gamma_{j/m}$ with the interval $[j/m,(j+1)/m]$.  Now we subdivide the rectangles $R_j = \bar e_j \times [j/m,(j+1)/m]$ and  $R_j' = \bar e_j' \times [j/m,(j+1)/m]$  each into two triangles.  More precisely, taking (orientation preserving) locally isometric characteristic maps for these $2$--cells
\[ \rho_j \colon [0,1] \times [j/m,(j+1)/m] \to R_j = \bar e_j \times [j/m,(j+1)/m] \]
and
\[ \rho_j' \colon [0,1] \times [j/m,(j+1)/m] \to R_j' = \bar e_j' \times [j/m,(j+1)/m], \]
we subdivide $[0,1] \times [j/m,(j+1)/m]$ into two triangles along the diagonal from the point $(0,j/m)$ to $(1,(j+1)/m)$.  Denote the upper triangle by $\tau_+$ and the lower triangle by $\tau_-$.  Now subdivide $\Gamma_{j/m} \times [j/m,(j+1)/m]$ by subdividing each of $\bar e_j \times [j/m,(j+1)/m]$ and  $\bar e_j' \times [j/m,(j+1)/m]$ into the triangles $\rho_j(\tau_+) \cup \rho_j(\tau_-)$ and $\rho_j'(\tau_+) \cup \rho_j'(\tau_-)$, respectively; see Figure \ref{F:subdividing rectangles figure}.

\begin{center}
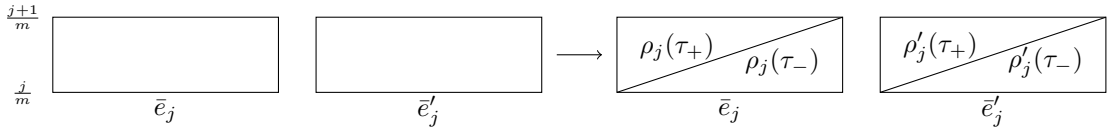
\begin{figure}[htb]
\begin{tikzpicture}
\draw
(0,0) rectangle (3,1)
(3.5,0) rectangle (6.5,1)
(7.5,0) rectangle (10.5,1)
(11,0) rectangle (14,1); 
\node at (1.5,-.25) {$\bar e_j$};
\node at (5,-.25) {$\bar e_j'$};
\node at (9,-.25) {$\bar e_j$};
\node at (12.5,-.25) {$\bar e_j'$};
\node at (-.4,0) {\tiny $\frac j m$};
\node at (-.4,1) {\tiny $\frac{j+1} m$};
\node at (8.3,.6) {$\rho_j(\tau_+)$};
\node at (9.7,.4) {$\rho_j(\tau_-)$};
\node at (11.8,.6) {$\rho_j'(\tau_+)$};
\node at (13.2,.4) {$\rho_j'(\tau_-)$};
\draw (7.5,0) -- (10.5,1);
\draw (11,0) -- (14,1);
\draw [->] (6.7,.5) -- (7.3,.5);
\end{tikzpicture}
\caption{On the left we have the rectangles $R_j,R_j'$ in $\Gamma_{j/m} \times [j/m,(j+1)/m]$ that are to be subdivided, and the right shows their subdivision.}
\label{F:subdividing rectangles figure}
\end{figure}
\end{center}

Now we let $\Pi_j' \colon \Gamma_{j/m} \times [j/m,(j+1)/m] \to X_j'$ be the quotient space obtained by identifying $\rho_j(\tau_+)$ with $\rho_j'(\tau_+)$ via the restriction to $\rho_j(\tau_+)$ of $\rho_j' \circ \rho_j^{-1}$;  see Figure \ref{F:local_folded}.

\begin{figure} [htb]
\begin{center}
\includegraphics[height=4.5cm]{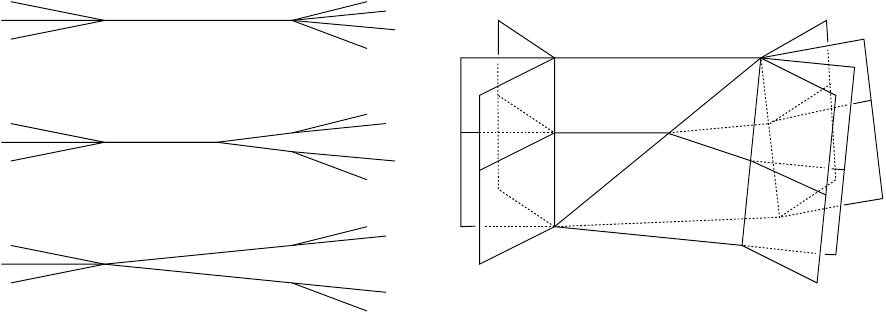} \caption{A local picture of $X_j'$ near the identified triangles.  On the left we have pictures of $\Gamma_{j/m}$, $\Gamma_{(2j+1)/2m}$, and $\Gamma_{(j+1)/m}$ on the bottom, middle, and top, respectively, near the folded edges.  On the right we have the corresponding part of $X_j'$.}
\label{F:local_folded}
\end{center}
\end{figure}

One minor complication is that even restricted to $\tau_+$, the map $\rho_j$ may not be invertible.  This can only occur if $e_j$ is a loop edge.  Since $\Gamma_{j/m} \to \Gamma_{(j+1)/m}$ is a homotopy equivalence, it follows that both $e_j$ and $e_j'$ cannot both be loop edges (if they were, then $e_j' e_j^{-1}$ is a non-null homotopic loop in $\Gamma_{j/m}$ sent to a null homotopic loop in $\Gamma_{(j+1)/m}$).  So, without loss of generality, we assume that $e_j$ is not a loop edge and hence we may invert $\rho_j$ as needed; see Figure \ref{F:local_folded2}.

\begin{figure} [htb]
\begin{center}
\includegraphics[height=3.5cm]{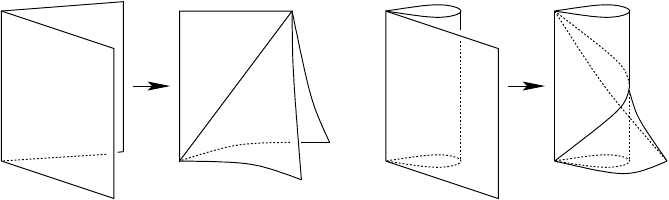} \caption{Two possibilities for $\Pi'_j$ near $R_j\cup R_j'$ depending on whether $e_j'$ is a loop (right) or not a loop (left).}
\label{F:local_folded2}
\end{center}
\end{figure}

We see that by construction, $X_j'$ has a natural cell structure so that $\Pi_j'$ is a cellular map.  Furthermore, we note that for all $t \in [j/m,(j+1)/m]$, the restriction of $\Pi_j'$ to $\Gamma_{j/m} \times \{ t \}$ isometrically identifies the initial segments of $e_j \times \{t \}$ and $e_j' \times \{t \}$ of length $tm-j$.  It is now easy to see that we may identify $\Pi_j'(\Gamma_{j/m} \times \{ t\})$ with $\Gamma_t$, for all $t \in [j/m,(j+1)/m]$, {\em as topological spaces}.  
Using this identification we construct a quotient space $X''$ from the disjoint union $X_0' \sqcup \ldots \sqcup X_{m-1}'$ by gluing the ``top'' $\Gamma_{(j+1)/m} \subset X_j'$ to the ``bottom'' $\Gamma_{(j+1)/m} \subset X_{j+1}'$, for each $j = 0,\ldots,m-1$.  The space $X''$ has a natural cell structure so that each $X_j'$ includes into $X''$ by a cellular embedding.  

The maps
\[ \xymatrix{ 
 {\Gamma_0 \times [j/m,(j+1)/m]} \ar[rr]^(.47){r_{j/m,0} \times id} & &  {\Gamma_{j/m} \times [j/m,(j+1)/m]} \ar[rr]^(.7){\Pi_j'} & & X_j' }\]
piece together to induce a map $\Pi''\colon \Gamma \times [0,1] \to X''$.  For each $j = 0,\ldots,m-1$ and $j/m \leq t < (j+1)/m$, the restriction of $\Pi''$ to each $\Gamma \times \{t \}$ identifies $(x,t)$ to $(x',t)$ if and only if $r_{t,j/m}(r_{j/m,0}(x)) = r_{t,j/m}(r_{j/m,0}(x'))$.  By (\ref{Eqn:semiflow for folding maps}), $r_{t,j/m} \circ r_{j/m,0} = r_{t,0}$, and it follows that the fibers of $\Pi'$ and $\Pi''$ are the same.  Since $\Gamma \times [0,1]$ and $X''$ are compact Hausdorff spaces, $X''$ is homeomorphic to the quotient space $X'$ (and identifying $X' = X''$, we have $\Pi' = \Pi''$).
Furthermore, the homeomorphic identification of $\Pi'(\Gamma \times \{t \}) = \Pi''(\Gamma \times \{t \})$ with $\Gamma_t$ proves the second statement. 
\end{proof}

The spaces $X_j'$ in the construction of $X'' = X'$ are useful in analyzing the local pictures of $X'$, and we will refer to these as {\em building blocks}.

\begin{defn}[Folded mapping torus] \label{D:folded mapping torus}
The {\em folded mapping torus} $X = X_\f$ is the quotient
\[ X' \to X_\f \]
obtained by gluing $\Gamma_1 \subset X'$ to $\Gamma_0 \subset X'$ using the identifications $\Gamma_0 = \Gamma = \Gamma_1$.
Subdividing $\Gamma_1$, the gluing becomes cellular, and we can induce on $X_\f$ a cell structure from the one constructed in the proof of Proposition \ref{P:X' cell structure} for which the $2$--cells are either {\em rectangles} or {\em triangles}.  We call this the  {\em initial cell structure} on $X_\f$ (a more useful cell structure will be defined in \S\ref{sect:trapezoids}).
\end{defn}

The folded mapping torus $X_\f$ is closely related to the mapping torus of $\f$, which is the space
\[ Z_\f = \Gamma \times [0,1]/\sim\]
where $(x,1) \sim (\f(x),0)$.
\begin{proposition} \label{P:Pi' descends}
The map $\Pi'\colon \Gamma\times[0,1]\to X'$ descends to a continuous map $\Pi \colon Z_\f \to X_\f$.
\end{proposition}
\begin{proof}
Since $(x,1)$ is identified with $(\f(x),0)$ in $Z_\f$ (and these are the only identifications), we need only verify that $\Pi'(x,1)$ is identified with $\Pi'(\f(x),0)$ in $X_\f$.  Using the identification $\Gamma_0 = \Gamma = \Gamma_1$ together with (\ref{Eqn:folding composition gives f}) and (\ref{Eqn:Pi' property}) we have
\[ \Pi'(x,1) = r_{1,0}(x) = \f(x) \in \Gamma_1 = \Gamma \]
and
\[ \Pi'(\f(x),0) = r_{0,0}(\f(x)) = \f(x) \in \Gamma_0 = \Gamma.\]
Since $\f(x)$ in $\Gamma_1$ is identified with $\f(x)$ in $\Gamma_0$, we are done.
\end{proof}

In fact, it is not difficult to prove that $\Pi$ is a homotopy equivalence, but this will follow naturally from the construction of certain semiflows in the next section; see Theorem \ref{T:homotopy type}.

\subsection{Semiflows} \label{S:semiflow construction}

Here we construct and analyze (local) semiflows---that is, (local) actions of the additive semigroup $\R_+$ of positive real numbers---on the spaces constructed in \S\ref{sect: folded mapping torus}.
First, let $\hat \flow'$ be the local flow on $\Gamma \times [0,1]$ defined by $\hat \flow'_s(x,t) = (x,t+s)$, wherever this makes sense.
\begin{proposition}
For all $0 \leq s,t \leq s+t \leq 1$ we have $\Pi'(\hat \flow'_s(x,t)) = r_{s+t,t}(\Pi'(x,t))$.
\end{proposition}
\begin{proof}  We simply appeal to (\ref{Eqn:folding composition gives f}) and (\ref{Eqn:Pi' property}) then calculate:
\[ \Pi'(\hat \flow'_s(x,t)) = \Pi'(x,t+s) = r_{t+s,0}(x) = r_{t+s,t}(r_{t,0}(x)) = r_{t+s,t}(\Pi'(x,t)).\qedhere\]
\end{proof}
From this it follows that the maps $r_{t',t}$ define a {\em local semiflow} $\flow'$ on $X'$ given by $\flow'_s(\Pi'(x,t)) = r_{t+s,t}(\Pi'(x,t))$, wherever this is defined.   Furthermore, $\Pi'$ semi-conjugates $\hat \flow'$ to $\flow'$:
\[ \Pi' \circ \hat \flow_s = \flow'_s \circ \Pi'.\]

We note that flowlines of $\hat \flow'$ only terminate at $\Gamma \times \{1 \}$, and so $\hat \flow'$ can be extended to a semiflow $\hat \flow$ on $Z_\f$ (it is only guaranteed to be a semiflow since $\f$ is not in a homeomorphism in general).  In fact, $\hat \flow$ is given by 
\[ \hat \flow_s(x,t) = \left(\f^{[s+t]}(x),\; s+t-[s+t]\right) \]
for all $(x,t) \in \Gamma \times [0,1) \subset Z_\f$, where $[s+t]$ is the greatest integer less than or equal to $s+t$. Since the surjection $\Pi'$ semi-conjugates $\hat \flow'$ to $\flow'$, it follows that $\Pi$ semi-conjugates $\hat \flow$ to a semiflow $\flow$ on $X_\f$.  On the graph $\Gamma = \Gamma_0 \subset X' \to X_\f$, this semiflow is given by
\[ \flow_s(x) = r_{s-[s],0} \circ f^{[s]}(x) \in \Gamma_{s-[s]}.\]
In particular, $\flow_1$ restricted to $\Gamma_0$ is precisely the map $\f$.  Consequently, we call $\flow$ the {\em suspension semiflow on $X_\f$.}

As a first application, we now prove

\begin{theorem} \label{T:homotopy type}
The map $\Pi \colon  Z_\f \to X_\f$ is a homotopy equivalence.
\end{theorem}
\begin{proof}
Construct a map $F' \colon X' \to \Gamma \times [0,1]$ defined by $F'(\Pi'(x,t)) = (r_{1,t}(\Pi'(x,t)),t)$.  By (\ref{Eqn:folding composition gives f}) and (\ref{Eqn:Pi' property}) we have
\[ F'(\Pi'(x,t)) = F'(r_{t,0}(x)) = (r_{1,t}(r_{t,0}(x)),t)) = (\f(x),t).\]
It follows that $F'$ descends to a map $F \colon X_\f \to Z_\f$.  Moreover, this equation implies $F \circ \Pi = \hat \flow_1$.

On the other hand, for $t \in [0,1)$ and $x \in \Gamma$, we have
\[ \Pi(F(\Pi(x,t))) = \Pi(\hat \flow_1(x,t)) = \flow_1(\Pi(x,t)).\]
Consequently, $\Pi \circ F = \flow_1$.  Therefore, $F$ is a homotopy inverse to $\Pi$ with the semiflows $\hat \flow$ and $\flow$ providing the homotopies $F \circ \Pi \simeq id_{Z_\f}$ and $\Pi \circ F \simeq id_{X_\f}$, respectively.
\end{proof}

Recall that $G = G_\fee = \langle F_N, r \mid r^{-1} w r = \Phi(w), w\in F_N
\rangle$, where $\Phi\in \Aut(F_N)$ is a representative of the outer automorphism class $\fee$.   The projection $\Gamma \times [0,1] \to [0,1]$ descends to
a surjection $\fib_0'\colon Z_\f \to \mathbb{S}^1 = \R/\Z$.  Similarly, the projection $X' \to [0,1]$, given by $\Pi'(x,t) \mapsto t$, descends to a map
\[ \fib_0\colon X_\f \to \mathbb{S}^1 = \R/\Z.\]
For every $t \in \R/\Z$, the fiber $\fib_0^{-1}(t)$ is precisely $\Gamma_t$ (where we take a representative $t \in [0,1]$ for the point of $\mathbb{S}^1 = \R/\Z$). 

An application of van Kampen's Theorem implies that the isomorphism $F_N \to \pi_1(\Gamma)$ from the marking extends to an isomorphism $\Psi\colon G \to \pi_1(Z_\f)$ for which $(\fib_0')_*(\Psi(r)) = 1 \in \Z = \pi_1(\mathbb{S}^1)$.  Thus, $u_0\colon  G \to \Z$ defined by $u_0(r^n w) = n$ for all $w \in F_N$ satisfies $u_0 = (\fib_0')_* \circ \Psi$.  

\begin{conv} \label{conv:G identified}
For the remainder of this paper, we use the isomorphism
\[ \Pi_* \circ \Psi\colon G \to \pi_1(X_\f)\]
 to identify $G$ with $\pi_1(X_\f)$.  With this convention, and the fact that $\fib_0' = \fib_0 \circ \Pi$ we see that $(\fib_0)_* = u_0$. 
 \end{conv}
 
Since $Z_\f$ is a $K(G,1)$ for $G=G_\fee$ (see e.g.~\cite[Theorem 1.B 11]{Hatcher}), as a consequence of Theorem~\ref{T:homotopy type} we have:
\begin{corollary} \label{C:K(G,1)}
The folded mapping torus $X_\f$ is a $K(G,1)$ space.
\end{corollary}

\begin{example}\label{Ex:X_f}

Returning to our running example $\f\colon \Gamma \to \Gamma$, Figure \ref{F:folded_mapping_torus_ex} shows the folded mapping torus $X_\f$ in this case.  We have represented it as four large rectangles with a subdivided cell structure, and with identifications as indicated.  See the caption of the figure for detailed description of the identifications.

\begin{center}
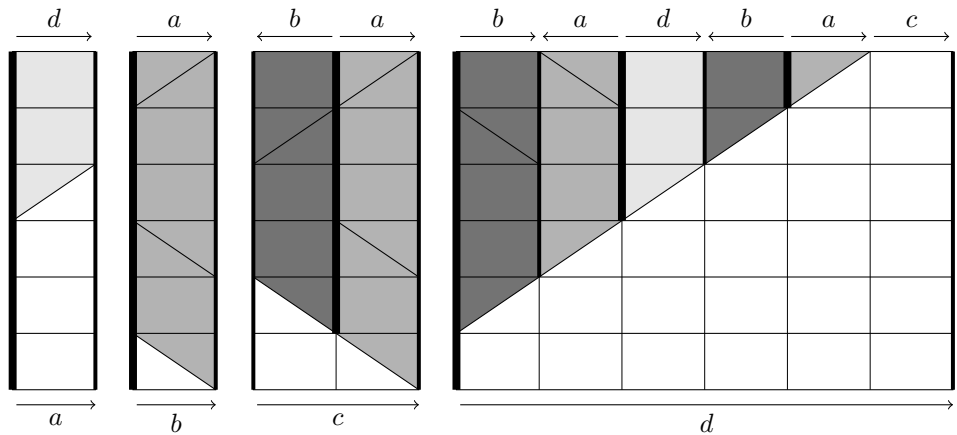
\begin{figure}[htb]
\begin{tikzpicture}
%

\pgfmathsetmacro\hstep{1.1}   
\pgfmathsetmacro\vstep{.75} 
\pgfmathsetmacro\hsp{.5}    
\pgfmathsetmacro\arsp{.05}  
\pgfmathsetmacro\varr{.2}    

\pgfmathsetmacro\vzero{0}                  
\pgfmathsetmacro\vone{\vzero   + 1*\vstep} 
\pgfmathsetmacro\vtwo{\vzero   + 2*\vstep} 
\pgfmathsetmacro\vthree{\vzero + 3*\vstep} 
\pgfmathsetmacro\vfour{\vzero  + 4*\vstep} 
\pgfmathsetmacro\vfive{\vzero  + 5*\vstep} 
\pgfmathsetmacro\vsix{\vzero   + 6*\vstep} 

\pgfmathsetmacro\azero{0}                  
\pgfmathsetmacro\aone{\azero + \hstep}     
\pgfmathsetmacro\bzero{\aone + \hsp}       
\pgfmathsetmacro\bone{\bzero + \hstep}     
\pgfmathsetmacro\czero{\bone + \hsp}       
\pgfmathsetmacro\cone{\czero + \hstep}     
\pgfmathsetmacro\ctwo{\czero + 2*\hstep}   
\pgfmathsetmacro\dzero{\ctwo  + \hsp}      
\pgfmathsetmacro\done{\dzero  + 1*\hstep}  
\pgfmathsetmacro\dtwo{\dzero  + 2*\hstep}  
\pgfmathsetmacro\dthree{\dzero+ 3*\hstep}  
\pgfmathsetmacro\dfour{\dzero + 4*\hstep}  
\pgfmathsetmacro\dfive{\dzero + 5*\hstep}  
\pgfmathsetmacro\dsix{\dzero  + 6*\hstep}  
 
\fill[black!30!white] 
(\bone,\vzero) -- (\bone,\vsix) -- (\bzero, \vsix) -- (\bzero,\vone)
(\ctwo,\vzero) -- (\ctwo,\vsix) -- (\cone, \vsix) -- (\cone,\vone)
(\done,\vtwo) -- (\done,\vsix) -- (\dtwo, \vsix) -- (\dtwo,\vthree)
(\dfour,\vfive) -- (\dfour,\vsix) -- (\dfive, \vsix);

\fill[black!55!white]
(\czero,\vtwo) -- (\czero,\vsix) -- (\cone,\vsix) -- (\cone,\vone)
(\done,\vtwo) -- (\done,\vsix) -- (\dzero,\vsix) -- (\dzero,\vone)
(\dthree,\vfour) -- (\dthree,\vsix) -- (\dfour,\vsix) -- (\dfour,\vfive);

\fill[black!10!white]
(\azero,\vthree) -- (\azero,\vsix) -- (\aone,\vsix) -- (\aone,\vfour)
(\dtwo,\vthree) -- (\dtwo,\vsix) -- (\dthree,\vsix) -- (\dthree,\vfour);

\draw (\azero,\vzero) rectangle (\aone,\vsix)
(\bzero,\vzero) rectangle (\bone,\vsix)
(\czero,\vzero) rectangle (\ctwo,\vsix)
(\dzero,\vzero) rectangle (\dsix,\vsix);

\draw [line width = 3]
(\azero,\vzero) -- (\azero,\vsix)
(\bzero,\vzero) -- (\bzero,\vsix)
(\dzero,\vzero) -- (\dzero,\vsix);
\draw [line width = 3] 
(\cone,\vone) -- (\cone,\vsix)
(\dtwo,\vthree) -- (\dtwo,\vsix)
(\dfour,\vfive) -- (\dfour,\vsix);
\draw [line width = 1.5]
(\aone,\vzero) -- (\aone,\vsix)
(\bone,\vzero) -- (\bone,\vsix)
(\czero,\vzero) -- (\czero,\vsix)
(\ctwo,\vzero) -- (\ctwo,\vsix)
(\dsix,\vzero) -- (\dsix,\vsix);
\draw [line width = 1.5]
(\done,\vtwo) -- (\done,\vsix)
(\dthree,\vfour) -- (\dthree,\vsix);
\draw [ultra thin] 
(\cone,\vzero) -- (\cone,\vone)
(\done,\vzero) -- (\done,\vtwo)
(\dtwo,\vzero) -- (\dtwo,\vthree)
(\dthree,\vzero) -- (\dthree,\vfour)
(\dfour,\vzero) -- (\dfour,\vfive)
(\dfive,\vzero) -- (\dfive,\vsix);

\draw (\azero,\vone)--(\aone,\vone) (\bzero,\vone)--(\bone,\vone) (\czero,\vone)--(\ctwo,\vone) (\dzero,\vone)--(\dsix,\vone)
(\azero,\vtwo)--(\aone,\vtwo) (\bzero,\vtwo)--(\bone,\vtwo) (\czero,\vtwo)--(\ctwo,\vtwo) (\dzero,\vtwo)--(\dsix,\vtwo)
(\azero,\vthree)--(\aone,\vthree) (\bzero,\vthree)--(\bone,\vthree) (\czero,\vthree)--(\ctwo,\vthree) (\dzero,\vthree)--(\dsix,\vthree)
(\azero,\vfour)--(\aone,\vfour) (\bzero,\vfour)--(\bone,\vfour) (\czero,\vfour)--(\ctwo,\vfour) (\dzero,\vfour)--(\dsix,\vfour)
(\azero,\vfive)--(\aone,\vfive) (\bzero,\vfive)--(\bone,\vfive) (\czero,\vfive)--(\ctwo,\vfive) (\dzero,\vfive)--(\dsix,\vfive);

\draw [->] (\azero+\arsp,\vzero -\varr)--node[below]{$a$}(\aone, \vzero -\varr); 
\draw [->] (\bzero+\arsp,\vzero -\varr)--node[below]{$b$}(\bone, \vzero -\varr); 
\draw [->] (\czero+\arsp,\vzero -\varr)--node[below]{$c$}(\ctwo, \vzero -\varr); 
\draw [->] (\dzero+\arsp,\vzero -\varr)--node[below]{$d$}(\dsix, \vzero -\varr); 
\draw [->] (\azero +\arsp,\vsix + \varr) -- node[above]{$d$} (\aone  -\arsp,\vsix + \varr);
\draw [->] (\bzero +\arsp,\vsix + \varr) -- node[above]{$a$} (\bone  -\arsp,\vsix + \varr);
\draw [<-] (\czero +\arsp,\vsix + \varr) -- node[above]{$b$} (\cone  -\arsp,\vsix + \varr);
\draw [->] (\cone  +\arsp,\vsix + \varr)  -- node[above]{$a$}(\ctwo  -\arsp,\vsix + \varr);
\draw [->] (\dzero +\arsp,\vsix + \varr) -- node[above]{$b$} (\done  -\arsp,\vsix + \varr);
\draw [<-] (\done  +\arsp,\vsix + \varr)  -- node[above]{$a$}(\dtwo  -\arsp,\vsix + \varr);
\draw [->] (\dtwo  +\arsp,\vsix + \varr)  -- node[above]{$d$}(\dthree-\arsp,\vsix + \varr);
\draw [<-] (\dthree+\arsp,\vsix + \varr) -- node[above]{$b$} (\dfour -\arsp,\vsix + \varr);
\draw [->] (\dfour +\arsp,\vsix + \varr) -- node[above]{$a$} (\dfive -\arsp,\vsix + \varr);
\draw [->] (\dfive +\arsp,\vsix + \varr) -- node[above]{$c$} (\dsix  -\arsp,\vsix + \varr);

\draw (\azero, \vthree) -- (\aone, \vfour); 
\draw (\bone, \vzero) -- (\bzero, \vone); 
\draw (\bone, \vtwo) -- (\bzero, \vthree); 
\draw (\bzero, \vfive) -- (\bone, \vsix); 
\draw (\ctwo, \vzero) -- (\czero, \vtwo); 
\draw (\ctwo, \vtwo) -- (\cone, \vthree); 
\draw (\czero, \vfour) -- (\ctwo, \vsix); 
\draw (\dzero,\vone) -- (\dfive,\vsix); 
\draw (\done,\vfour) -- (\dzero,\vfive); 
\draw (\dtwo,\vfive) -- (\done,\vsix); 
\end{tikzpicture}

\caption{The folded mapping torus $X_\f$ for the example $\f\colon \Gamma \to \Gamma$.  The identifications on the sides of the large rectangles are indicated by thick and medium-thick lines, and these are by isometries.  This identification gives the product $\Gamma \times [0,1]$.  The tops of the rectangles are identified to the bottoms as indicated by the labeled arrows, and these identifications are isometries for $a$ and $b$, and local homotheties by $2$ for $c$ and by $6$ for $d$.
The result of these identifications (together with the side identifications) gives the mapping torus $Z_\f$.  Finally, the $2$--cells are identified as indicated by the shading and the patterns of $2$--cells (white cells are not identified to any others) to produce $X_\f$.  We have also indicated which $1$--cells in the interiors of the rectangles are identified with those coming from the boundaries of the rectangles via thick and medium-thick lines.   We note that performing the side identifications and the $2$--cell identifications only (which can be done by isometries) gives the complex $X'$.}
\label{F:folded_mapping_torus_ex}
\end{figure}
\end{center}
\end{example}

\subsection{Horizontal, vertical, and edge degree}

\begin{defn} \label{D:degree}
We intuitively define the {\em degree} $d(e)$ of a $1$--cell $e$ of the initial cell structure on $X_\f$ to be
the number of $2$--cells attached to $e$, counted with multiplicity.
More precisely, we note that the attaching maps for the $2$--cells are
locally injective, and for every $1$--cell $e$, we pick a point $x$ in the interior of $e$, and $d(e)$ is the total number of points in the preimage, over all
attaching maps (local injectivity makes this definition independent of $x$).    Equivalently, $d(e)$ is the minimum number of components of $U-e$, over all arbitrarily small neighborhoods $U$ of $x$.   For example, any point in an edge of degree $2$ has a neighborhood homeomorphic to $\R^2$.
\end{defn}

\begin{defn}  The {\em horizontal foliation} of $X_\f$ is the decomposition of $X_\f$ into fibers $\Gamma_t$, for $t \in \mathbb{S}^1$.  We refer to any subset of a single fiber $\Gamma_t$ as being horizontal.  We refer to any arc of a flow line as a {\em vertical} arc.  More generally, any set of points that all flow forward to a common point will be called vertical.  
\end{defn}

The complement of the closure of the $1$--cells of degree different than $2$ is a topological surface.  We note that the restriction of the horizontal foliation to this surface is a foliation in the traditional sense.  Likewise, the restriction of the semiflow to this surface is a local flow.

The $1$--cells fall into three types: vertical, horizontal and skew.  
\begin{defn}
The {\em vertical} $1$--cells are $1$--cells contained in semiflow lines, while the {\em horizontal} $1$--cells are contained in fibers of $\fib_0$.   The skew $1$--cells are the $m$ remaining $1$--cells which occur as the (nonhorizontal/nonvertical) side of some triangle. 
\end{defn}

Note that the horizontal $1$--cells all have degree $2$ and the skew $1$--cells all have degree $3$.  On the other hand, the vertical $1$--cells have more general degrees.

\begin{lemma} \label{L:val=deg}
If $e$ is a vertical or skew $1$--cell and $v \in V\Gamma_t$ is the vertex defined by $v  = \Gamma_t \cap e$, then $d(e) = \val(v)$, where $\val(v)$ is the valence of $v$ in $\Gamma_t$.  In particular, $d(e) \geq 2$ for every vertical $1$--cell.
\end{lemma}
\begin{proof}
First suppose $e$ is a vertical $1$--cell, and observe that $t \not \in \{ j/m \}_{j=0}^m$ because $e = \Pi(v_0 \times (j/m,(j+1)/m))$ for some vertex $v_0 \in V\Gamma_{j/m}$ for some $j$, and $v = \Gamma_t \cap e \neq \emptyset$.  The fact that $d(e) = \val(v)$ follows from the fact that one has a basis of neighborhoods of $v$ in $X$ of the form  $U = W \times (t-\epsilon,t+\epsilon)$ for small $\epsilon > 0$ and $W$ a neighborhood of $v$ in $\Gamma_t$.  The minimal number of components of $W-v$ is equal to the minimal number of components of $U-e$ over all sufficiently small neighborhoods $W$ of $v$ in $\Gamma_t$, and this is equal to $\val(v)$, so $d(e) \leq \val(v)$.  Any sufficiently small neighborhood about $v$ in $X$ is contained in some basis open set $U$, and from here one easily deduces $d(e) = \val(v)$.

Since the intersection of a skew $1$--cell $e$ with $\Gamma_t$ corresponds to the ``new'' vertex of a fold, and this vertex always has $\val(v) = 3$ (and as we've already noted, $d(e) = 3$), the first statement for $e$ follows. 

For the second statement, we note that our folding lines never produce graphs with valence $1$ vertices, by Remark~\ref{R:deg1}.
\end{proof}

\subsection{The trapezoid structure.} \label{sect:trapezoids}

To carry out the constructions in \S\ref{sect:allowable perturbations}, it is desirable to modify the cell structure so that there are no horizontal $1$--cells while maintaining a simple structure for the $2$--cells.  Simply erasing the horizontal $1$--cells is insufficient, since the complement of the (closure of the) vertical and skew $1$--cells need not be a union of disks (although it is a surface).  We therefore proceed as follows.

For every edge $e$ of $\Gamma_0 \subset X_\f$, we consider the image under the semiflow for time up to $1$:
\[ R_e = \bigcup_{0 \leq s \leq 1} \flow_s(\bar e).\]
We call this set the {\em strip above $e$}.  The complex $X$ is a union of the strips above edges $e$ of $\Gamma_0$.  Some strips contain a skew edge and others do not.  The example in Figure \ref{F:folded_mapping_torus_ex} has $10$ strips, $9$ of which contain a skew edge, and one does not.

For each strip that does not contain a skew $1$--cell, we subdivide the first rectangle of the initial cell structure
\[ \bar e \times [0,1/m] = \bigcup_{0 \leq s \leq 1/m} \flow_s(\bar e)\]
into two triangles by adding one of the diagonals (it does not matter which diagonal we add).   The added diagonal becomes a $1$--cell which we also call a skew $1$--cell.  The new skew $1$--cells are distinguished from the previous ones by their degree: the new skew $1$--cells all have degree $2$, while the old ones have degree $3$.

For any skew $1$--cell $e$, consider the set of points $T_e \subset X - e$ that flow into $e$ before meeting any other skew $1$--cell.   That is $y \in T_e$ if and only if there exists $r > 0$ so that $\flow_r(y) \in e$, and for all $0 \leq s < r$, $\flow_s(y)$ is not contained in any skew $1$--cell.  We now describe $T_e$ in more detail.

Suppose first that $d(e) = 2$.  Below $e$ is a unique triangle $2$--cell, and call it $U_1$.  The triangle $U_1$ lies below $e$ and above a horizontal $1$--cell $\delta_1$.   The $1$--cell $\delta_1$ is contained in the top arc of a $2$--cell $\sigma_2$ (it may be a proper subarc of the top of this $2$--cell only if $\delta_1 \subset \Gamma_0 = \Gamma_1$).  Let $U_2$ denote the subset of $\sigma_2$ directly below $\delta_1$ (that is, the subset of this $2$--cell that flows into $\delta_1$).   If $\sigma_2$ is a triangle, then the bottom of $U_2$ is an arc of a skew $1$--cell, and hence $T_e = U_1 \cup \delta_1 \cup U_2$.  If $\sigma_2$ is a rectangle, then $U_2$ lies above an arc $\delta_2$ of a horizontal $1$--cell $e_2$, and $\delta_2$ is a subarc of a $2$--cell $\sigma_3$.  Let $U_3$ denote the subset of $\sigma_3$ below $\delta_2$.  If $\sigma_3$ is a triangle $2$--cell, then
\[ T_e = U_1 \cup \delta_1 \cup U_2 \cup \delta_2 \cup U_3.\]
Otherwise, $\sigma_3$ is a rectangle, and we repeat again.  In this way, we eventually write $T_e$ as
\[ T_e = U_1 \cup \delta_1 \cup U_2 \cup \delta_2 \cup \ldots \cup \delta_k \cup U_k\]
for some $k \geq 1$.
See Figure \ref{F:trapezoids}.  This procedure terminates because every strip contains a skew $1$--cell, hence a triangle (which explains why the new skew $1$--cells were added).

If $d(e) = 3$, then below $e$ there are exactly two triangle $2$--cells.  The above description can be carried out for each of the two triangles thus producing two trapezoids $T_e = T_e' \sqcup T_e''$, each of which can be described as
\[ U_1 \cup \delta_1 \cup U_2 \cup \delta_2  \cup \ldots \cup \delta_k \cup U_k\]
for some $k \geq 1$. 

\begin{defn}
A {\em trapezoid} for $X$ is a disk $T_e$ when $d(e) = 2$ or one of $T_e'$ or $T_e''$ when $d(e) = 3$.
\end{defn}

\begin{figure} [htb]
\begin{center}
\includegraphics[height=4cm]{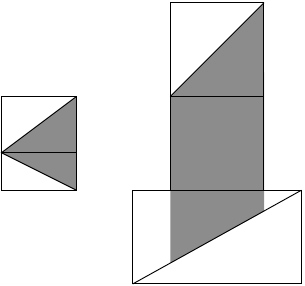} \caption{Two possible trapezoids.}
\label{F:trapezoids}
\end{center}
\end{figure}

Let $\mathcal T$ denote the set of all trapezoids.  The closure of any $T \in \mathcal T$ is a union
\[ \bar T = T \cup e_+(T) \cup \ell_+(T) \cup \ell_-(T) \cup e_-(T) \]
where $e_+(T) = \bar e$ and $e$ is the initial skew $1$--cell defining the trapezoid, $e_-(T)$ is the arc of a skew $1$--cell below $T$, and $\ell_\pm(T)$ are the arcs of flow lines bounding $T$ on either side (one of which may degenerate to a single point).  Although $T$ is an open disk, $\bar T$ need not be a closed disk (for example, it can happen that $e_-(T) \subset e_+(T)$; see Figure \ref{F:trapezoid_example}).

Because every flow line eventually intersects some skew $1$--cell (again because of the addition of the new skew $1$--cells), it follows that $X$ is the union
\[ X = \bigcup_{T \in \mathcal T} \bar T.\]
Since $T \cap T' = \emptyset$ for all $T \neq T' \in \mathcal T$ (since the first skew $1$--cell a point meets under $\psi$ is unique), we can define a new cell structure for which $\mathcal T$ is the set of $2$--cells.

\begin{defn}[Trapezoid cell structure] \label{def:trapezoidcell} The {\em trapezoid cell structure} is the cell structure for which $\mathcal T$ is the set of $2$--cells, and the vertices are precisely the union of the (three or) four corners of $\bar T$ over all $T \in \mathcal T$.  We give the $1$--skeleton $X_\f^{(1)}$ a metric structure so that the restriction of $\fib_0$ to each edge is a local isometry to $\mathbb{S}^1$, and an orientation by insisting that these maps be orientation preserving.
\end{defn}

\begin{conv}
In all of what follows, we will assume that {\em the} cell structure on $X$ is the trapezoid cell structure, unless otherwise stated.  We refer to the initial cell structure {\em with the added diagonals} as the {\em old cell structure}.    We also refer to the $1$--cells of the trapezoid cell structure as either vertical or skew.
\end{conv}

\begin{remark}  \label{R:old v trap}
We note that every skew $1$--cell for the trapezoid cell structure is contained in a skew $1$--cell of the old cell structure, but that this is not true for the vertical $1$--cells.  However, any vertical $1$--cell of the trapezoid structure eventually flows into a vertical $1$--cell of the old cell structure.  Since the union of the vertical $1$--cells of the old structure is invariant under $\flow$ (because $\f(V\Gamma) \subset V\Gamma$), so is the union of the vertical $1$--cells of the trapezoid structure.
\end{remark}

\begin{example}\label{Ex:trapezoid_structure}
The trapezoid cell structure for our running example $\f\colon \Gamma \to \Gamma$ is shown in Figure \ref{F:trapezoid_example}.  There are fourteen $0$--cells, twenty-seven $1$--cells and thirteen $2$--cells.

\begin{center}
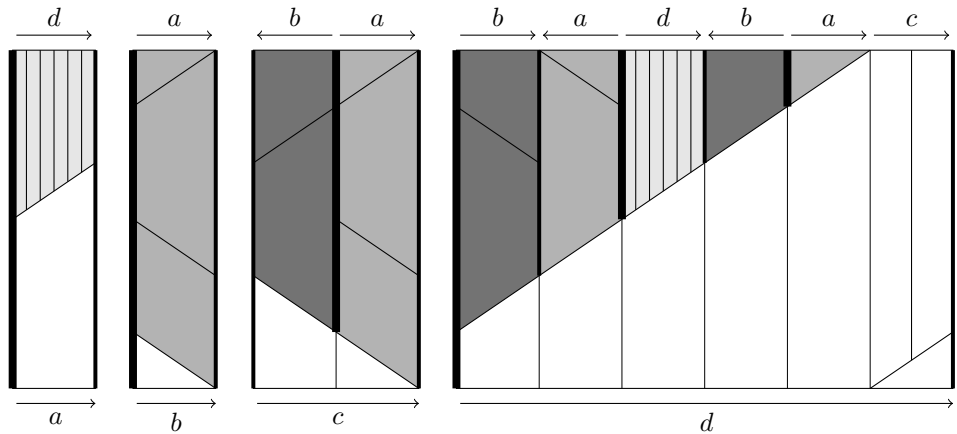
\begin{figure}[htb]
\begin{tikzpicture}
%

\pgfmathsetmacro\hstep{1.1}   
\pgfmathsetmacro\vstep{.75} 
\pgfmathsetmacro\hsp{.5}    
\pgfmathsetmacro\arsp{.05}  
\pgfmathsetmacro\varr{.2}    
\pgfmathsetmacro\subc{1.0/2.0} 
\pgfmathsetmacro\subd{1.0/6.0} 

\pgfmathsetmacro\vzero{0}                  
\pgfmathsetmacro\vone{\vzero   + 1*\vstep} 
\pgfmathsetmacro\vtwo{\vzero   + 2*\vstep} 
\pgfmathsetmacro\vthree{\vzero + 3*\vstep} 
\pgfmathsetmacro\vfour{\vzero  + 4*\vstep} 
\pgfmathsetmacro\vfive{\vzero  + 5*\vstep} 
\pgfmathsetmacro\vsix{\vzero   + 6*\vstep} 

\pgfmathsetmacro\azero{0}                  
\pgfmathsetmacro\aone{\azero + \hstep}     
\pgfmathsetmacro\bzero{\aone + \hsp}       
\pgfmathsetmacro\bone{\bzero + \hstep}     
\pgfmathsetmacro\czero{\bone + \hsp}       
\pgfmathsetmacro\cone{\czero + \hstep}     
\pgfmathsetmacro\ctwo{\czero + 2*\hstep}   
\pgfmathsetmacro\dzero{\ctwo  + \hsp}      
\pgfmathsetmacro\done{\dzero  + 1*\hstep}  
\pgfmathsetmacro\dtwo{\dzero  + 2*\hstep}  
\pgfmathsetmacro\dthree{\dzero+ 3*\hstep}  
\pgfmathsetmacro\dfour{\dzero + 4*\hstep}  
\pgfmathsetmacro\dfive{\dzero + 5*\hstep}  
\pgfmathsetmacro\dsix{\dzero  + 6*\hstep}  
 
\fill[black!30!white] 
(\bone,\vzero) -- (\bone,\vsix) -- (\bzero, \vsix) -- (\bzero,\vone)
(\ctwo,\vzero) -- (\ctwo,\vsix) -- (\cone, \vsix) -- (\cone,\vone)
(\done,\vtwo) -- (\done,\vsix) -- (\dtwo, \vsix) -- (\dtwo,\vthree)
(\dfour,\vfive) -- (\dfour,\vsix) -- (\dfive, \vsix);

\fill[black!55!white]
(\czero,\vtwo) -- (\czero,\vsix) -- (\cone,\vsix) -- (\cone,\vone)
(\done,\vtwo) -- (\done,\vsix) -- (\dzero,\vsix) -- (\dzero,\vone)
(\dthree,\vfour) -- (\dthree,\vsix) -- (\dfour,\vsix) -- (\dfour,\vfive);

\fill[black!10!white]
(\azero,\vthree) -- (\azero,\vsix) -- (\aone,\vsix) -- (\aone,\vfour)
(\dtwo,\vthree) -- (\dtwo,\vsix) -- (\dthree,\vsix) -- (\dthree,\vfour);

\draw (\azero,\vzero) rectangle (\aone,\vsix)
(\bzero,\vzero) rectangle (\bone,\vsix)
(\czero,\vzero) rectangle (\ctwo,\vsix)
(\dzero,\vzero) rectangle (\dsix,\vsix);

\draw [line width = 3]
(\azero,\vzero) -- (\azero,\vsix)
(\bzero,\vzero) -- (\bzero,\vsix)
(\dzero,\vzero) -- (\dzero,\vsix);
\draw [line width = 3] 
(\cone,\vone) -- (\cone,\vsix)
(\dtwo,\vthree) -- (\dtwo,\vsix)
(\dfour,\vfive) -- (\dfour,\vsix);
\draw [line width = 1.5]
(\aone,\vzero) -- (\aone,\vsix)
(\bone,\vzero) -- (\bone,\vsix)
(\czero,\vzero) -- (\czero,\vsix)
(\ctwo,\vzero) -- (\ctwo,\vsix)
(\dsix,\vzero) -- (\dsix,\vsix);
\draw [line width = 1.5]
(\done,\vtwo) -- (\done,\vsix)
(\dthree,\vfour) -- (\dthree,\vsix);
\draw [ultra thin] 
(\cone,\vzero) -- (\cone,\vone)
(\done,\vzero) -- (\done,\vtwo)
(\dtwo,\vzero) -- (\dtwo,\vthree)
(\dthree,\vzero) -- (\dthree,\vfour)
(\dfour,\vzero) -- (\dfour,\vfive)
(\dfive,\vzero) -- (\dfive,\vsix);

\draw 
(\azero+1*\subd*\hstep,\vthree+1*\subd*\vstep)--(\azero+1*\subd*\hstep,\vsix)
(\azero+2*\subd*\hstep,\vthree+2*\subd*\vstep)--(\azero+2*\subd*\hstep,\vsix)
(\azero+3*\subd*\hstep,\vthree+3*\subd*\vstep)--(\azero+3*\subd*\hstep,\vsix)
(\azero+4*\subd*\hstep,\vthree+4*\subd*\vstep)--(\azero+4*\subd*\hstep,\vsix)
(\azero+5*\subd*\hstep,\vthree+5*\subd*\vstep)--(\azero+5*\subd*\hstep,\vsix);

\draw 
(\dtwo+1*\subd*\hstep,\vthree+1*\subd*\vstep)--(\dtwo+1*\subd*\hstep,\vsix)
(\dtwo+2*\subd*\hstep,\vthree+2*\subd*\vstep)--(\dtwo+2*\subd*\hstep,\vsix)
(\dtwo+3*\subd*\hstep,\vthree+3*\subd*\vstep)--(\dtwo+3*\subd*\hstep,\vsix)
(\dtwo+4*\subd*\hstep,\vthree+4*\subd*\vstep)--(\dtwo+4*\subd*\hstep,\vsix)
(\dtwo+5*\subd*\hstep,\vthree+5*\subd*\vstep)--(\dtwo+5*\subd*\hstep,\vsix);
\draw
(\dfive+\subc*\hstep,\vzero+\subc*\vstep)--(\dfive+\subc*\hstep,\vsix);


\draw [->] (\azero+\arsp,\vzero -\varr)--node[below]{$a$}(\aone, \vzero -\varr); 
\draw [->] (\bzero+\arsp,\vzero -\varr)--node[below]{$b$}(\bone, \vzero -\varr); 
\draw [->] (\czero+\arsp,\vzero -\varr)--node[below]{$c$}(\ctwo, \vzero -\varr); 
\draw [->] (\dzero+\arsp,\vzero -\varr)--node[below]{$d$}(\dsix, \vzero -\varr); 
\draw [->] (\azero +\arsp,\vsix + \varr) -- node[above]{$d$} (\aone  -\arsp,\vsix + \varr);
\draw [->] (\bzero +\arsp,\vsix + \varr) -- node[above]{$a$} (\bone  -\arsp,\vsix + \varr);
\draw [<-] (\czero +\arsp,\vsix + \varr) -- node[above]{$b$} (\cone  -\arsp,\vsix + \varr);
\draw [->] (\cone  +\arsp,\vsix + \varr)  -- node[above]{$a$}(\ctwo  -\arsp,\vsix + \varr);
\draw [->] (\dzero +\arsp,\vsix + \varr) -- node[above]{$b$} (\done  -\arsp,\vsix + \varr);
\draw [<-] (\done  +\arsp,\vsix + \varr)  -- node[above]{$a$}(\dtwo  -\arsp,\vsix + \varr);
\draw [->] (\dtwo  +\arsp,\vsix + \varr)  -- node[above]{$d$}(\dthree-\arsp,\vsix + \varr);
\draw [<-] (\dthree+\arsp,\vsix + \varr) -- node[above]{$b$} (\dfour -\arsp,\vsix + \varr);
\draw [->] (\dfour +\arsp,\vsix + \varr) -- node[above]{$a$} (\dfive -\arsp,\vsix + \varr);
\draw [->] (\dfive +\arsp,\vsix + \varr) -- node[above]{$c$} (\dsix  -\arsp,\vsix + \varr);

\draw (\azero, \vthree) -- (\aone, \vfour); 
\draw (\bone, \vzero) -- (\bzero, \vone); 
\draw (\bone, \vtwo) -- (\bzero, \vthree); 
\draw (\bzero, \vfive) -- (\bone, \vsix); 
\draw (\ctwo, \vzero) -- (\czero, \vtwo); 
\draw (\ctwo, \vtwo) -- (\cone, \vthree); 
\draw (\czero, \vfour) -- (\ctwo, \vsix); 
\draw (\dzero,\vone) -- (\dfive,\vsix); 
\draw (\done,\vfour) -- (\dzero,\vfive); 
\draw (\dtwo,\vfive) -- (\done,\vsix); 

\draw (\dfive,\vzero)--(\dsix,\vone);
\end{tikzpicture}
\caption{The trapezoid cell structure for the folded mapping torus $X_\f$ for the example $\f\colon \Gamma \to \Gamma$.}
\label{F:trapezoid_example}
\end{figure}
\end{center}

\end{example}

\subsection{Trapezoidal subdivisions}
\label{S:trapezoidal_subdivisions}

In \S\ref{sect:cocycles fib and E_0} we will construct `perturbed' fibrations $\eta_z\colon X_\f\to\mathbb{S}^1$ associated to certain $1$--cocycles $z$ on $X_\f$.
To aid in the construction of $\eta_z$, which will be carried out explicitly in \S\ref{sect:allowable perturbations}, it will be helpful to use a subdivided cell structure on $X_f$ that is tailored to the cocycle $z$. Here we describe the general type of cell structures that we will use.

For a trapezoidal $2$--cell $T$ of $X_\f$, recall that $\ell_-(T)$, $\ell_+(T)$, $e_-(T)$, and $e_+(T)$ denote the left, right, bottom and top arcs of $T$, respectively. As each of these arcs is in fact a union of $1$--cells of $X_\f$, we may regard them as cellular $1$--chains in $X_\f$. The trapezoid $T$ induces a \emph{$T$--orientation} on these arcs as follows: By definition, $\ell_\pm(T)$ and $e_-(T)$ are given the positive orientation (as specified in Definition~\ref{def:trapezoidcell}) so that the $1$--cells comprising $\ell_\pm(T)$ and $e_-(T)$ must appear with positive coefficient in those $1$--chains. The two sides $\ell_\pm(T)$ of $T$ are then distinguished by the convention that $e_-(T)$ is oriented from $\ell_-(T)$  to $\ell_+(T)$. Lastly the $T$--orientation on the top arc $e_+(T)$ is defined so that it is also oriented from $\ell_-(T)$ to $\ell_+(T)$.  Alternatively, using $\flow$ to flow the bottom to the top defines a map $h_T \colon e_-(T) \to e_+(T)$ which is a homeomorphism on the interior of $e_-(T)$; the $T$--orientation on $e_+(T)$ is defined such that $h_T$ is orientation preserving.  As the $T$--orientation on $e_+(T)$ may not agree with the positive orientation on $X_\f^{(1)}$, we see that the $1$--cells comprising the $1$--chain $e_+(T)$ may appear with positive or negative coefficient in this $1$--chain.  Since $e_+(T)$ is a sum of $1$--cells all contained in a single skew $1$--cell of the old cell structure, it follows that all coefficients of $e_+(T)$ have the {\em same} sign, and we define the {\em sign of $T$} to be $\zeta(T) \in \{ \pm 1\}$ such that the cellular $1$--chain $\zeta(T)e_+(T)$ has all positive coefficients. 
With these definitions, the boundary of $T$ is the cellular $1$--chain
\begin{equation}
\label{eq:trapezoid_boundary}
\partial T = e_-(T) + \ell_+(T) - e_+(T) - \ell_-(T).
\end{equation}

Note that $X_\f$ may contain some trapezoids $T$ for which one of the sides $\ell_\pm(T)$ degenerates to a single point (see the left trapezoid of Figure~\ref{F:trapezoids}). Such trapezoids are said to be \emph{degenerate}.

\begin{remark}
\label{rem:degenerate_trapezoids}
If $T$ is a degenerate trapezoid, our naming conventions ensure that $\ell_+(T)$ is the degenerate side of $T$; that is, $\ell_+(T)$ is a point whereas $\ell_-(T)$ is a nondegenerate arc.  Moreover, in this case the positive orientation on the skew $1$--cells comprising $e_+(T)$ disagrees with the $T$--orientation on the arc $e_+(T)$. Thus we necessarily have $\zeta(T) = -1$ for every degenerate trapezoid.
\end{remark}

Recall that the cell structure on $X_\f$ is constructed so that every $1$--cell is either vertical or skew. The sides $\ell_\pm(T)$ of each trapezoid $T$ consist of vertical $1$--cells, and the top $e_+(T)$ consists of (possibly several) skew $1$--cells.  On the other hand, the bottom $e_-(T)$ of each trapezoid consists of a \emph{single} skew $1$--cell. Conversely, every skew $1$--cell of $X_\f$ is equal to $e_-(T)$ for a unique trapezoid $T$. The subdivisions of $X_\f$ that we construct in \S\ref{S:subdivision} below will again be \emph{trapezoidal cell structures}, meaning that they have exactly this structure. Furthermore, each $2$--cell $T$ of the subdivision will be a trapezoid that is contained in a unique $2$--cell $T'$ of $X_\f$. In fact $T$ will be a subtrapezoid of $T'$ having top and bottom arcs, again denoted $e_\pm(T)$, that are subarcs of $e_\pm(T')$ and with side arcs, again denoted $\ell_\pm(T)$, being arcs of flowlines running all the way from the bottom to the top of $T'$.  Such a subdivision will be called a {\em trapezoidal subdivision}. As with $X_\f$, the $1$--cells are all skew or vertical and hence inherit orientations making the restriction of the fibration $\fib_0$ to each $1$--cell orientation preserving. We use the same conventions for the $T$--orientation on the sides of a trapezoid $T$ of the subdivision and define the sign $\zeta(T)$ just as for trapezoids of $X_\f$. With these definitions we have $\zeta(T) = \zeta(T')$, where $T'$ is the unique $2$--cell of $X_\f$ containing $T$.

\subsection{The semiflow on trapezoids}

As described above, each trapezoid is built from a union of subrectangles of rectangles, two subtriangles of triangles, and the horizontal arcs in between:
\[T = U_1 \cup \delta_1 \cup U_2 \cup \delta_2 \ldots \cup \delta_k \cup U_k.\]
As $T$ is a disk, $\flow$ actually restricts to a local flow on $T$.
This local flow has a very simple description, see Figure \ref{F:trapezoids2}.
\begin{proposition} \label{P:trapezoid simple flow}
For every trapezoid $T \in \mathcal T$ there are real numbers $m_1,m_2,t_0 \in \R$ defining a (possibly degenerate) Euclidean trapezoid 
\[ \hat T = \{ (r,t) \in \R^2 \mid 0 \leq r \leq 1, \, m_1 r \leq t \leq m_2 r + t_0 \}\]
and a continuous map
\[ \sigma_T\colon \hat T \to \bar T \]
such that $\sigma_T$ sends horizontal lines to horizontal arcs, and for all $(r,t) \in \hat T$ and $0 \leq s \leq m_2 r + t_0 - t$ we have
\[ \sigma_T(r,t+s) = \flow_s(\sigma_T(r,t)).\]
Moreover, $\sigma_T$ restricted to the interior of $\hat T$ is a homeomorphism onto $T$.
\end{proposition}

\begin{figure} [htb]
\begin{center}
\includegraphics[height=5cm]{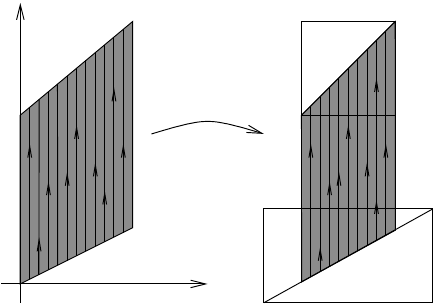} \caption{The local flow preserving map $\sigma_T\colon \hat T \to \bar T$.}
\label{F:trapezoids2}
\end{center}
\end{figure}

\begin{proof}
Since $\Pi'$ semi-conjugates $\hat \flow'$ to $\flow'$, and since $\flow'$ generates $\flow$, it follows that the map $\hat \Pi\colon \Gamma \times [0,1] \to X$ obtained by composing $\Pi'\colon \Gamma \times [0,1] \to X'$ with the quotient map $X' \to X$ satisfies
\[ (x,t+s) \mapsto \flow_s(\hat \Pi(x,t)). \]
Since $T$ is a union of images of product rectangles and triangles in $\Gamma \times [0,1]$ under $\hat \Pi$, it is straight forward to piece together the local flows $(x,t) \mapsto (x,t+s)$ on these and construct the required $\hat T$ and map $\sigma_T$.
\end{proof}

\begin{conv}
We use the maps $\sigma_T$ from Proposition \ref{P:trapezoid simple flow} to give the interior of every trapezoid a smooth structure.   Note that the local flow $\flow|_T$ on any trapezoid $T$ is thus smooth, as is $\fib_0|_T$.
\end{conv}

\section{Local models for the semiflow} \label{sect:local models}

The simplest behavior of $\flow$ is in the interior of a trapezoid where Proposition \ref{P:trapezoid simple flow} gives a completely transparent picture of what $\flow$ looks like, namely it is a local flow on a trapezoid by translation parallel to the direction of the vertical sides.   Our next goal is to complete this description by giving ``local models'' for the semiflow near any point of the $1$--skeleton of $X_\f$.

\begin{remark}
The discussion in this section should be compared with \cite[Section 2.3]{Gau1}.  In \cite{Wang}, similar issues are addressed using work of Bestvina and Brady \cite{BB97}.
\end{remark}

\begin{defn}[Local model]
\label{D:local model}
A {\em local model} for $\flow$ is a subset of $X_\f$ of the form
\[ \M(K,s_0) = \bigcup_{0 \leq s \leq s_0} \flow_s(K)\]
where $s_0 > 0$ and $K$ is a closed contractible neighborhood of a point of $\Gamma_t$, for some $t$.
We call $K$ the {\em bottom} of $\M$ and $\flow_{s_0}(K)$ the {\em top} and we call $\M$ minus the top and bottom the {\em flow-interior} of $\M$.
\end{defn}

About every point of $X_\f$ we will construct a local model $\M = \M(K,s_0)$ and study the maps 
\[ \flow_s\colon \flow_t(K) \to \flow_{t+s}(K) \]
induced by the flow on the horizontal pieces $\flow_t(K)$ and $\flow_{t+s}(K)$, for $0\le t < t+s \leq s_0$. As these maps will enjoy several useful properties, it is convenient to introduce the following terminology.

\begin{defn}[Control rel boundary]
Given two finite graphs $Q_0$ and $Q_1$ which are non-degenerate (i.e.~have no zero-dimensional components), let $\partial Q_0$ and $\partial Q_1$ denote the set of valence $1$ vertices.  We will say that a continuous map $\Upsilon \colon Q_0 \to Q_1$ is {\em controlled rel boundary} if the following holds:
\begin{enumerate}
\item For every $x \in Q_1$, $\Upsilon^{-1}(x)$ is finite;
\item $\Upsilon$ restricts to a bijection from $\partial Q_0$ to $\partial Q_1$; and
\item $\Upsilon \colon (Q_0,\partial Q_0) \to (Q_1,\partial Q_1)$ is a relative homotopy equivalence.  That is, there exists a continuous map $\Upsilon_0 \colon (Q_1,\partial Q_1) \to (Q_0,\partial Q_0)$ so that 
\[ \Upsilon_0 \circ \Upsilon\colon Q_0 \to Q_0 \mbox{ and } \Upsilon \circ \Upsilon_0\colon Q_1 \to Q_1 \]
are homotopic rel $\partial Q_0$ and $\partial Q_1$, respectively, to the identity.
\end{enumerate}
Note that $\Upsilon_0|_{\partial Q_1} = (\Upsilon|_{\partial Q_0})^{-1}$.
\end{defn}

This notion gives rise to a useful local definition for a map between graphs (cf.~\cite[Definition 2.11]{Gau1}).
\begin{defn}[local control]
If $Q_0$ and $Q_1$ are graphs, then a map $\Upsilon\colon Q_0 \to Q_1$ is {\em locally controlled} if there exists subdivisions $Q_0'$ of $Q_0$ and $Q_1'$ of $Q_1$ with the following properties.
\begin{enumerate}
\item $Q_0'$ and $Q_1'$ are each a finite union of non-degenerate, connected subgraphs:
\[ Q_0' = Q_0^1 \cup \ldots \cup Q_0^r \, , \, Q_1' = Q_1^1 \cup \ldots \cup Q_1^r;\]
\item For each $i = 0,1$ and every $j \neq j'$, $Q_i^j \cap Q_i^{j'} \subset \partial Q_i^j \cup \partial Q_i^{j'}$; and 
\item For each $j = 1,\ldots,r$, $\Upsilon(Q_0^j) = Q_1^j$ and the restriction $\Upsilon|_{Q_0^j}\colon Q_0^j \to Q_1^j$ is controlled rel boundary.
\end{enumerate}
\end{defn}

\begin{proposition} \label{P:local control h.e.}
If $Q_0,Q_1$ are finite graphs and $\Upsilon \colon Q_0 \to Q_1$ is locally controlled, then $\Upsilon$ is a homotopy equivalence and for every $x \in Q_1$, $\Upsilon^{-1}(x)$ is finite.
\end{proposition}
\begin{proof}
Since $\Upsilon|_{Q_0^j}$ is locally controlled rel boundary for all $j$, it follows that
\[ \Upsilon \left( \bigcup_{j=1}^r \partial Q_0^j \right) = \bigcup_{j=1}^r \partial Q_1^j\]
and the restriction of $\Upsilon$ to $\cup \partial Q_0^j$ is a bijection onto $\cup \partial Q_1^j$ (since this is true for the restriction $\Upsilon|_{\partial Q_0^j}$ for each $j$). By definition of controlled rel boundary, there exist $r$ continuous maps
\[ \Upsilon_0^j \colon Q_1^j \to Q_0^j \]
such that $\Upsilon_0^j \circ \Upsilon|_{Q_0^j} \colon Q_0^j \to Q_0^j$ and $\Upsilon|_{Q_0^j} \circ \Upsilon_0^j\colon Q_1^j\to Q_1^j$ are homotopic relative to $\partial Q_0^j$ and $\partial Q_1^j$, respectively, to the identity.  Since
\[ \Upsilon_0^j|_{\partial Q_1^j}  = (\Upsilon|_{\partial Q_0^j})^{-1} \]
it follows that $\Upsilon_0^j$ and $\Upsilon_0^{j'}$ agree on $Q_1^j \cap Q_1^{j'}$ for all $j, j'$.   Therefore, the $\Upsilon_0^j$ are the restrictions of a continuous map
\[ \Upsilon_0\colon Q_1 \to Q_0.\]
The various homotopies to the identities glue together to define homotopies of $\Upsilon_0 \circ \Upsilon \colon Q_0 \to Q_0$ and $\Upsilon \circ \Upsilon_0\colon Q_1 \to Q_1$ to the respective identities.  Thus, $\Upsilon$ is a homotopy equivalence.

For any $x \in Q_1$, we have $x \in Q_1^j$ for some $j$ and hence $\Upsilon^{-1}(x) = \Upsilon|_{Q_0^j}^{-1}(x)$ is finite.
\end{proof}

\begin{proposition} \label{P:local model maps}
For every open set $U\subset X_\f$ and every point $x\in U$, there exists a local model neighborhood $\M(K,s_0)\subset U$ containing $x$ in its interior such that
\[ \flow_s\colon \flow_t(K) \to \flow_{t+s}(K)\]
is controlled rel boundary for all $0 \leq s, t \leq s+t \leq s_0$.
\end{proposition}
\begin{proof}
Recall that  $0 < 1/m < \dotsb < (m-1)/m < 1$ are the times at which the Stallings folds $\Delta_j = \Gamma_{j/m}$ occur during the folding process. Let $t'\in [0,1)$ be such that $x\in \Gamma_{t'}$ and $t' \in [(j-1)/m,j/m)$ for some $j = 1,\dotsc,m$. Similarly choose $t''\in (0,1]$ so that $x\in \Gamma_{t''}$ and $t''\in ((i-1)/m,i/m]$ for some $i=1,\dotsc,m$. Thus $t'=t''$ unless $x\in \Gamma_0=\Gamma =\Gamma_1$, in which case we have $t'=0$ and $t''=1$. The proposition essentially follows from the fact that for small $s > 0$ the maps $\flow_s\colon\Gamma_{t'}\to \Gamma_{t'+s}$ and $\flow_s\colon\Gamma_{t''-s}\to\Gamma_{t''}$ are just the folding maps $r_{t'+s,t'}$ and $r_{t'',t''-s}$, respectively.

Note that while $\Gamma_{t'}=\Gamma_{t''}$ as sets and topological spaces, the graph structure on $\Gamma_{t'}$ may be a combinatorial subdivision of that in $\Gamma_{t''}$ (when $t'=0$ and $t''=1$). In any case, we may choose a closed contractible neighborhood $K$ of $x$ in $\Gamma_{t'}=\Gamma_{t''}$ so that $\partial K$ is disjoint from $V\Gamma_{t'}$ and $V\Gamma_{t''}$ and $K \subset U$.  We let $D > 0$ be any number less than the minimal distance from $\partial K$ to $V\Gamma_{t'}\cup V\Gamma_{t''}$ in either of the metrics on $\Gamma_{t'}$ and $\Gamma_{t''}$.   Fix any $0 < \delta < \min\{\tfrac{D}{m},\vert j/m-t' \vert\,, \vert t'' - (i-1)/m \vert \}$.  In particular
\[ [t',t'+\delta] \subset [(j-1)/m,j/m) \mbox{ and } [t''-\delta,t''] \subset ((i-1)/m,i/m].\]

For any $t' \leq s \leq t' +\delta$.  The map $r_{s,t'}$ identifies initial segments of a pair of edges $e_1,e_2 \subset \Gamma_{t'}$ of length $(s-t')m$ and is a homeomorphism in the complement of these segments.  Set $v = o(e_1)=o(e_2)$ of $\Gamma_{t'}$ and $K_s = r_{s,t'}(K)$. 
For any $t' \leq s \leq s' \leq t'+\delta$ we claim that $r_{s',s}\vert_{K_s} \colon K_s \to K_{s'}$ is controlled rel boundary.  To see this, first note that $r_{s',t'} = r_{s',s} \circ r_{s,t'}$ is a homeomorphism outside the $(s'-t')m \leq m \delta < D$ neighborhood of $v$.  Since the distance between $v$ and $\partial K$ is less than $D$ in $\Gamma_{t'}$, if $v\notin K$, then $K$ is outside the $D$--neighborhood of $v$ and $r_{s',t'}\vert_{K}$ is a homeomorphism.  Consequently, $r_{s',s}|_{K_s}$ is also a homeomorphism.  Otherwise the $D$ neighborhood of $v$ is entirely contained in $K$.  In this case, $r_{s',s}\vert_{K_s}$ simply identifies initial segments of length $(s'-s)m$ of a pair of edges, and is nevertheless controlled rel boundary.

Now suppose $t'' - \delta \leq s \leq t''$ and write $K_s = (r_{t'',s})^{-1}(K) \subset\Gamma_s$.  The map $r_{s,t''-\delta}$ identifies initial segments of length $(s-t''+\delta)m \leq m \delta < D$ of a pair of edges $e_1,e_2 \subset \Gamma_{t''-\delta}$ and is a homeomorphism outside the $m \delta < D$ neighborhood of $v = o(e_1) = o(e_2)$.  For each $i =1,2$, let $\gamma_i \subset e_i$ be the initial segment of length $m \delta$, and let $u_i \in \gamma_i$ be the endpoint different than $v$.  Then $r_{t'',t''-\delta}(\gamma_1) = r_{t'',t''-\delta}(\gamma_2)$ is a path in $\Gamma_{t''}$ of length $m \delta < D$ and $u = r_{t'',t''-\delta}(u_1) = r_{t'',t''-\delta}(u_2) \in V\Gamma_{t''}$.  Since the distance from $u$ to $\partial K$ in $\Gamma_{t''}$ is at most $D$, either $r_{t'',t''-\delta}(\gamma_1) = r_{t'',t''-\delta}(\gamma_2) \subset K$ or $r_{t'',t''-\delta}(\gamma_1) = r_{t'',t''-\delta}(\gamma_2)$ is disjoint from $K$.  In the former case, $\gamma_1 \cup \gamma_2 \subset K_{t''-\delta}$ and $r_{t'',t''-\delta}\vert_{K_{t''-\delta}}$ just folds $\gamma_1$ and $\gamma_2$ together and is homeomorphism outside these segments.  In the latter case, $r_{t'',t''-\delta}\vert_{K_{t''-\delta}}$ is a homeomorphism onto $K$.  In either case, for any $t''-\delta \leq s \leq s' \leq t''$ we may argue as in the situation of $t' \leq s \leq s' \leq t'+\delta$ replacing $K$ with $K_{t''-\delta}$ to prove that  the map $r_{s',s}\vert_{K_s}$ is controlled rel boundary.

Therefore, by choosing $\delta$ sufficiently small and setting $s_0 = 2\delta$ and $K' = K_{t''-\delta}$, we see that $\M(K',s_0)$ is a local model contained in $U$ satisfying the conclusion of the proposition.
\end{proof}

\subsection{Examples of local models}\label{sect:example local}
Local models will play an important role in the sequel. In an effort to familiarize the reader with these tools, and to clarify the proof of Proposition~\ref{P:local model maps}, in this subsection we will describe and illustrate some specific types of local models. In all figures throughout this subsection, we have indicated $1$--cells of $X_\f$ with thick lines. A single arrow indicates a vertical $1$--cell and double arrow indicates a skew $1$--cell. 

The simplest type of local models are homeomorphic to surfaces, in which case $K\cong[-1,1]$ and $K\times[0,s_0]\cong\M(K,s_0)$ via a homeomorphism $(r,t)\mapsto \flow_t(r)$. We call these \emph{surface local models}. These exist for any point in the interior of a trapezoid (for example, by Proposition~\ref{P:trapezoid simple flow}), but they also exist near $1$--cells with degree $2$, and $0$--cells for which all adjacent $1$--cells are degree $2$. In this case, note that the semiflow restricts to a local flow, and the maps
\[ \flow_s\colon \flow_t(K) \to \flow_{t+s}(K)\]
are all homeomorphisms (in particular, they are controlled rel boundary).

However, local models near the $1$--cells $e$ with $d(e) > 2$ are necessarily more complicated. It is thus useful to distinguish these, and we do so with the following
\begin{defn} \label{def:E_0}
Let $E_0$ denote the set of $1$--cells $e$ with $d(e) > 2$.  Let $\mathcal E_0 \subset X_\f^{(1)}$ denote the subgraph of the $1$--skeleton with edge set $E_0$ and vertex set $\{ o(e),t(e) \mid e \in E_0\}$.  We give $\mathcal E_0$ the restricted metric structure and orientation from that of $X_\f^{(1)}$; see Definition \ref{def:trapezoidcell}.
\end{defn}

Let us next consider the local model about a point $x$ in a $1$--cell $e$ with $d(e) > 2$. When $e$ is a vertical $1$--cell there is a local model for which the semiflow restricts to a local flow. Indeed, if $x = \Gamma_{t}\cap e$, then for nearby $t'<t<t''$ with $t''-t'$ sufficiently small, we will have that $\Gamma_s\cap e$ is a vertex of $\Gamma_s$ of degree $d(e)$ for all $s\in [t',t'']$.  Subdividing each edge of $\Gamma_{t'}$ adjacent to $x'=\Gamma_{t'}\cap e$ by adding a point very close to $x'$, and letting $K$ be the star of $x'$, we see that the flow defines a homeomorphism $K\times[t',t'']\to\M(K,t''-t')$ via $(y,s)\mapsto \flow_{s-t'}(y)$. See Figure~\ref{F:localedge}.

\begin{figure} [htb]
\begin{center}
\includegraphics[width=12cm]{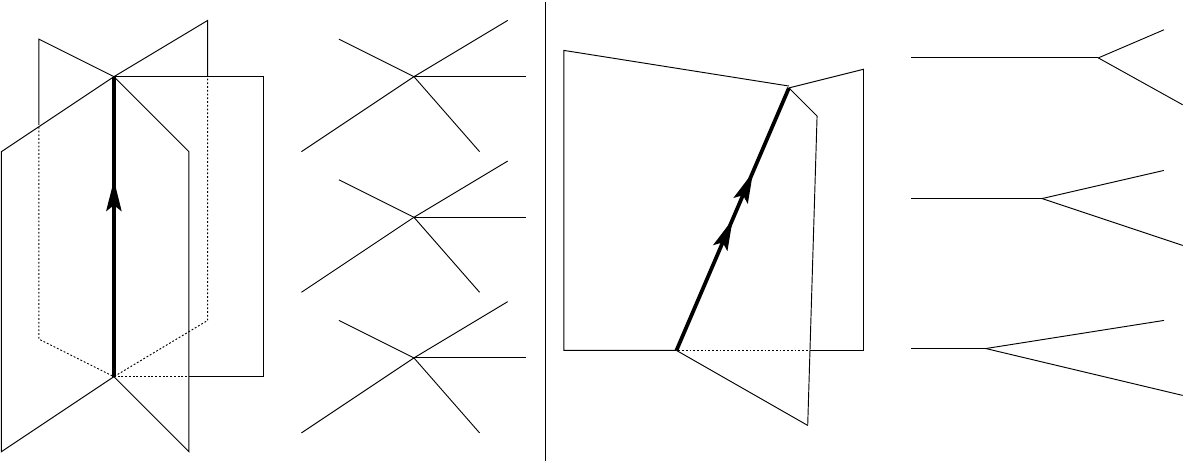} \caption{A local model for a vertical $1$--cell $e$ with $d(e) = 5$ (left) and a skew $1$--cell (right).  In both cases, $\flow_s\colon \flow_t(K) \to \flow_{t+s}(K)$ is controlled rel boundary, but only on the left is it a homeomorphism, and these are indicated to the right of each local model.}
\label{F:localedge}
\end{center}
\end{figure}

When $e$ is a skew $1$--cell with $d(e) = 3$ there is a local model neighborhood as shown at the right of Figure \ref{F:localedge}.  Observe that any such $x$ is contained in $\Gamma_t$ for some $i/m < t < (i+1)/m$ and some building block $X_i'$ from the proof of Proposition \ref{P:X' cell structure}.   In particular, the local models are obtained from the explicit construction of the building blocks in the proof of Proposition \ref{P:X' cell structure}.
Here we note that the maps $\flow_s\colon \flow_t(K) \to \flow_{t+s}(K)$ are not homeomorphisms, although they are controlled rel boundary.

Finally we consider local models near $0$--cells, for which there is a variety of possible behavior. To simplify the discussion, let us assume $v$ is a $0$--cell that does not lie in the interior of an old skew $1$--cell and does not have a local model homeomorphic to a surface. For such a point, we will actually illustrate two local models, one above $v$ and one below $v$. By combining these, one obtains what we call a \emph{stacked local model}.

\begin{figure} [htb]
\begin{center}
\includegraphics[width=12cm]{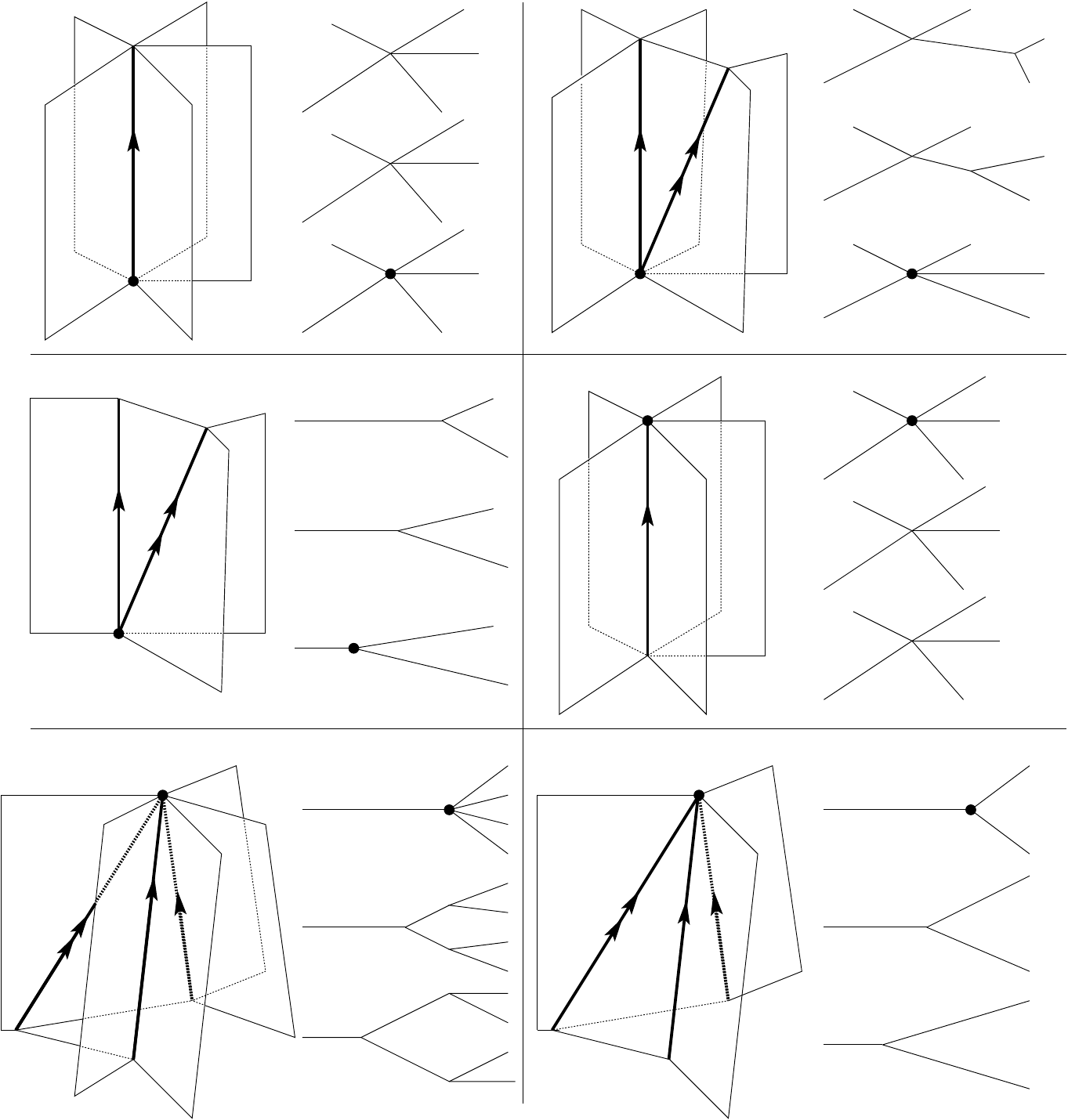} \caption{{Pictures of local models above the $0$--cell {(\bf (a)}--{\bf (c)}) and below the $0$--cell ({\bf (d)}--{\bf (e)}).   The maps $\flow_s\colon \flow_t(K) \to \flow_{t+s}(K)$ are indicated on the right of each local model.}}
\label{F:sinksource}
\end{center}
	\setlength{\unitlength}{1cm}
	\begin{picture}(0,0)(0,0)
	\put(-6,14.6){\bf{(a)}}
	\put(0,14.6){\bf{(b)}}
	\put(-6,10.3){\bf{(c)}}
	\put(0,10.3){\bf{(d)}}
	\put(-6,6){\bf{(e)}}
	\put(0,6){\bf{(f)}}
	\end{picture}
\end{figure}

Under our assumptions, $v$ is a $0$--cell of the old cell structure, and is thus a vertex of valence at least $3$ in $\Gamma_{i/m}$ for some  $i = 0,\ldots, m-1$ (if $v$ had valence $2$, it would have a surface local model neighborhood). We construct a local model above $v$ by subdividing each edge of $\Gamma_{i/m}$ adjacent to $v$ and taking $K$ to be the closed star of $v$. Taking $s_0 \in (0,1/m)$, we obtain a local model $\M(K,s_0)$ for $v$ that is in fact contained in the image of the building block $X'_i$. 

There are essentially two possibilities for $\M(K,s_0)$. If $v = o(e) = o(e')$ for two distinct edges $e,e'$ which are folded together in $\Gamma_{(i+1)/m}$, then this folding produces a skew $1$--cell $e_1$ with $o(e_1)= v$ and $d(e_1) = 3$, and a vertical $1$--cell $e_2$ with $o(e_2) = v$ with $d(e_2) = \val(v) - 1$, where $\val(v)$ is the valence of $v$ in $\Gamma_{i/m}$.  We have illustrated $\M(K,s_0)$ in Figure \ref{F:sinksource} {\bf (b)} (for $\val(v) = 5$) and {\bf (c)} (for $\val(v) = 3$).  If $v$ is not the initial vertex of a pair of edges folded together in $\Gamma_{(i+1)/m}$, then (for $s_0$ sufficiently small), $\M(K,s_0)$ is a product.  See Figure \ref{F:sinksource} {\bf (a)} (for $\val(v) = 5$).

A local model below $v$ is described similarly.  If $v \in \Gamma_0$, then $v \in \Gamma_1$ as well, so we may assume that $v \in \Gamma_{i/m}$ with $i = 1,\ldots, m$ is a vertex with $\val(v) \geq 3$.  We again take the star $K'$ of $v$ in $\Gamma_{i/m}$ (after subdividing), and set $K = (\flow_{s_0})^{-1}(K')$ for a small $s_0\in (0,1/m)$. Then 
\[ \M(K,s_0) = \{ (\flow_s)^{-1}(K') \mid 0 \leq s \leq s_0 \}\]
is a local model below $v$. Again there are essentially two possibilities, which are illustrated in Figure~\ref{F:sinksource} {\bf (d)--(f)} (for $\val(v) = 5$ or $3$). Note that in all cases, the local models have the property that the flow maps $\flow_s\colon \flow_t(K) \to \flow_{t+s}(K)$ are controlled rel boundary.

\subsection{The $E_0$--class.} \label{sect:epsilon defined} 

Recall that the $1$--skeleton $X_\f^{(1)}$ is oriented; see Definition \ref{def:trapezoidcell}.  Define a cellular $1$--chain $\epsilon \in C_1(X_\f;\R)$ by
\[ \epsilon = \frac{1}{2}\sum_e [2-d(e)] e, \]
where the sum is over all $1$--cells $e$ in $X_\f^{(1)}$, although the
coefficients are nonzero only on edges $e \in E_0$. As an application of the local models, we prove:
\begin{proposition} \label{P:epsilon cycle}
The chain $\epsilon$ is a cycle $\epsilon \in Z_1(X_\f;\R)$.
\end{proposition}
\begin{proof}
Fix a $0$--cell $v \in V\mathcal E_0$.  We need to show that the value
of $\partial \epsilon$ on $v$ is zero.  To prove this fact, choose a local model neighborhood $\M(K,s_0)$ around $v$, so that $v \in \flow_{s'}(K)$ for some $0 < s' < s_0$.  Note that because $v \in V\mathcal E_0$, this is not a surface local model, and hence it is a stacked local model which we write as $\M = \M^+\cup\M^-$ for  local models $\M^-$ and $\M^+$ below and above $v$, respectively.

Now let $e_1^-,\ldots,e_r^-$ (respectively, $e_1^+,\ldots,e_{r'}^+$) denote $1$--cells in $E_0$ with $t(e_j^-) = v$ (respectively, $o(e_j^+) = v$).  Every $e_j^{\pm}$ nontrivially intersects the  local model $\M^{\pm}$, as shown in Figure \ref{F:sinksource}.
If we let $v_j^- = e_j^- \cap K = \flow_0(K)$ and $v_j^+ = e_j^+ \cap \flow_{s_0}(K)$ then for all $j$ we have
\[ \val(v_j^\pm) = d(e_j^\pm)\]
where $\val(v_j^-)$ denotes the valence of $v_j^-$ in $K$, and $\val(v_j^+)$ denotes the valence of $v_j^+$ in $\flow_{s_0}(K)$.
Moreover, by inspecting the local models, we can see that
\[ \sum_{j=1}^r 2- \val(v_j^-) = 2- \val(v) = \sum_{j=1}^{r'} 2 - \val(v_j^+) \]
where $\val(v)$ denotes the valence of $v$ in $\flow_s(K)$.  Combining these two equations we obtain
\[ \sum_{j=1}^r 2- d(e_j^-) = 2- \val(v) = \sum_{j=1}^{r'} 2 - d(e_j^+). \]
Finally, we observe that the coefficient of $v$ in $\partial \epsilon$ is equal to one half of the difference of the two sums here, and hence is zero.
This completes the proof of the proposition.
\end{proof}

\begin{defn}\label{D:E_0-class}
We refer to the homology class of $\epsilon$ as the {\em $E_0$--class}.  As an abuse of notation, we also denote the homology class by $\epsilon \in H_1(X_\f;\R)$.
\end{defn}

\begin{remark} \label{R:Euler comparison}
We now explain the strong geometric analogy between our construction of the $E_0$--class and the Poincar\'e dual of the Euler class in the $3$--manifold setting.

Suppose that $\eta_0 \colon M \to \mathbb{S}^1$ is a closed, oriented $3$--manifold that fibers over $\mathbb{S}^1$ inducing some primitive integral element $(\eta_0)_*= u_0$ with pseudo-Anosov monodromy $F \colon S \to S$.   Let $T\mathfrak F \to M$ denote the tangent plane bundle of the foliation $\mathfrak F$ by fibers (which is also the kernel of the $1$--form $\eta_0^*(vol_{\mathbb{S}^1})$).  Taking any section $\sigma \colon M \to T\mathfrak F$ transverse to the zero section $Z \subset T\mathfrak F$, the Poincar\'e dual $e$ of the Euler class of $T\mathfrak F$ is represented by the oriented $1$--submanifold $\sigma^{-1}(Z)$; see \cite[\S12]{BT}.

Now suppose that the stable foliation $\mathcal F$ is orientable, all singularities are $4$--prong, and that the orientation is preserved by $F$.  Suspending $\mathcal F$ defines a $2$--dimensional singular foliation that intersects each fiber in the stable foliation for the first return map to that fiber.  Because $\mathcal F$ is orientable (and the orientation is preserved by $F$), we can choose a section $\sigma$ of $T\mathfrak F$ that is also tangent to $\mathcal F$ at every nonsingular point, and is zero precisely at the singular points.  Moreover, since the singularities are $4$--prong (i.e.~index $-1$), this can be done so that $\sigma$ is transverse to $Z$.  Therefore, $e$ is represented by the suspension of the singular points of $\mathcal F$.

This conclusion actually holds without any assumption on orientability of $\mathcal F$ or on the number of prongs at the singularity, with the appropriate modifications.  Namely, if $\gamma_1,\ldots,\gamma_n$ are all the closed orbits of singular points under the suspension flow, and if the number of prongs of a singular point $p_i$ contained in $\gamma_i$ is denoted $d(\gamma_i)$, then $e$ is represented by the oriented $1$--manifold
\[ \frac{1}{2} \sum_{i=1}^n [ 2- d(\gamma_i) ] \gamma_i.\]
Here, $\frac{1}{2}(2-d(\gamma_i))$ is just the index of the singularity (see e.g.~\cite[Expos\'e 5]{FLP}).   The $E_0$--class is the natural analogue in which the closed orbits are replaced by the $1$--cells of the cell structure, and the number of prongs is replaced by degree (as the notation suggests).
\end{remark}

\section{Cocycles, fibrations, and the $E_0$--class} \label{sect:cocycles fib and E_0} 

Here we begin a more detailed analysis of the map $\fib_0\colon X_\f \to \mathbb{S}^1$.  As described in \S\ref{sect: folded mapping torus} $\fib_0$ induces the integral element $u_0 = (\fib_0)_* \in H^1(X_\f ; \mathbb R) \cong H^1(G ; \mathbb R) \cong \Hom(G,\mathbb R)$;  see Theorem \ref{T:homotopy type} and Convention \ref{conv:G identified}.   In this section we explain how to construct the neighborhood $\A \subset H^1(X_\f;\R)$ for which the integral elements in this neighborhood are similarly induced by maps $\fib\colon X_\f \to \mathbb{S}^1$.  We also begin to explain how much of the structure of $\fib_0$ is also present in $\fib$.  We complete this section with the proof of Theorem \ref{T:folded basics}.

\subsection{Cellular cocycle representatives}

The first task is to provide an explicit cellular cocycle $z_0 \in Z^1(X_\f;\Z)$ representing $u_0$.  This is done as follows.  Recall that unless otherwise stated, the cell structure is always the trapezoid cell structure.

The restriction of $\fib_0$ to any cell $\sigma \subset X_\f$ can be lifted to the universal cover:
\[ \xymatrix{ 
 & \mathbb R \ar[d] \\
\sigma \ar[ur]^{\widetilde \fib_0|_\sigma} \ar[r]_{\fib_0|_\sigma} & \mathbb{S}^1} \]
Moreover, this lift is unique up to composing with a covering transformation $x \mapsto x + n$, $n \in \mathbb Z$.

For any $1$--cell $e$ the difference in the values at the end points
of the lift $\widetilde \fib_0|_e$ is independent of the lift, and we
denote this difference by $z_0(e)$.  This assignment defines a cellular cochain $z_0 \in C^1(X_\f; \mathbb R)$.

\begin{lemma} \label{L:z_0 represents u_0}
The cochain $z_0$ is a cocycle representing the class $u_0$.
\end{lemma}
\begin{proof}
For any based loop $\gamma\colon [0,1] \to X_\f$ constructed as a concatenation of oriented $1$--cells $e_1^{\delta_1} e_2^{\delta_2} \cdots e_k^{\delta_k}$, where $\delta_j \in \{ \pm 1\}$, we obtain lifts $\widetilde \fib_0|_{e_j}$ to $\mathbb R$ by restricting a lift $\widetilde{\fib_0 \circ \gamma}\colon [0,1] \to \mathbb R$ of $\fib_0 \circ \gamma$.  Then
\begin{equation} \label{E:cocycle = hom} \begin{array}{rcl} u_0(\gamma) & = & (\fib_0)_*(\gamma) = \widetilde{\fib_0 \circ \gamma}(1) - \widetilde{\fib_0 \circ \gamma}(0) \\
 & = & \displaystyle{ \sum_{j=1}^k \delta_j z_0(e_j)}  = \displaystyle{ z_0 ( \sum_{j=1}^k \delta_j e_j )}\\
 & = & z_0(\gamma). \end{array}
 \end{equation}
In particular, the value of $z_0$ around the boundary of any $2$--cell is the value of $u_0$ around the cell which is zero.  Thus $z_0$ is a cocycle, and it agrees with $u_0$ on every loop.
\end{proof}

\subsection{Positive cocycles} \label{sect:positive cocycles}

\begin{defn}
Call a cocycle $z \in Z^1(X_\f; \R)$ {\em positive} if 
\[ z(e) > 0 \]
for every $1$--cell $e$.
\end{defn}

\begin{remark}
This definition is analogous to the notion of {\em positivity} in Gautero's dynamical $2$--complex (see \cite[Definition 2.1]{Gau1}), and will serve a similar purpose.  See also Wang \cite{Wang}.
\end{remark}

Since the restriction of $\fib_0$ to any $1$--cell $e$ is an orientation preserving diffeomorphism, the discussion in the previous section together with Lemma \ref{L:z_0 represents u_0} shows that $z_0$ is a positive cocycle.
\begin{proposition} \label{P:A open and convex}
The set of cohomology classes represented by positive $1$--cocycles is an open, convex cone $\A \subset H^1(X_\f; \mathbb R)$.
\end{proposition}
\begin{proof}  Choose a basis for $H^1(X_\f; \R)$ represented by cocycles $\zeta_1,\ldots,\zeta_b$.  Suppose $z$ is positive, so that $z(e) > 0$ for every $e \in E_0$.  Then one can find $s > 0$ so that for all $s_1,\ldots,s_b$ with $|s_j| < s$ for every $j$, we also have
\[ z(e) + \sum_{j=1}^b s_j \zeta_j(e) > 0 \]
for every $1$--cell $e$.   The set
\[ \{ [t(z + s_j \zeta_j)] \in H^1(X_\f;\R) \mid |s_j| < s, \, t > 0\} \]
is a neighborhood of $[z] \in H^1(X_\f; \R)$, consisting entirely of cohomology classes represented by positive cocycles.  Therefore, $\A$ is open.

Suppose $z_1,z_2$ are positive and $s,t > 0$.  For any $1$--cell $e$, we have $z_1(e) > 0$ and $z_2(e) > 0$ and thus
\[ (t z_1 + s z_2)(e) = tz_1(e) + s z_2(e) > 0.\]
Therefore, $\A$ is a convex cone.
\end{proof}

\begin{defn}[Positive cone]
We denote by $\A_{X_\f}$ the set of all $u\in H^1(G_\fee;\R)=H^1(X_\f; \R)$ such that $u$ can be represented by a positive 1--cocycle. We call $\A_{X_\f}$ the \emph{positive cone} corresponding to $X_\f$.
\end{defn}

For each positive $1$--cocycle $z$ representing an integral class in $u\in\A$, we now seek to construct a `perturbed' fibration $\eta_z\colon X_\f\to \mathbb{S}^1$ that, among other things, induces the homomorphism $u\colon G\to \Z$ on the level of fundamental group. Our construction will make use of the general structure of trapezoidal subdivisions of $X_\f$ introduced in \S\ref{S:trapezoidal_subdivisions}, and we refer the reader to that section for our conventions regarding $T$--orientations and the boundary arcs $\ell_\pm(T)$ and $e_\pm(T)$ of a trapezoid $T$.

\subsection{The subdivision procedure}
\label{S:subdivision}

Consider a trapezoidal subdivision $Y$ of $X_\f$ (see \S\ref{S:trapezoidal_subdivisions}). Suppose that $Y$ contains a trapezoid $T_0$ whose top arc $e_+(T_0)$ consists of a single skew $1$--cell. This $1$--cell is then necessarily the bottom of another trapezoid $T_1$ which may again have a single $1$--cell along its top arc $e_+(T_1)$. Continuing in this manner we may be able to find a sequence of trapezoids $T_0,T_1,\dotsc$ whose top arcs each consist of a single skew $1$--cell $\zeta(T_i)e_+(T_i) =  e_-(T_{i+1})$. A maximal such sequence (which is completely determined by the initial trapezoid $T_0$) must either terminate in a trapezoid whose top arc contains multiple $1$--cells, or, since there are only finitely many such trapezoids in $Y$, must eventually become periodic. In the latter case, the trapezoids $T_i,T_{i+1},\dotsc,T_{i+k} = T_i$ comprising the periodic part of the sequence glue together to give an annulus or m\"obius strip in $Y$ which is by construction invariant under the flow $\flow$. Let us call such an object an \emph{invariant band}.

\begin{lemma}
\label{L:invariant band}
Every invariant band $B$ of a trapezoidal subdivision $Y$ of $X_\f$ contains a closed orbit $\mathcal{O}$ of $\flow$ in its interior that intersects each trapezoid of the band in a single vertical arc.
\end{lemma}
\begin{proof}
Let us write the invariant band as $B = T_0 \cup \dotsb \cup T_k$, where $T_k = T_0$. Since $B$ is invariant under $\flow$ and every flowline of $X_\f$ intersects the graph $\Gamma\subset X_\f$ infinitely often, it must be the case that some trapezoid of $B$, say $T_0$, intersects $\Gamma$ nontrivially. Moreover, since $T_0$ is a subtrapezoid of a unique trapezoid $T'$ of $X_\f$ and since $\Gamma$ is by construction disjoint from the interior of $e_\pm(T')$ and can intersect $T'$ in at most one connected component (see \S\ref{sect:trapezoids}), we see that $\Gamma \cap T_0$ must consist of a single arc $c$ running from $\ell_-(T_0)$ to $\ell_+(T_0)$. As the intersection of $\Gamma$ with any trapezoid of $X_\f$ is either empty or contained in a single edge of $\Gamma$, we furthermore see that $c \subset \Gamma$ is a subarc of a single edge $e$ of $\Gamma$.

Let $f'\colon c\to c$ denote the first return map of the restricted flow $\flow\vert_B$ to $c\subset B$. Since the first return map of $\flow$ to $\Gamma$ is the original map $f$, it follows that $f' = f^m\vert_c$ for some $m\geq 1$. Using the linear structure on $\Gamma$ to identify $c$ with a closed interval $[0,1]$, we now see that $f'$ is an affine homeomorphism from $c$ to itself. Therefore, $f'$ necessarily has a fixed point $x$ in the interior of $c$. Taking the flowline through $x$ then gives a closed orbit of $\flow$ contained in the interior of $B$ that intersects each trapezoid $T_i$ in a single vertical arc.
\end{proof}

\begin{remark}
\label{R:band_subdivision_unnecessary}
When $f\colon \Gamma\to\Gamma$ is an expanding irreducible train track map, Lemma~\ref{L:linear maps expanding on all scales} shows that $f$ is expanding on all scales and so there cannot exist a subarc $c\subset \Gamma$ of an edge of $\Gamma$ and a power $m\geq 1$ for which $f^m\vert_c$ is a self-homeomorphism. Therefore the proof of Lemma~\ref{L:invariant band} implies that a trapezoidal subdivision of $X_\f$ cannot contain any invariant bands when $f\colon \Gamma\to\Gamma$ is an expanding irreducible train track map.
\end{remark}

We now define the subdivision procedure that will be used in \S\ref{sect:allowable perturbations} below.

\begin{defn}[Standard subdivision]
\label{D:standard subdivision}
Let $Y$ be a trapezoidal subdivision of $X_\f$. The \emph{standard subdivision} $\widehat{Y}$ of $Y$ is constructed as follows. Firstly, for each invariant band $B=T_0 \cup \dotsb \cup T_k$ (with $T_k = T_0$) of $Y$, Lemma~\ref{L:invariant band} provides a closed orbit $\mathcal{O}$ of $\flow$ in the interior of $B$. The top and bottom arcs of each trapezoid $T_i$ consist of single skew $1$--cells of $Y$, and the maps $h_{T_i}\colon e_-(T_i)\to e_+(T_i)$ give us homeomorphisms $e_-(T_i)\to e_-(T_{i+1})$ between these $1$--cells.
We subdivide each of these $1$--cell into two $1$--cells by adding a vertex at the point $v_i = \mathcal{O}\cap e_-(T_i)$ for $0 \leq i < k$. 
We furthermore subdivide each trapezoid $T_i$ of $B$ into two trapezoids by adding a vertical $1$--cell $\beta_i =\mathcal{O}\cap T_i$ from $v_i$ to $v_{i+1}$ (notice that these $1$--cells are by construction arcs of flowlines). Let $Y^*$ denote the auxiliary trapezoidal subdivision of $Y$ obtained by simultaneously performing this procedure to each invariant band of $Y$. In light of Remark~\ref{R:band_subdivision_unnecessary}, we see that this has no effect (that is $Y^* = Y$) in the case that $f\colon \Gamma\to\Gamma$ is an expanding irreducible train track map. 

We now further subdivide $Y^*$ as follows. First, the vertical $1$--cells are not subdivided.  Each skew $1$--cell of $Y^*$ is $e_-(T)$ for some trapezoid $T$ of $Y^*$, and we may write the top arc of $T$ as a concatenation of skew $1$--cells of $Y^*$.  Let $v_0,v_1,\ldots,v_k$ be the vertices of $e_+(T)$, traversed in that order along $e_+(T)$ (with the $T$--orientation), so that $v_0$ is the top endpoint of $\ell_-(T)$ and $v_k$ is the top endpoint of $\ell_+(T)$. 
We let $u_i = h_T^{-1}(v_i)$, where $h_T \colon e_-(T) \to e_+(T)$ is the map defined by $\flow$, and we subdivide $e_-(T)$ at the points $u_1,\ldots,u_{k-1}$ yielding new $1$--cells $\alpha_i = (u_{i-1},u_i)$ for $1 \leq i \leq k$ (so if $k = 1$, $e_-(T)$ is not subdivided). Lastly, we subdivide the trapezoid $T$ of $Y^*$ into $k$ trapezoids $T_1,\dots,T_k$ by adding vertical $1$--cells $\beta_i$ (i.e., arcs of flowlines) connecting $u_i$ to $v_i$ for each $1\le i < k$; see Figure~\ref{F:cocyle_refinement}. It is straightforward to check that simultaneously applying these subdivisions (of $e_-(T)$ and $T$) for each trapezoid $T$ of $Y^*$ yields a trapezoidal cell structure $\widehat{Y}$ that subdivides $Y$. 
\end{defn}

If $Y'$ is a subdivision of a trapezoidal cell structure $Y$, we say that a $1$--cell  $\sigma$ of $Y$ \emph{becomes subdivided} in $Y'$ if there exists a $1$--cell $\sigma'$ of $Y'$ such that $\sigma'\subsetneq \sigma$. Equivalently, this is the case if and only if $\sigma$ is not itself a $1$--cell of $Y'$. The same terminology applies to $2$--cells of $Y$, and we note that some $2$--cells of $Y$ may not become subdivided in the standard subdivision $\widehat{Y}$ (namely those trapezoids $T$ that are not contained in an invariant band and for which $e_+(T)$ consists of a single skew $1$--cell).

\begin{lemma}
\label{L:eventual_subdivision}
Let $Y_0$ be a trapezoidal subdivision of $X_\f$, and let $Y_0,Y_1,\dotsc$ be the sequence of standard subdivisions defined recursively by $Y_{n+1} = \widehat{Y_n}$. Then every skew $1$--cell of $Y_n$ eventually becomes subdivided in some further subdivision $Y_m$ with $m > n$.
\end{lemma}
\begin{proof}
Suppose that a skew $1$--cell $\sigma$ of $Y_n$ remains unsubdivided in $Y_m$ for all $m > n$. Then by nature of the standard subdivision procedure, $\sigma$ must equal $e_-(T_0)$ for a trapezoid $T_0$ of $Y_n$ whose top arc $\zeta(T_0)e_+(T_0)$ is a \emph{single} skew $1$--cell of $Y_n$ that also remains unsubdivided in $Y_m$ for all $m> n$. In particular, we must have $\zeta(T_0)e_+(T_0) = e_-(T_1)$ for another trapezoid satisfying the same property. Continuing in this way, we construct an infinite sequence of trapezoids $T_0,T_1,T_2,\ldots$ whose top arcs each consist of a single skew $1$--cell $\zeta(T_i)e_+(T_i) = e_-(T_{i+1})$ that remains unsubdivided in $Y_m$ for all $m> n$. Such an infinite sequence must eventually become periodic yielding an invariant band in $Y_n$. By construction, every skew $1$--cell in this invariant band subdivided in $(Y_n)^*$ and therefore also in $Y_{n+1}$. But this contradicts the fact that each skew $1$--cell $e_-(T_i)$ remains unsubdivided in $Y_m$ for all $m > n$.
\end{proof}

\subsection{Cocycle refinement}
\label{S:refinement}

If $Y$ is any trapezoidal cell structure on $X_\f$ and $Y'$ is a further trapezoidal subdivision of $Y$, then any $1$--cell $\sigma$ of $Y$ may be viewed as a cellular $1$--chain of $Y'$.  With this in mind, we say that a cocycle $z' \in Z^1(Y';\R)$ is a \emph{refinement} of $z \in Z^1(Y;\R)$ if $z'(\sigma) = z(\sigma)$ for every $1$--cell $\sigma$ of $Y$.  

The standard subdivision procedure from Definition~\ref{D:standard
  subdivision} has an associated method of refinement for positive
$1$--cocycles. Before describing this refinement, we first introduce some notation.
\begin{defn}[Skew ratio]
\label{D:skew_ratio}
Let $Y$ be a trapezoidal subdivision of $X_\f$ and let $z\in Z^1(Y;\R)$ be a positive cocycle on $Y$. For a trapezoid $T$ of $Y$ whose top arc is a $1$--chain $e_+(T) = \zeta(T)(\sigma_1+\dotsb + \sigma_k)$ we define the \emph{skew ratios} of $T$ to be the numbers
\[ 0 < \rho_i(T) = \frac{z(\sigma_i)}{\zeta(T)z(e_+(T))} \leq 1.\]
As $z(\sigma_1) + \dotsb + z(\sigma_k) = \zeta(T)z(e_+(T))$, we observe that $\rho_1(T) +  \dotsb + \rho_k(T) = 1$. The \emph{maximal ratio} for $T$ is then defined to be $\rho(T) = \max_{1\leq i \leq k} \rho_i(T)$, with the convention that $\rho(T) = 0$ in the case that $k=1$. We thus have $0 \le \rho(T) < 1$ for every trapezoid $T$ of $Y$. Finally, we set
\[\rho(Y) = \max_T \rho(T) < 1,\]
where the maximum is over all trapezoids $T$ of $Y$. To emphasize the dependence on the cocycle $z$, these quantities will sometimes be notated more elaborately as $\rho_i(T,z)$, $\rho(T,z)$, and $\rho(Y,z)$.
\end{defn}

\begin{defn}[Standard refinement]
\label{D:trap ratio and standard refinement}
Let $z\in Z^1(Y;\R)$ be a positive cocycle on a trapezoidal subdivision $Y$ of $X_\f$. The \emph{standard refinement} of $z$ is the positive cocycle $\widehat{z}\in Z^1(\widehat{Y};\R)$ defined as follows: First consider a trapezoid $T$ contained in an invariant band of $Y$. Then in the auxiliary subdivision $Y^*$, the bottom arc $e_-(T)$ is subdivided into two skew $1$--cells $\sigma_1,\sigma_2$ and the trapezoid $T$ is subdivided into two trapezoids by the addition of a new vertical $1$--cell $\beta$. Correspondingly, we define an auxiliary refinement $z^*\in Z^1(Y^*;\R)$ of $z$ by declaring
\begin{equation}
\label{eq:auxiliary_refinement}
z^*(\sigma_1) = z^*(\sigma_2) = \frac{z(e_-(T))}{2}\qquad\text{and}\qquad z^*(\beta) = \frac{z(\ell_-(T)) + z(\ell_+(T))}{2}
\end{equation}
for the new $1$--cells arising in any such trapezoid and leaving $z$ unchanged on the remaining $1$--cells of $Y$. It is straightforward to check that $z^*$ then defines a positive refinement of $z$.

Next consider any trapezoid $T$ of $Y^*$ with top arc a $1$--chain $e_+(T) = \zeta(T)(\sigma_1+\dotsb + \sigma_k)$ in $Y^*$. As in Definition~\ref{D:standard subdivision} we let $\alpha_1,\ldots,\alpha_k$ be the new skew $1$--cells for the subdivision of $e_-(T)$ in $\widehat{Y}$, and $\beta_1,\ldots,\beta_{k-1}$ the new vertical $1$--cells of $\widehat{Y}$ inside $T$. These are indexed so that $\beta_i$ connects the terminal vertex of $\alpha_i$ to the terminal vertex of $h_T(\alpha_i) = \zeta(T)\sigma_i$. We then construct the refinement $\widehat{z}$ of $z^*$ by leaving the value unchanged on the vertical $1$--cells of $\ell_\pm(T)$ (which are again vertical $1$--cells of $\widehat{Y}$), and for each $1 \leq i \leq k$ declaring
\begin{equation}
\label{eq:standard_refinement_skew}
\widehat{z}(\alpha_i) = \rho_i(T) z^*(e_-(T))
\end{equation}
and
\begin{equation}
\label{eq:standard_refinement_vert}
\widehat{z}(\beta_i) = z^*(\ell_+(T))(\rho_1(T) + \dotsb + \rho_i(T)) + z^*(\ell_-(T))(\rho_{i+1}(T) + \dotsb + \rho_k(T)).
\end{equation}
It is straightforward to see that this procedure defines a positive cocycle that refines $z^*$ and thus also $z$. See Figure~\ref{F:cocyle_refinement} for an illustration of this refinement procedure.
\end{defn}

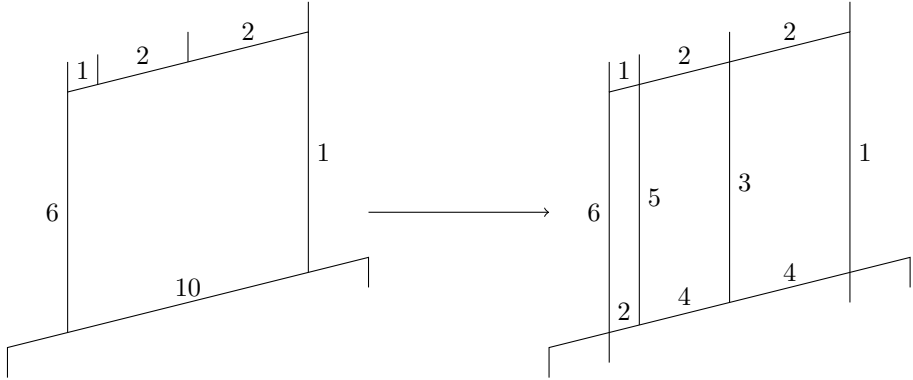
\begin{figure}[htb]
\begin{center}
\begin{tikzpicture}[scale=.8]
\draw (4,1)--(4,5)--(0,4)--(0,0);
\draw (-1,-.75)--(-1,-.25)--(5,1.25)--(5,.75);
\draw (0,4)--(0,4.5);
\draw (.5,4.125)--(.5,4.625);
\draw (2,4.5)--(2,5);
\draw (4,5)--(4,5.5);
\node at (-.25,2) {$6$};
\node at (2,.75) {$10$};
\node at (.25,4.35) {$1$};
\node at (1.25,4.6) {$2$};
\node at (3,5) {$2$};
\node at (4.25,3) {$1$};
\draw[->] (5,2)--(8,2);
\draw (13,.5)--(13,5)--(9,4)--(9,-.5);
\draw (8,-.75)--(8,-.25)--(14,1.25)--(14,.75);
\draw (9,4)--(9,4.5);
\draw (9.5,.125)--(9.5,4.625);
\draw (11,.5)--(11,5);
\draw (13,5)--(13,5.5);
\node at (8.75,2) {$6$};
\node at (9.25,4.35) {$1$};
\node at (10.25,4.6) {$2$};
\node at (12,5) {$2$};
\node at (9.25,.35) {$2$};
\node at (10.25,.6) {$4$};
\node at (12,1) {$4$};
\node at (13.25,3) {$1$};
\node at (9.75,2.25) {$5$};
\node at (11.25,2.5) {$3$};
\end{tikzpicture}
\caption{A local picture of the subdivision and cocycle refinement procedure for a trapezoid $T$ when $e_+(T)$ consists of more than one skew $1$--cell.  The numbers along the $1$--cells on the left represent the values of some positive cocycle $z$ and the numbers along the $1$--cells on the right represent the values of the refinement $\widehat z$.   In this example, we have assumed for simplicity that the trapezoids above $T$ are not subdivided.  In general the values along the top will also be ``distributed'' among the $1$--cells in the respective subdivisions.}
\label{F:cocyle_refinement}
\end{center}
\end{figure}

Notice that skew ratios are essentially preserved under the standard refinement procedure:

\begin{lemma} \label{L:skew ratio persistence}
Let $z\in Z^1(Y;\R)$ be a positive cocycle on a trapezoidal subdivision $Y$ of $X_\f$. Then for the standard subdivision $\widehat{Y}$ and corresponding refinement $\widehat{z}\in Z^1(\widehat{Y};\R)$ we have
\[\rho(\widehat{Y},\widehat{z}) \leq \max\{\rho(Y,z),1/2\}.\]
\end{lemma}
\begin{proof}
First consider a trapezoid $T^*$ of $Y^*$. If $T^*$ is a new trapezoid that was created while constructing the subdivision $Y^*$, then its top arc $e_+(T^*)$ necessarily consists of a single skew $1$--cell of $Y^*$ and so we have have $\rho(T^*,z^*) = 0$ in this case. Otherwise $T^* = T_0$ for some trapezoid $T_0$ of $Y$ and we have $\rho(T^*,z^*) = \rho(T_0,z)$ except in the case that the $1$--chain $e_+(T_0) = \zeta(T_0)(\sigma_1 + \dotsb \sigma_k)$ in $Y$ contains a $1$--cell $\sigma_i$ that becomes subdivided into two $1$--cells $\sigma_i'$ and $\sigma_i''$ in $Y^*$. In that case by \eqref{eq:auxiliary_refinement} and Definition~\ref{D:skew_ratio} we have
\[z^*(\sigma_i') = z^*(\sigma_i'') = \frac{z(\sigma_i)}{2} =  \frac{1}{2}\rho_i(T_0,z)\zeta(T_0)z(e_+(T_0)) = \frac{1}{2}\rho_i(T_0,z)\zeta(T^*)z^*(e_+(T^*)).\]
If $k = 1$, so that $\rho_i(T_0,z) =1$, then $\rho(T^*,z^*) = 1/2$,
and otherwise it follows that $\rho(T^*,z^*) \leq \rho(T_0,z)$. Whence we have $\rho(T^*,z^*) \leq \max\{\rho(Y,z),1/2\}$ for every trapezoid $T^*$ of $Y^*$.

Next suppose that $T$ is a trapezoid of $\widehat{Y}$ and that  we have $e_+(T) = \zeta(T)(\alpha_1 + \dotsb + \alpha_k)$ with $k \geq 2$. Notice that, by construction, the top arc $e_+(T)$ of $T$ consists of a \emph{single} $1$--cell $\alpha$ of $Y^*$. Let $T^*$ be the trapezoid of $Y^*$ with $e_-(T^*) = \alpha$.  Since the $1$--cell $\alpha$ evidently becomes subdivided in $\widehat{Y}$, by definition of the standard subdivision we must have $e_+(T^*) = \zeta(T^*)(\sigma_1 + \dotsb + \sigma_k)$, as a $1$--chain in $Y^*$, where $h_{T^*}$ sends $\alpha_i$ to $\sigma_i$.  From \eqref{eq:standard_refinement_skew} we have
\[ \widehat{z}(\alpha_i) = \rho_i(T^*,z^*) z^*(e_-(T^*)).\]
By the refinement property, we also have $z^*(e_-(T^*)) = \widehat{z}(\alpha) = \zeta(T)\widehat{z}(e_+(T))$. Therefore
\[ \rho_i(T,\widehat{z}) = \frac{\widehat{z}(\alpha_i)}{\zeta(T)\widehat{z}(e_+(T))}  = \rho_i(T^*,z^*) \leq \rho(T^*,z^*) \leq \max\{\rho(Y,z),1/2\}\]
as required.
\end{proof}

Lemma~\ref{L:skew ratio persistence}, Lemma~\ref{L:eventual_subdivision} and the definition of standard refinement now give the following:

\begin{lemma}
\label{L:skews_tend_to_zero}
Let $Y_0$ be a trapezoidal subdivision of $X_\f$ and $z_0\in Z^1(Y_0;\R)$ a positive cocycle. Consider the sequence of standard subdivisions $Y_0,Y_1,\dotsc$ with corresponding cocycle refinements $z_0,z_1,\dotsc$ defined recursively by $Y_{n+1} = \widehat{Y_n}$ and $z_{n+1} = \widehat{z_n}$. Then the numbers
\[S_n =\max\{z_n(\sigma) \mid \text{ $\sigma$ is a skew $1$--cell of $Y_n$}\}\]
tend to zero as $n\to \infty$.
\end{lemma}
\begin{proof}
Let $\delta = \max\{\rho(Y_0,z_0),1/2\} < 1$. Applying Lemma~\ref{L:skew ratio persistence} inductively, we see that $\rho(T,z_n)\leq \delta$ for all trapezoids $T$ of $Y_n$ and for all $n\geq 0$. According to \eqref{eq:auxiliary_refinement} and \eqref{eq:standard_refinement_skew}, whenever a skew $1$--cell $\sigma = e_-(T)$ of $Y_n$ becomes subdivided in $Y_{n+1}$, each resulting subcell $\sigma'\subsetneq \sigma$ satisfies
\[z_{n+1}(\sigma') \leq \max\{\rho(T),1/2\}z_n(e_-(T)) \leq \delta z_n(\sigma).\]
By Lemma~\ref{L:eventual_subdivision}, for each $n$ there exists $m > n$ so that every $1$--cell of $Y_n$ has become subdivided in $Y_m$. The above inequality then implies $S_m \leq \delta S_n$ and the result follows.
\end{proof}

\subsection{An unconstrained refinement}
\label{S:unconstraint}

Our construction of the perturbed fibration $\eta_z \colon X_\f \to \mathbb{S}^1$ associated to a primitive integral element $[z] \in \A$ will rely on finding a subdivision of $X_\f$ and a positive refinement of $z$ in which every trapezoid enjoys the following property.
\begin{defn}[Unconstrained trapezoids]
\label{D:unconstrained}
Let $Y$ be a trapezoidal subdivision of $X_\f$ and let $z\in Z^1(Y;\R)$ be a $1$--cocycle.  A trapezoid $T$ of $Y$ is said to be \emph{unconstrained by $z$} if
\[\max\big\{0,\,\,z(e_-(T))\big\} \leq \min\big\{z(\ell_-(T)),\,\, z(\ell_-(T)) + z(e_+(T))\big\}.\]
\end{defn}

We note that when $z$ is positive, the above definition reduces to $z(e_-(T))\leq z(\ell_-(T))$ since the cocycle condition $z(\partial T) = 0$ (see \eqref{eq:trapezoid_boundary}) immediately yields
\[ 0 <  z(e_-(T)) < z(\ell_+(T)) + z(e_-(T))  = z(\ell_-(T)) + z(e_+(T)) .\]
Therefore, in this case a trapezoid is automatically unconstrained whenever $z(e_+(T)) < 0$. In particular, we see that trapezoids with $\zeta(T) = -1$ are unconstrained by {\em every} positive $1$--cocycle.

We now employ the constructions from \S\S\ref{S:subdivision}--\ref{S:refinement} to find such a refinement. For the proof, we recall from \S\ref{S:trapezoidal_subdivisions} that $X_\f$ may contain degenerate trapezoids and that the sign of every degenerate trapezoid $T$ is $\zeta(T) = -1$ (Remark~\ref{rem:degenerate_trapezoids}).

\begin{proposition} \label{P:standard to unconstrained}
Let $Y_0$ denote the trapezoidal cell structure $X_\f$ and let $z_0 = z\in Z^1(X_\f;\R)$ be a positive cocycle. Consider the sequence $Y_0,Y_1,\dotsc$ of standard subdivisions defined recursively by $Y_{n+1} = \widehat{Y_n}$ with corresponding cocycle refinements $z_{n+1} = \widehat{z_n}$. Then there exist some $n \geq 0$ so that every trapezoid $T$ of $Y_n$ is unconstrained by $z_n \in Z^1(Y_n;\R)$.
\end{proposition}
\begin{proof}
Let $D$ denote the set of degenerate trapezoids of $X_\f$. Any $T' \in D$ has $\zeta(T') = -1$, and furthermore any trapezoid $T$ of a subdivision $Y_n$ which is contained in some $T' \in D$ also has $\zeta(T) = -1$ (see Remark~\ref{rem:degenerate_trapezoids} and the last paragraph of \S\ref{S:trapezoidal_subdivisions}). Therefore, any such $T$ is unconstrained by $z_n$ as the calculation following Definition \ref{D:unconstrained} demonstrates.  So we need only consider trapezoids $T$ not contained in any trapezoid in $D$.

Let $M>0$ be the minimum of $z(\beta)$ over all vertical $1$--cells of $X_\f$. We claim that for $n\ge 0$ every vertical $1$--cell $\beta$ of $Y_n$ that is not contained in the interior of some $T'\in D$ also has $z_n(\beta)\geq M$. This is true for $n = 0$ by definition of $M$; arguing by induction we assume it is true for some $n \geq 0$ and prove it for $n+1$. The condition holds for all vertical $1$--cells of $Y_{n+1}$ that are also vertical $1$--cells of $Y_n$ by induction hypothesis. Thus it suffices to consider a new vertical $1$--cell of $Y_{n+1}$ that is not contained in any trapezoid in $D$ and is not a $1$--cell of $Y_n$. By construction, such a $1$--cell can only arise as a vertical arc $\beta_i$ connecting $e_-(T)$ to $e_+(T)$ for some trapezoid $T$ of $Y_n$ that is not contained in any $T'\in D$. We then have $z_n(\ell_\pm(T))\geq M$ since the $1$--cells comprising these chains are not contained in the interior of any $T'\in D$. As $\rho_1(T) +  \dotsb + \rho_k(T) = 1$, equations \eqref{eq:auxiliary_refinement} and \eqref{eq:standard_refinement_vert} show that $z_{n+1}(\beta_i)$ is a convex combination of $z_n(\ell_+(T)) \geq M$ and $z_n(\ell_-(T)) \geq M$, and hence $z_{n+1}(\beta_i) \geq M$, as required.

By Lemma~\ref{L:skews_tend_to_zero} we may now choose $n$ so that every skew $1$--cell $\sigma$ of $Y_n$ satisfies $z_n(\sigma) < M$. For any trapezoid $T$ of $Y_n$ that is not contained in a trapezoid of $D$, we then have $z_n(\ell_\pm(T))\geq M$ by the preceding paragraph. Therefore
\[ z_n(e_-(T)) < M \leq z_n(\ell_-(T)),\]
showing that $T$ is unconstrained by $z_n$ and completing the proof.
\end{proof}

\subsection{Perturbed fibrations} \label{sect:allowable perturbations}

We now utilize the trapezoidal subdivisions and cocycle refinements developed in \S\S\ref{S:trapezoidal_subdivisions}--\ref{S:unconstraint} to construct perturbed fibrations $\eta_z\colon X_\f\to \mathbb{S}^1$ for positive $1$--cocycles representing primitive integral classes in $\A$. The proof will also utilize the local model neighborhoods introduced in \S\ref{sect:local models}. For any local model $\M \subset X_\f$, we denote a lift to the universal covering $\pi\colon \R \to \mathbb{S}^1$ of the restriction of $\fib_0$ to $\M$ by
\[ \fib_0^\M\colon \M \to \R\]
and write $I_\M = \fib_0^\M(\M)$.

\begin{lemma} \label{L:constructing f_z}
Let $z \in Z^1(X_\f; \R)$ be a positive cocycle representing an integral class $u \in \A$. Then there exists a map $\fib_z\colon X_\f \to \mathbb{S}^1$ and a finite covering $\MM$ of $X_\f^{(1)}$ by local model neighborhoods with the following properties.
\begin{enumerate}
\item $(\fib_z)_*  =  u \colon \pi_1(X_\f) \to \pi_1(\mathbb{S}^1) = \Z$ in $H^1(X_\f;\R)$.
\item For each local model $\M \in \MM$, there is an orientation preserving injective local diffeomorphism $h_\M\colon I_\M \to \R$ so that the restriction $\fib_z|_\M$ satisfies
\[ \fib_z|_\M = \pi \circ h_\M \circ \fib_0^\M.\]
\item The restriction of $\fib_z$ to any semiflow line is a local diffeomorphism and the restriction to the interior of each trapezoid is a submersion.
\item For some $x_0 \in \mathbb{S}^1$, $\fib_z^{-1}(x_0) \cap X_\f^{(1)} \subset (X_\f^{(1)} \setminus X_\f^{(0)})$ consists of a finite set of points which flow into periodic orbits of $\flow$.  Moreover, if  $\f$ is an expanding irreducible train track map, then all points of  $\fib_z^{-1}(x_0) \cap X_\f^{(1)}$ eventually flow into vertical $1$--cells of $X_\f$.
\end{enumerate}
\end{lemma}
Part (2) of the lemma basically says that in each local model $\M$, the map $\fib_z$ factors through the restriction to $\M$ of $\fib_0$ (but to make this correct, we need to use $\fib_0^\M$ instead of the restriction of $\fib_0$ to $\M$ exactly).  As a consequence of (2), we see that the fibers of $\fib_z$ are horizontal in $\M$.
\begin{proof}
The map $\fib_z\colon X_\f \to \mathbb{S}^1$ will be defined so that for any $1$--cell $e$, if we lift the restriction to the universal cover $\widetilde \fib_z|_e\colon  e \to \mathbb R$, then the difference in the values at the endpoints (which is independent of the lift) is exactly $z(e)$.  A calculation as in Equation \eqref{E:cocycle = hom} shows that any such map will satisfy conclusion (1) of the lemma. The plan is to first construct a map on the $1$--skeleton $\fib_z'\colon X_\f^{(1)} \to \mathbb{S}^1$ satisfying this property, so that the restriction to each $1$--cell $e$ is a local diffeomorphism onto its image, and so that (4) is satisfied. Assuming that every trapezoid of $X_\f$ is unconstrained by $z$, we may then extend $\fib_z'$ over the $2$--cells so that (3) is satisfied. Finally, we modify $\fib_z'$ in a neighborhood of $X_f^{(1)}$ and construct the covering $\MM$ so that (2) is satisfied.

By Proposition~\ref{P:standard to unconstrained}, after passing to a trapezoidal subdivision $Y$ of $X_\f$ and a corresponding positive refinement $z'$ of $z$, we may assume that every trapezoid of $Y$ is unconstrained by $z'$. 
Suppose we can prove the lemma for the cell structure $Y$ and corresponding refinement $z'$.  
The first assertion of (4) will then follow from the fact that $ X_\f^{(1)} \subset Y^{(1)}$. Furthermore, in the case that $f$ is an expanding irreducible train track map, all vertical $1$--cells of $Y$ eventually flow into the vertical $1$--cells of $X_\f$ (see Definition~\ref{D:standard subdivision}), so the second assertion of (4) will follow as well. Since any finite cover of the $1$--skeleton of $Y$ also covers the $1$--skeleton of $X_\f$, part (2) will follow, and (3) will follow similarly by using local models near the new vertical $1$--cells. Finally, (1) does not make reference to the cell structure and so will follow automatically.
Thus it suffices to prove the lemma for $Y$ and $z'$; by an abuse of notation we therefore continue to refer to this subdivision and refined cocycle simply as $Y = X_\f$ and $z' = z$, respectively.

Choose a maximal tree $Q \subset X_\f^{(1)}$ and fix some vertex $v_0 \in Q$ to serve as a basepoint.  Define a map $\hat \fib_z\colon Q \to \R$ so that $\hat \fib_z(v_0) = 0$ and so that the restriction to every $1$--cell $e$ in $Q$ is an injective local diffeomorphism and so that
\[ \hat \fib_z(t(e)) - \hat \fib_z(o(e)) = z(e). \]
Because $z$ represents an integral cohomology class, we see that for any $1$--cell $e \in E_0$ which is not in $Q$, although this equation can fail, we still have
\[ \hat \fib_z(t(e)) - \hat \fib_z(o(e)) - z(e) \in \Z. \]
Therefore, composing $\hat \fib_z$ with $\pi$, we obtain a map $\fib_z'\colon Q \to \mathbb{S}^1$ which can we can extend to all of $X_\f^{(1)}$ with the property that for any $1$--cell $e$,  any lift $\widetilde \fib_z'|_e\colon  e \to \mathbb R$ is an injective local diffeomorphism and
\[ \widetilde \fib_z'|_e(t(e)) - \widetilde \fib_z'|_e(o(e)) = z(e). \]

We next modify $\fib'_z\colon X_\f^{(1)}\to \mathbb{S}^1$ by a homotopy rel $X_\f^{(0)}$, maintaining the property that $\fib'_z$ is a local diffeomorphism on each edge, so that on a neighborhood $U$ of $X_\f^{(0)}$ our map $\fib'_z$ factors through the restriction of $\fib_0$ to $U$. More precisely, we arrange that each connected component $U_0$ of $U$ contains a single vertex of $X_\f^{(0)}$ and that for any lift $\widetilde \fib_0\colon U_0\to \R$ of the restriction of $\fib_0$ to $U_0$ there exists a local diffeomorphism $h\colon \R\to \R$ so that
\[\fib'_z\vert_{U_0\cap X_\f^{(1)}} =  \pi \circ h \circ \widetilde{\fib}_0\vert_{U_0\cap X_\f^{(1)}}.\]
This modification ensures that two points of $U_0\cap X_\f^{(1)}$ are identified by $\fib'_z$ if and only if they are also identified by $\widetilde \fib_0$.

After shrinking $U$ if necessary, we may then choose a point $x_0\in \mathbb{S}^1 \setminus \fib'_z(U)$. Note that when $\f \colon \Gamma \to \Gamma$ is an expanding irreducible train track map, the set $\cup_{k \geq 1} f^{-k}(V\Gamma)$ is dense in $\Gamma$ (see the proof of Theorem \ref{T:topological graph to graph}).  Therefore, the set of points in $X_\f$ that eventually flow into vertical $1$--cells by $\flow$ is dense and these eventually flow into periodic orbits of $\flow$ (see Remark~\ref{R:old v trap}). More generally, if $\mathcal P \subset \Gamma$ is the set of periodic points for $\f$, then $\cup_{k \geq 1} \f^{-k}(\mathcal P)$ is also dense (as can be seen from the linear structure of $\f$), and hence the set of points of $X_\f$ that eventually flow into periodic orbits is always dense. Therefore, we may modify $\fib'_z\colon X_\f^{(1)}\to \mathbb{S}^1$ by a homotopy rel $U\supset X_\f^{(0)}$ so that $(\fib'_z)^{-1}(x_0)$ (which is a finite subset of $X_\f^{(1)}\setminus U$) consists of points that eventually flow into periodic orbits of $\flow$ (and that furthermore flow into vertical $1$--cells of $X_\f$ in the case that $\f$ is an expanding irreducible train track map). The extension $\fib_z\colon X_\f\to \mathbb{S}^1$ of $\fib'_z$ that we construct below will agree with $\fib'_z$ on $X_\f^{(1)}$ and will therefore satisfy condition (4).

We then extend $\fib'_z$ over the $2$--cells of $X_\f$ as follows. Consider a trapezoid $T$ of $X_\f$, and let $v\in X_\f^{(0)}$ denote the common initial vertex of $e_-(T)$ and $\ell_-(T)$ (recall that $e_-(T)$ is oriented from $\ell_-(T)$ to $\ell_+(T)$). Lifting $\fib'_z\vert_{\partial T}$ to the universal cover gives us a map $\widetilde \fib'_z\vert_{\partial T} \colon \partial T \to \R$ which, after composing with a translation, we may assume sends $v$ to $0$. By construction, $\widetilde \fib'_z\vert_{\partial T}$ then maps $e_-(T)$ and $e_+(T)$ locally diffeomorphically onto the intervals
\[ [0,z(e_-(T))] \qquad\text{ and }\qquad [z(\ell_-(T)), z(\ell_-(T)) + z(e_+(T))],\]
respectively (in the case $\zeta(T) = -1$, note that the endpoints of this latter interval have been written out of order). Since $T$ is unconstrained by $z$, these intervals have disjoint interiors. Therefore the map $\widetilde \fib'_z\vert_{\partial T}$ does not identify any interior point of $e_-(T)$ with any interior point of $e_+(T)$. Rather, $\widetilde \fib'_z\vert_{\partial T}$ identifies each point of $e_-(T)$ with a point of $\ell_-(T)$, and each point of $e_+(T)$ with a point of either $\ell_+(T)$ (in the case $\zeta(T) = 1$) or $\ell_-(T)$ (in the case $\zeta(T) = -1$). It follows that we may extend $\pi\circ\widetilde \fib'_z\vert_{\partial T}$ to a map $\fib'_z\colon \bar{T}\to \mathbb{S}^1$ whose fibers are transverse to the flow. Explicitly, identifying $\bar T$ with the Euclidean trapezoid $\hat{T}$ via the map $\sigma_T\colon \hat{T}\to \bar{T}$ provided by Proposition~\ref{P:trapezoid simple flow}, we extend $\widetilde\fib'_z\vert_{\partial T}$ over $\bar T$ so that the fibers are straight (Euclidean) lines in $\hat{T}$ and define $\fib'_z$ on $\bar T$ as $\pi \circ \widetilde \fib'_z\vert_{\bar T}$ . The unconstrained condition then ensures that each fiber has nonvertical slope and is therefore transverse to the foliation of $\hat T$ by (vertical) flowlines. Extending $\fib'_z$ over every trapezoidal $2$--cell of $X_\f$ in this manner, we obtain the map $\fib'_z\colon X_\f\to \mathbb{S}^1$.

Notice that the fibers of $\fib'_z$ are by construction transverse to $X_\f^{(1)}$. We may therefore modify $\fib'_z$ by a homotopy rel the $1$--skeleton to obtain a map $\fib_z\colon X_\f\to \mathbb{S}^1$ whose fibers remain transverse to the flow and are furthermore horizontal in a neighborhood $W$ of $X_\f^{(1)}$. This map satisfies conclusions (1) and (3) of the lemma.

Finally to prove part (2), we construct the cover $\MM$ of $X_\f^{(1)}$ by local models. For each $0$--cell $v$ of $X_\f$ we use Proposition~\ref{P:local model maps} to find a local model neighborhood $\M_v$ of $v$ contained in $U\cap W$. The local models $\M_v$ and $\M_{v'}$ are then disjoint for distinct $0$--cells (since each component of $U$ contains a single $0$--cell). We next cover each edge by finitely many $1$--cell local models that are each contained in $W$ and let $\MM$ denote the resulting cover.  Since each $\M\in \MM$ is contained in $W$, we see that the fibers of $\fib_z\vert_\M$ are necessarily horizontal. This implies that two points of $\M$ are identified by a lift $\widetilde \fib_z\vert_{\M}$ if and only if they are also identified by $\fib_0^\M$. Therefore the maps $\widetilde \fib_z\vert_\M$ and $(\fib_0^\M)^{-1}$ may be composed to yield a well defined map
\[ h_\M = \widetilde \fib_z|_\M \circ (\fib_0^\M)^{-1}\colon  I_{\M} \to \R.\]
This map $h_\M$ is a local diffeomorphism (as can be seen by restricting to a flowline in $\M$) and thus satisfies conclusion (2) of the lemma.
\end{proof}

\begin{remark}
Alternatively, one could work with {\em continuous $1$--forms} on $X$ as described in \cite{FGS10}.  We do not need any of the technical machinery developed in that setting here, and our situation requires other special considerations, so we decided to avoid adding another layer of terminology.
\end{remark}

\begin{defn}[Associated fibrations and the graph $\Theta_z$]\label{D:Theta_z}
Let  $u\in H^1(X_\f;\R)$ be an integral class represented by a positive $1$--cocycle $z\in Z^1(X_\f;\R)$. A map $\fib_z\colon X_f\to \mathbb{S}^1$ satisfying the conclusions of Lemma~\ref{L:constructing f_z} will be called an {\em associated fibration} for $z$.  
For any point $x \in \mathbb{S}^1$, let $\Theta_{x,z} = \fib_z^{-1}(x)$ be the full preimage of $x$.  By the second and third condition of Lemma~\ref{L:constructing f_z}, the fiber $\Theta_{x,z}$ is a finite topological graph with vertex set equal to the intersection with $X_\f^{(1)}$. The {\em associated graph $\Theta_z$} for this fibration is defined by choosing a point $x_0 \in \mathbb{S}^1$ as in part (4) and letting $\Theta_z := \Theta_{x_0,z}$ as a topological space; as topological graph we equip $\Theta_z$ with the vertex set
\[ V\Theta_z =  \left(  \bigcup_{s \geq 0} \flow_s(V\Theta_{x_0,z}) \right) \cap \Theta_{x_0,z}.\]
 According to part (4), $V \Theta_{x_0,z}$ flows into periodic orbits of $\flow$, and hence $\Theta_z$ is indeed a finite subdivision of $\Theta_{x_0,z}$.  Note that by construction, every edge of $\Theta_z$ is contained in a trapezoidal $2$--cell of $X_\f$ (since this is already true for $\Theta_{x_0,z}$).
\end{defn}

Let $u,z$, $\eta_z$ and $\Theta_z$ be as in Definition~\ref{D:Theta_z} above.
To understand the graph $\Theta_z$,  we need to study some general properties of the graphs $\Theta_{x,z}$ (where $x\in \mathbb{S}^1$ is arbitrary) and how they transform under $\flow$. 
By the third condition of Lemma~\ref{L:constructing f_z}, we can reparameterize the semiflow $\flow_s$ so that the composition  $\R_+ \to \mathbb{S}^1$ of any semiflow line with the projection to the circle has unit speed.  We denote this reparameterization by $\flow^z_s$.  
With this reparameterization we have
\[ \flow^z_s(\Theta_{x,z}) = \Theta_{x+s,z}\]
where $s \in \R$ and $x+s \in \mathbb{S}^1 = \R/\Z$.  By Lemma \ref{L:constructing f_z} (2), given a local model $\M = \M(K,s_0) \in \MM$, there exists $s_0^z$ so that 
\[ \M = \bigcup_{0 \leq s \leq s_0^z} \flow_s^z(K).\]
Furthermore, there is a homeomorphism $h_{\M,z}\colon [0,s_0] \to [0,s_0^z]$ so that
\[ \flow_s(K) = \flow_{h_{\M,z}(s)}^z(K)\quad\text{for all } s\in[0,s_0]\]
and consequently Proposition~\ref{P:local model maps} is valid for $\flow^z$ in place of $\flow$.

\begin{lemma} \label{L:local h.e.}
For any $x \in \mathbb{S}^1$ there exists an $s_x > 0$ such that for all $-\frac{s_x}{2} \leq s \leq s' \leq \frac{s_x}{2}$ the map
\[ \flow_{s'-s}^z \colon \Theta_{x+s,z} \to \Theta_{x+s',z}\]
is a homotopy equivalence.
\end{lemma}
\begin{proof}
For any $x' \in \mathbb{S}^1$ and $0 < s' < 1$ let
\[ N(x',s') = \bigcup_{0 \leq s \leq s'} \flow_s^z(\Theta_{x',z}) = \fib_z^{-1}([x',x'+s']).\]

Now cover the intersections of $\Theta_{x,z}$ with $X_\f^{(1)}$ by the flow-interiors of local models.  By replacing the local models $\M \in \MM$ provided by Lemma~\ref{L:constructing f_z} with smaller local models if necessary, we can assume that $\Theta_{x,z}$ does not intersect the top or bottom of any $\M$ and that for all $\M \neq \M' \in \MM$, we have
\[ \Theta_{x,z} \cap \M \cap \M' = \emptyset.\]
Then we can choose $0 < s_x < 1$ so that for any local model $\M = \M(K,s_0)$ there exist $0 < s_1^\M < s_2^\M < \ldots < s_r^\M < s_0^z$ (with $r = r(\M)$) with
\[ N(x - \tfrac{s_x}2,s_x) \cap \M = \bigcup_{j=1}^r \left( \bigcup_{0 \leq s \leq s_x} \flow_{s_j^\M + s}^z(K) \right). \]
Furthermore, by choosing $s_x$ sufficiently small we can also guarantee that
\[ N(x-\tfrac{s_x}2,s_x) \cap \M \cap \M' = \emptyset \]
for all $\M \neq \M' \in \MM$.

Applying Proposition \ref{P:local model maps} it follows that for all $-\tfrac{s_x}2 \leq s \leq s' \leq \tfrac{s_x}2$ and all $\M \in \MM$, we have
\[ \flow_{s'-s}^z \colon \Theta_{x+s,z} \cap \M \to \Theta_{x + s',z} \cap \M \]
is controlled rel boundary.  Any component of the complement
\[ U \subset N(x - \tfrac{s_x}2,s_x) \setminus \bigcup_{\M \in \MM} \M \]
is entirely contained in some trapezoid $T$.  In $T$ the semiflow is a flow, and so the restriction
\[ \flow_{s'-s}^z \colon \Theta_{x+s,z} \cap U \to \Theta_{x+s',z} \cap U \]
is a homeomorphism.

It follows that $\flow_{s'-s}^z \colon \Theta_{x+s,z} \to \Theta_{x+s',z}$ is locally controlled, and so by Proposition \ref{P:local control h.e.} it is a homotopy equivalence.
\end{proof}

\begin{corollary} \label{C:all h.e.}
For any $x \in \mathbb{S}^1$ and $s \in \R$ we have that
\[ \flow^z_s\colon \Theta_{x,z} \to \Theta_{x+s,z} \]
is a homotopy equivalence.  In particular, the first return map
\[ \f_z = \flow^z_1 \colon \Theta_z \to \Theta_z\]
is a homotopy equivalence.  Furthermore, $\f_z(V \Theta_z) \subset V \Theta_z$.
\end{corollary}
\begin{proof}
We can write $\flow^z_s = \flow_{s_1}^z \circ\flow_{s_2}^z \circ \ldots \circ \flow_{s_k}^z$, where each restriction $\flow_{s_j}^z\colon \Theta_{x_{j-1},z} \to \Theta_{x_j,z}$, for $j = 1,\ldots,k$, is a homotopy equivalence given by Lemma \ref{L:local h.e.} for some $x_0,\ldots,x_k$ with $x_0 = x$ and $x_k = x+s$.   This proves the first (and hence also the second) statement.
The last statement follows from the definition of $V \Theta_z$ (see Definition \ref{D:Theta_z}
) since $\f_z$ is the first return map of $\flow$.
\end{proof}

As a consequence we have the following Corollary (cf.~\cite[Theorem 3.1]{Gau1} and \cite[Lemma 4.1.3]{Wang}).

\begin{corollary} \label{C:perturbed semidirect product}
For any primitive integral element $u \in \A$ represented by a positive $1$--cocycle $z$ with associated fibration $\fib_z$, the fiber $\Theta_z$ is connected and the induced map $\pi_1(\Theta_z) \to \pi_1(X_\f) = G$ is an isomorphism onto the kernel $\ker(u)$ of $u = (\fib_z)_*\colon G \to \mathbb Z$.
\end{corollary}
\begin{proof}  The flow $\flow^z_s\colon \Theta_{x,z} \to
  \Theta_{x+s,z}$ is a homotopy equivalence for all $x,s$, so the
  number of components $n$ of $\Theta_{x,z}$ is constant, independent
  of $x$.  We note that $X = \cup_{x \in \mathbb S^1} \Theta_{x,z}$ and use this to construct a quotient $F\colon X \to X/\!_\sim \,$, where
  $\sim$ is the equivalence relation whose equivalence classes are
  precisely the connected components of $\Theta_{x,z}$ for $x \in
  \mathbb{S}^1$.  There is a larger equivalence relation $\sim'$ whose
  equivalence classes are the sets $\Theta_{x,z}$ for $x \in \mathbb{S}^1$, and
  this quotient $X/\!_{\sim'}$ is just $\mathbb{S}^1$ itself.  The quotient $X
  \to \mathbb{S}^1$ factors as $X \to X/\!_\sim \, \to \mathbb{S}^1$, and the $X/\!_\sim
  \, \to \mathbb{S}^1$ is everywhere $n$-to-$1$.  In fact, this map is an $n$--fold covering map, $X/\!_\sim \, = \mathbb{S}^1 \to \mathbb{S}^1$.  Therefore, the homomorphism $u\colon G \to \Z$ factors through the induced map of this cover $F_*\colon G \to \pi_1(X/\!_\sim \,)$.  But the induced map from $X/\!_\sim \, = \mathbb{S}^1 \to \mathbb{S}^1$ is multiplication by $\pm n$, and hence $u = \pm n F_*$.  Since $u$ is primitive, it follows that $n = 1$, so $\Theta_{x,z}$ is connected for all $x$.

For the second statement, apply van Kampen's Theorem to calculate $\pi_1(X_\f)$ as a semi-direct product $\pi_1(\Theta_z) \rtimes \mathbb Z$, and observe that the projection onto $\mathbb Z$ is $(\fib_z)_*$.
\end{proof}

\subsection{$E_0$--class and the proof of Theorem~\ref{T:folded basics}} \label{S:proof_of_theorem_B}

In this section we will establish Theorem~\ref{T:folded basics} from the introduction.
Recall  (c.f., Convention~\ref{conv:nice}) that for any $\fee\in \Out(F_N)$, where $N\ge 2$, given a tame (linear) graph-map $f \colon \Gamma \to \Gamma$ that is a topological representative for $\fee$ along with some folding sequence as in $(\ref{E:foldings})$, we have constructed the corresponding folded mapping torus $X_f$ for $\fee$. As noted in Remark~\ref{R:tame}, for any $\fee\in \Out(F_N)$ there exists a tame representative $f$ as above.

Recall also  that in Definition~\ref{def:E_0} we defined a subgraph $\mathcal E_0$ of $X_\f^{(1)}$, that in Definition~\ref{D:E_0-class} we defined the $E_0$-class $\epsilon\in H_1(X_\f;\R)$, and that sections of the semiflow $\flow$ were defined in \S\ref{sec:intro}.

We now prove:
\begin{theorem:folded basics}
Let $G_\fee = F_N\rtimes_\fee \Z$ be the mapping torus group of $\fee \in \Out(F_N)$ (where $N\ge 2$), let $f \colon \Gamma \to \Gamma$ be a tame graph-map which is a topological representative of $\fee$, and let $X_f$ be a folded mapping torus for $f$ with the associated semiflow $\flow$, positive cone $\A_{X_f} \subset H^1(X_f;\R)$, and $E_0$--class $\epsilon \in H_1(X_f;\R)$.  Given a positive cellular $1$--cocycle $z \in Z^1(X_f;\R)$ representing a primitive integral element $u \in \A_{X_f} \subset H^1(X_f;\R) = H^1(G_\fee;\R)$, there exists a map
$\fib_z \colon X_f \to \mathbb{S}^1$ and a fiber $\Theta_z = \fib_z^{-1}(x_0)$ for some $x_0 \in \mathbb{S}^1$ so that
\begin{enumerate}
\item $\Theta_z$ is a finite connected topological graph such that the
  inclusion $\Theta_z\subset X_f$ is $\pi_1$--injective and such that
  $\pi_1(\Theta_z)=\ker(u)\le \pi_1(X_f)=G_\fee$.
\item $u(\epsilon) = \chi(\Theta_z) = 1 - \rk(\ker(u))$,
\item $\Theta_z$ is a section of $\flow$,
\item the first return map $f_z \colon \Theta_z \to \Theta_z$ of $\flow$ is a homotopy equivalence with $f(V\Theta_z)\subset V\Theta_z$ and $(f_z)_*$ represents $\fee_u \in \Out(\ker(u))$.
\end{enumerate}
\end{theorem:folded basics}

\begin{proof} The existence of $\fib_z$ comes from Lemma \ref{L:constructing f_z}, and property (1) is a consequence of Corollary \ref{C:perturbed semidirect product}.  Property (3) is immediate and (4) follows from Corollary \ref{C:all h.e.}.  

All that remains is to prove (2).  Since $\chi(\Theta_z) =
1-\rk(\ker(u))$, to establish (2) it suffices to prove $u(\epsilon) = \chi(\Theta_z)$.

Write $\epsilon = \sum_{e \in E_0} c_e e \in Z_1(X_\f;\R)$ where $c_e = (2-d(e))/2$.   Let $x$ be any point not in the $\fib_z$--image of the zero-skeleton $X_\f^{(0)}$.    Since $(\fib_z)_* = u$ and $z$ is a cocycle that represents $u$, we can calculate $z(\epsilon)$ by pushing $\epsilon$ forward by $\fib_z$, and pairing this $1$--cocycle on $\mathbb{S}^1$ with the Poincar\'e dual of the fundamental class.  This class is represented in deRham cohomology, $H^1(\mathbb{S}^1;\R) = H^1_{dR}(\mathbb{S}^1;\R)$ by any $1$--form $\omega$ on $\mathbb{S}^1$ with total volume $1$.  Pick such a $1$--form $\omega$ with support contained in an interval about $x$ disjoint from $\fib_z(X_\f^{(0)})$.

Next, recall that the restriction of $\fib_z$ to any $e \in E_0$ is an orientation preserving local diffeomorphism.  In particular, we can consider the restriction of $\fib_z$ to $e$ as a smooth, singular $1$--chain in $\mathbb{S}^1$, denoted $e_z$, which is something that we can integrate $\omega$ over.  By our choice of $\omega$ and since $\fib_z$ is orientation preserving on $e$, we have
\[ \int_{e_z} \omega = \left| \fib_z^{-1}(x) \cap e \right|.\]
Thus, from the previous paragraph we have
\[z(\epsilon) = \sum_{e \in E_0} c_e \int_{e_z} \omega = \frac{1}{2} \sum_{e \in E_0} (2-d(e)) \left| \fib_z^{-1}(x) \cap e \right|.
\]

Now observe that $\fib_z^{-1}(x) \cap \mathcal E_0$ is precisely the
set of vertices of $\Theta_{x,z}$ of valence greater than $2$.  Indeed, if $y \in \fib_z^{-1}(x)
\cap e$, then the valence $\val(y)$ in $\Theta_{x,z}$ is precisely $d(e)$.
Therefore, from this fact and the previous equation we have
\[
z(\epsilon) =  \frac{1}{2}\sum_{y \in V \Theta_{x,z}} (2-\val(y)) = \chi(\Theta_{x,z}).
\]

By definition, $\Theta_z$ is a subdivision of $\Theta_{x_0,z}$ for some $x_0$ not in the image of $X_\f^{(0)}$. Hence $z(\epsilon)= \chi(\Theta_{x_0,x}) = \chi(\Theta_z)$, as required.
\end{proof}

\section{Train track maps and full irreducibility}

Recall that if $\fee\in \Out(F_N)$ is represented by an expanding irreducible train-track map $f\colon\Gamma\to\Gamma$, then $f$ is tame by Lemma~\ref{L:tt-tame} and so automatically satisfies Convention~\ref{Conv:simplicial for construction}. Thus given any folding sequence \eqref{E:foldings} for $f$ and compatible metric structure on $\Gamma$ (c.f., Convention~\ref{conv:nice}), we can construct the corresponding folded mapping torus $X_f$ for $\fee$. 

In this section we prove Theorem~\ref{T:new iwip main} from the Introduction:
\begin{theorem:new iwip main}
Let $\fee \in \Out(F_N)$ (where $N\ge 2$) be an outer automorphism represented by an expanding irreducible train track  map $f
\colon \Gamma \to \Gamma$, let
$X_f$ be a folded mapping torus, and $\flow$ the associated semiflow.
Let $u \in \A_{X_f} \subset H^1(G_\fee;\R) = H^1(X_f;\R)$ be a primitive integral element which is
represented by a positive $1$--cocycle $z \in Z^1(X_f;\R)$, and let $\Theta_z
\subset X_f$ be the topological graph from Theorem \ref{T:folded
  basics} so that the first return map $f_z \colon \Theta_z \to
\Theta_z$ is a homotopy equivalence representing
$\fee_u$. Then:
\begin{enumerate}
\item The graph $\Theta_z$ can be chosen so that the first return map $f_z \colon \Theta_z \to
\Theta_z$ is an expanding irreducible train-track map.
\item If we additionally assume that $\fee$ is hyperbolic and fully irreducible, then $\fee_u\in \Out(\ker(u))$ is also hyperbolic and fully irreducible.
\end{enumerate}
\end{theorem:new iwip main}

The proof of Theorem~\ref{T:new iwip main}, will be given after establishing a sequence of lemmas.  First we make the following
\begin{conv}
For the remainder of the section, for the purposes of proving part (1) of Theorem \ref{T:new iwip main}, we assume that $f\colon \Gamma \to \Gamma$ is an expanding irreducible train track map representing $\fee \in \Out(F_N)$.  
We do not assume $\fee$ is hyperbolic and fully irreducible until specifically stated; see Convention \ref{conv:iwiphyp}.

Note that by Lemma~\ref{L:tt-tame} the map $f$ is tame. We equip $\Gamma$ with the linear and metric structures  making $f$ into a combinatorial graph map (see Lemma~\ref{L:linear maps expanding on all scales} and Corollary~\ref{cor:find_combinatorial_metric}), and then construct the corresponding folding line and the folded mapping torus $X_\f$. 
We assume $z$ is a positive cocycle representing a primitive integral class $u \in \A = \A_{X_\f}$ and that $\fib_z$ is an associated fibration defining the connected topological graph $\Theta_z$ and $\Theta_{x,z}$ for all $x \in \mathbb S^1$---see Definition \ref{D:Theta_z}.  Recall that the time $1$ map $\flow_1^z$ of the reparameterized flow is the first return map for $\Theta_z$ and is denoted 
\[ \f_z\colon  \Theta_z \to \Theta_z.\]
From Corollary \ref{C:all h.e.}, $\f_z$ is a homotopy equivalence and $\f_z(V \Theta_z) = V \Theta_z$.
\end{conv}

By Corollary \ref{C:perturbed semidirect product}, the inclusion
induces an isomorphism $\pi_1(\Theta_z) \cong \ker(u)$, which we use
to identify $\ker(u) = \pi_1(\Theta_z)$.  Then $\f_z$ induces the
monodromy $(\f_z)_* = \fee_u \in \Out(\ker(u))$.

Thus $f_z$ satisfies all the requirements of being a topological representative of $\fee_u$ (in the sense of Definition~\ref{D:TR}) except that we don't know yet that $f_z$ is a topological graph-map.  We will show that $f_z$ is a topological graph-map, and that,  with respect to an appropriate linear structure on $\Theta_z$, the map $\f_z$ is an expanding irreducible train track map. When $\fee$ is fully irreducible, we will find that $\f_z$ is moreover a weakly clean train track map.  The fact that $\fee_u$ is fully irreducible will then follow from Theorem~\ref{T:clean}. 

We begin our analysis of $\f_z$ with the following technical lemma.

\begin{lemma} \label{L:forward immersion}
Suppose $\alpha\colon (-\delta,\delta) \to \sigma$ is an arc in the interior of a $2$--cell $\sigma$ of $X_\f$ transverse to flow $\flow_s$.  Then the map
\[ \Psi_\alpha\colon  (-\delta,\delta) \times \mathbb R_{\geq 0} \to X_\f\]
defined by $\Psi_\alpha(t,s) = \flow_s(\alpha(t))$ is locally injective.
\end{lemma}
\begin{proof}
The $2$--cell $\sigma$ is a trapezoid whose bottom edge is an arc in a skew $1$--cell of the old cell structure, which in turn lives in the strip $\{\flow_s(e_0) : 0\le s\le 1\}$ over an edge $e_0$ of $\Gamma_0$. Since $\Gamma_0$ is a subdivision of $\Gamma$, we see that there is an edge $e\in E\Gamma$ so that $\alpha(-\delta,\delta)$ is contained in $\{\flow_s(e) : s\geq 0\}$. If we take a characteristic map $j\colon [0,1] \to X_\f$ for $e$, it therefore suffices  to show that
\[ \Psi_j\colon  (0,1) \times \mathbb R_{\geq 0} \to X_\f \]
is locally injective.  For this, observe that by composing with $\fib_0$, it follows that if $\Psi_j(s,t) = \Psi_j(s',t')$ then $t-t'$ is an integer.  Given a point $(s_0,t_0)$, choose a neighborhood of the form $W = U \times (t_0-\frac{1}{2},t_0 + \frac{1}{2})$.  If $\Psi_j(s,t) = \Psi_j(s',t')$ for $(s,t),(s',t') \in W$, then $t=t'$.  

To find $U$, we proceed as follows.  Let $T > t_0 + \frac{1}{2}$ be an integer.  Since $\f$ is a train track map, there is a neighborhood $U$ of $s_0$ so that $\f^T|_{e(U)}$ is injective.   Suppose $s,s' \in U$ and $|t-t_0| < \frac{1}{2}$ is such that $\Psi_j(s,t) = \Psi_j(s',t)$.
Then $\flow_t(e(s)) = \flow_t(e(s'))$.  Since $t < T$, we have $T-t > 0$ and the semiflow property implies
\begin{eqnarray*}
\f^T(e(s)) & = & \flow_T(e(s)) = \flow_{T-t}(\flow_t(e(s))) \\
 & = & \flow_{T-t}(\flow_t(e(s'))) = \flow_T(e(s')) = \f^T(e(s')).
\end{eqnarray*}
Therefore, $s = s'$ and $\Psi_j$ is injective on $W$.   Since $(s_0,t_0)$ was arbitrary, $\Psi_j$ is locally injective.
\end{proof}

Now suppose $\flow^z_s$ is the reparameterizations of the semiflow $\flow_s$ coming from a positive perturbation $z$ of $z_0$.  We can define a map $\Psi^z_\alpha(t,s) = \flow^z_s(\alpha(t))$ as in the lemma.  Because $\flow^z_s$ is a reparameterization of $\flow_s$, it follows that $\Psi^z_\alpha = \Psi_\alpha \circ H$ for some homeomorphism
\[ H\colon (-\delta,\delta) \times \mathbb R_{\geq 0} \to (-\delta,\delta) \times \mathbb R_{\geq 0}\]
(determined by the reparameterization).  Thus, the conclusion of the lemma remains true for $\Psi^z_\alpha$ in place of $\Psi_\alpha$.
In particular, since edges of $\Theta_z$ are arcs in the interiors of $2$--cells of $X_\f$, this implies the following.

\begin{lemma} \label{L:1returntt}
For each $n\geq 1$ the map $\f_z^n\colon \Theta_z\to\Theta_z$ is locally injective on the interior of each edge of $\Theta_z$. Moreover, $\f_z$ is a regular topological graph map.
\end{lemma}
\begin{proof}
The first statement follows immediately from the above discussion of the map $\Psi^z_\alpha$. To see the second statement, first recall that by Corollary~\ref{C:all h.e.}, we have $\f_z(V\Theta_z)\subseteq V\Theta_z$.  Let $e$ be an edge of $\Theta_z$ and let $j\colon[0,1]\to \Theta_z$ be such that $j(0) = o(e)$, $j(1) = t(e)$ and $j\vert_{(0,1)}$ is a homeomorphism onto $e$. To show that $\f_z$ is a topological graph map, we must check that $\f_z \circ j$ is a topological edge-path in $\Theta_z$. The regularity of $\f_z$ will then follow immediately from local injectivity on edges.

The set $(\f_z \circ j)^{-1}(V\Theta_z)$ of points that map to vertices of $\Theta_z$ is a closed, and hence compact, subset of $[0,1]$. Therefore, by the local injectivity of $\f_z\circ j$ and finiteness of $V\Theta_z$, the set is finite and we may write its points as $0 = c_0 < \dotsb < c_n = 1$. Now consider a subinterval $I = (c_{i-1},c_i)$. Since $(\f_z\circ j)(I)$ is disjoint from $V\Theta_z$, this image must be contained in a single edge $e'$ of $\Theta_z$. Moreover, the Intermediate Value Theorem implies that $(\f_z\circ j)\vert_{I}$ maps onto $e'$. Therefore, by the local injectivity, $(\f_z\circ j)\vert_{I}$ is a homeomorphism onto $e'$, as required.
\end{proof}

We next desire to construct a linear structure on $\Theta_z$ and to establish the necessary combinatorial properties of $\f_z$. To this end we introduce the following notation.

Each edge $e$ of $\Theta_z$ lives in trapezoid of $X_\f$ whose bottom is an arc of a skew $1$--cell of the old cell structure. Therefore, there exists an arc $\lambda_e$ contained in an edge of $\Gamma$ so that a reparameterization of the flow maps $\lambda_e$ homeomorphically onto $e$. More precisely, there exists a continuous function $\tau_{e}\colon \lambda_e\to \R_{\geq 0}$ so that the map $\rho_{e}\colon \lambda_e\to X_\f$ defined by $\rho_{e}(y) = \flow_{\tau_{e}(y)}(y)$ sends $\lambda_e$ homeomorphically onto $e$. If we then choose an integer $m_e\geq \max\{\tau_e(\lambda_e)\}$ and define $s_{e}\colon e\to [0,m_e]$ by $s_{e}(x) = m_e-\tau_{e}(\rho_{e}^{-1}(x))$, we find that the map $\rho'_e\colon e\to X_\f$ defined by $\rho'_e(x) = \flow_{s_e(x)}(x)$ maps $e$ into $\Gamma_0=\Gamma$ in such a way that
\[\xymatrix{
\lambda_e \ar[r]^{\rho_e} & e \ar[r]^{\rho'_e} & \Gamma
}\]
is just $\f^{m_e}\vert_{\lambda_e}$.

Recall that by assumption $\f\colon\Gamma\to\Gamma$ is an expanding irreducible train-track map. Thus its transition matrix $A(\f)$ is irreducible and has $\lambda(f)>1$.

\begin{lemma}\label{L:1returnEAS}
The topological graph map $\f_z\colon \Theta_z\to\Theta_z$ is expanding on all scales. 
\end{lemma}
\begin{proof}
Consider any neighborhood $U\subset \Theta_z$ of a point $x\in \Theta_z$, and an edge $e\in E\Theta_z$. After restricting to a smaller open set we may assume that $U$ is an open interval contained in an edge $e_0$ of $\Theta_z$. As above, there are arcs $\lambda_{e_0}$ and $\lambda_{e}$ that get mapped homeomorphically onto $e_0$ and $e$ by the flow. In particular $\lambda_{e_0}$ contains an open subinterval $I$ so that the restriction $\rho_{e_0}\vert_{I}\colon I\to U$ is a homeomorphism.

Let $\hat{e}$ denote the edge of $\Gamma$ containing $\lambda_e$. By Lemma~\ref{L:linear maps expanding on all scales}, $\f\colon \Gamma\to\Gamma$ is expanding on all scales. Therefore there is a subinterval $I'\subset I$ and a power $n\geq 1$ so that $\f^n$ maps $I'$ homeomorphically onto $\lambda_e\subset\hat{e}$. In fact, by restricting to smaller neighborhoods, we can arrange for $n$ to satisfy $n\ge m_{e_0}$. Let $U' = \rho_{e_0}(I')\subset U$. It follows that $\rho_{e}\circ \f^{n-m_e}\circ \rho'_{e_0}$ restricts to a homeomorphism from $U' = \rho_{e_0}(I')$ to
\[\rho_{e}\circ \f^{n-m_e}\circ \rho'_{e_0}(\rho_{e_0}(I')) = \rho_{e}\circ \f^{n}(I') = \rho_{e}(\lambda_e) = e.\]
However, this map is just $\f_z^k\vert_{U'}$ for some integer $k\geq 1$.
\end{proof}

Lemmas~\ref{L:1returntt} and \ref{L:1returnEAS} show that $\f_z$ satisfies the hypotheses of Theorem~\ref{T:topological graph to graph}. Therefore  by Theorem~\ref{T:topological graph to graph} and Theorem~\ref{T:folded basics} part (4) we conclude the following:

\begin{corollary}\label{C:linearStructure}
There is a linear structure $\Lambda$ on $\Theta_z$ with respect to which $\f_z\colon\Theta_z\to\Theta_z$ is a linear graph map with irreducible transition matrix $A(\f_z)$ such that $\lambda(\f_z)>1$.  Consequently, $\f_z$ is an expanding irreducible train-track map representing $\fee_{u}$ for $u = [z]$.\qed
\end{corollary}

We have thus established the conclusion of Theorem~\ref{T:new iwip main} (1).  From this point on, we fix a linear structure $\Lambda$ on $\Theta_z$ provided by Corollary~\ref{C:linearStructure}.

\begin{conv} \label{conv:iwiphyp}
From now on and until the end of this section, unless stated otherwise we further assume that the element $\fee\in \Out(F_N)$ is hyperbolic and fully irreducible.
\end{conv}

By Theorem~\ref{T:clean} we now know that $\f$ is a clean train track map. Thus the Whitehead graph $\wh(f,v)$ associated to each vertex $v$ of $\Gamma$ is connected and $A(\f)^m$ is positive for some $m\geq 1$. We now show that all Whitehead graphs of $\f_z\colon \Theta_z\to\Theta_z$ are connected as well.

Recall that $\Gamma_t = \fib_0^{-1}(t)$ denotes the horizontal fiber of $X_\f$ at height $t\in [0,1]$, and that $\Gamma_0=\Gamma$. 
\begin{defn}[$z$--modified Whitehead graph]
For each point $x\in \Gamma_t$ we define a corresponding \emph{$z$--modified Whitehead graph} $\wh^z(\Gamma_t,x)$ as follows. If $x$ lies in the interior of an edge of $\Gamma_t$, then subdivide that edge at $x$. The vertices of $\wh^z(\Gamma_t,x)$ are as usual the set of ``directions'' at $x$, that is the set of edges in the (subdivided) graph that originate at $x$. Two such edges $e_1,e_2$ are adjacent in $\wh^z(\Gamma_t,x)$ if there exists an edge $e$ of $\Theta_z$ and a power $n\geq 0$ so that the edge-path $\flow_t\circ\f^n\circ\rho'_e(e)$ takes the corresponding  turn $e_1,e_2$; in this case we call $e_1,e_2$ a \emph{$z$--taken turn}.
\end{defn}

\begin{lemma}\label{lem:z-modifiedWhiteheadGraphs}
The $z$--modified Whitehead graph $\wh^z(\Gamma_t,x)$ is connected for all $t\in [0,1]$ and $x\in \Gamma_t$.
\end{lemma}
\begin{proof}
We first claim that $\wh^z(\Gamma_0,x)$ is connected for each $x\in \Gamma_0$. 
Indeed, for every edge $e$ of $\Theta_z$ there exists an $n\ge 1$ so that the path $f^n\circ \rho'_e(e)$ crosses every edge of $\Gamma_0=\Gamma$. This follows from the positivity of $A(f)^n$ for large $n$ together with the facts that $f$ is expanding on all scales and $\rho'_e(e)$ contains an open neighborhood in $\Gamma$. This shows that $\wh^z(\Gamma_0,x)$ is connected whenever $x$ lies in the interior of an edge of $\Gamma_0$. On the other hand, when $x$ is a vertex of $\Gamma_0$ it is easy to see that $\wh^z(\Gamma_0,x) =\wh(\f,x)$, which we already know is connected. Indeed, since $\f^n\circ\rho'_e(e)$ crosses every edge of $\Gamma_0$, whenever $k\geq 1$ and $c\in E\Gamma_0$ are such that $\f^k(c)$ takes a turn at $x$, then $\f^{n+k}\circ\rho'_e(e)$ takes the same turn at $x$.

Now let $t > 0$ and let $x$ be any point of $\Gamma_t$. Let $P = (\flow_t)^{-1}(x) \subset \Gamma_0$ be the preimage.  Notice that $\flow_t$ maps each $z$--taken turn at $p\in P$ to a $z$--taken turn at $x$; that is, $\flow_t$ induces a map of $z$--modified Whitehead graphs $\flow_t^*\colon\wh^z(\Gamma_0,p)\to\wh^z(\Gamma_t,x)$ for each $p\in P$. Indeed, each edge $e_1$ originating at $p$ is mapped to an edge $\flow_t(e_1)$ originating at $x$; and if $e\subset \Theta_z$ is an edge such that $\f^n\circ\rho'_e(e)$ takes a turn $e_1,e_2$ at $p\in P$, then $\flow_t\circ\f^n\circ\rho'_e(e)$ takes the turn $\f_t(e_1),\f_t(e_2)$ at $x$. 

Let $T\subset X_\f$ denote the union of the flow lines $\flow_{[0,t]}(p)$ for the points $p\in P$ (i.e., $T$ is the set $\{\flow_s(p)\mid s\in [0,t], p\in P\}$). Then $T$ is an oriented trivalent tree with $\lvert P\rvert$ leaves on the bottom and one point $x$ at the top. Each point $y\in T$ lives in some $\Gamma_s$ and has a corresponding $z$--modified Whitehead graph $\wh^z(\Gamma_s,y)$. The fact that $\wh^z(\Gamma_t,x)$ is connected will follow from the following claim.

\textbf{Claim:} $\wh^z(\Gamma_s,w)$ is connected for each $w\in T$. 

We have already demonstrated this for the $\lvert P\rvert$ points at the base of the tree. If $w\in \Gamma_s$ lies in the interior of an edge of $T$, and $y\in \Gamma_r$, with $r<s$, lies in the same edge, then the flow $\flow_{s-r}$ induces a quotient map $\wh^z(\Gamma_r,y)\to\wh^z(\Gamma_s,w)$ (the fact that $y$ and $w$ lie in the same edge of $T$ implies that $y$ does not get identified with any other points in the interval $[r,s]$). In particular, $\wh^z(\Gamma_s,w)$ is connected provided $\wh^z(\Gamma_r,y)$ is.

Now suppose that $2$ upward-oriented edges $b_1,b_2$ of $T$ meet at a point $w\in \Gamma_s$. Restricting to a sufficiently small time interval $[r,s)$, we may assume that there is a constant $z$--modified Whitehead graph $G_i$ along each $b_i$. By induction on the combinatorial distance to $P$ within $T$ we may suppose that each $G_i$ is connected. Let $e_1,e_2\in E\Gamma_r$ denote the two edges of $\Gamma_r$ whose initial segments are folded together by $r_{s,r}\colon \Gamma_r\to\Gamma_s$. The closures of these edges necessarily contain the points $y_1,y_2\in \Gamma_r\cap T$ for which $\flow_{s-r}(y_1)=\flow_{s-r}(y_2) = w$. The initial segment $\alpha_i$ of $e_i$ (oriented from $y$ to $o(e_i)$) then corresponds to a vertex in $G_i=\wh^z(\Gamma_r,y_i)$. Since $\alpha_1$ and $\alpha_2$ are folded together by $\flow_{s-r}$, we see that the induced map $\flow^*_{s-r}$ on $z$--modified Whitehead graphs identifies these two vertices. Therefore $\wh^z(\Gamma_s,w)$ is the wedge of the connected graphs $G_1$ and $G_1$ and in particular is connected.
\end{proof}

\begin{corollary}\label{C:connectedWhiteheadGraphs}
The Whitehead graphs of all vertices of $\Theta_z$ with respect to $\f_z\colon\Theta_z\to\Theta_z$ are connected.
\end{corollary}
\begin{proof}
Recall that $V\Theta_z$ is defined to be the set of intersection points of $\Theta_z$ with $X_\f^{(1)}$, 
together with some additional degree-2 vertices obtained by subdividing at points of intersection of the flow lines through this first set of points---see Definition \ref{D:Theta_z}.

Let $v$ be a vertex of $\Theta_z$. Let us first suppose that $v$ is not contained in a skew $1$--cell of $X_\f$.  One may then find a neighborhood $U\subset X_\f$ of $v$ and a time $t$ such that the flow $\flow$ maps $\Gamma_t\cap U$ homeomorphically onto $\Theta_z\cap U$. Let $x$ be the point of $\Gamma_t\cap U$ that flows onto $v$. From the definition of $z$--modified Whitehead graphs, we see that the flow induces a graph isomorphism $\wh^z(\Gamma_t,x)\to \wh_{\Theta_z}(f_z,v)$. Therefore $\wh_{\Theta_z}(\f_z,v)$ is connected by Lemma~\ref{lem:z-modifiedWhiteheadGraphs}.

Now suppose that $v$ lies on a skew $1$--cell. In this case we can find a neighborhood $U\subset X_\f$ of $v$ so that $U\cap \Theta_z$ is a tree with one vertex $v$ and $3$ edges, and a time $t$ so that $U\cap \Gamma_t$ consist of $2$ disjoint edges. Moreover, each component of $U\cap \Gamma_t$ contains a point $x_i$ that flows onto $v$. By Lemma~\ref{lem:z-modifiedWhiteheadGraphs}, we know that each graph $\wh^z(\Gamma_t,x_i)$ is connected (in this case just a single edge with two vertices). The flow now induces a quotient map
\[\wh^z(\Gamma_t,x_1) \sqcup \wh^z(\Gamma_t,x_2)\to \wh_{\Theta_z}(\f_z,v)\]
that takes a vertex from each $\wh^z(\Gamma_t,x_i)$ and identifies these to a single vertex in $\wh_{\Theta_z}(\f_z,v)$. Therefore $\wh_{\Theta_z}(\f_z,v)$ is the wedge of $2$ line segments and is in particular connected.
\end{proof}

\begin{corollary}\label{C:wc}
The map $f_z\colon \Theta_z\to\Theta_z$ is a weakly clean train-track map.
\end{corollary}
\begin{proof}
We already know, by Corollary~\ref{C:linearStructure},  that $f_z\colon \Theta_z\to\Theta_z$ is an expanding train track map with an irreducible transition matrix $A(f_z)$.
Corollary~\ref{C:connectedWhiteheadGraphs} implies that the Whitehead graphs $\wh_{\Theta_z}(\f_z,v)$ are connected for all $v\in V\Theta_z$. Thus $f_z$ is a weakly clean train-track map.
\end{proof}

\begin{proof}[Proof of Theorem~\ref{T:new iwip main}]

Suppose first that $f\colon\Gamma\to\Gamma$ is a combinatorial expanding irreducible train-track map representing $\fee$, without assuming that $\fee$ is hyperbolic and fully irreducible.
Recall that $u\in \A$ is a primitive integral element, that $z\in Z^1(X_f;\R)$ is a positive $1$--cocycle representing the cohomology class $u$ and that $\f_z \colon \Theta_z\to\Theta_z$ induces the monodromy $(\f_z)_* = \fee_u \in \Out(\ker(u)) = \Out(\pi_1(\Theta_z))$.
Lemma~\ref{L:1returntt} shows that $f_z$ is a topological graph-map and thus $f_z$ is a topological representative of the automorphism $\fee_u$ of $\ker(u)=\pi_1(\Theta_z)$.
Corollary~\ref{C:linearStructure} implies that there exists a linear structure $\Lambda$ on $\Theta_z$ so that $f_z\colon\Theta_z\to\Theta_z$ is an expanding irreducible train-track map representing $\fee_u$. This verifies part (1) of Theorem~\ref{T:new iwip main}. 

Suppose now that $\fee$ is additionally assumed to be hyperbolic and fully irreducible, in which case $G_\fee$ is word-hyperbolic and $f$ is a clean train-track representing $\fee$. Since $G_{\fee_u} \cong G_\fee$ is word-hyperbolic, we also know that $\fee_u$ is a hyperbolic automorphism. By Corollary~\ref{C:wc} we additionally have $\f_z$ is a weakly clean train track map.   Hence $\fee_u$ is fully irreducible by Theorem~\ref{T:clean}.
\end{proof}

\begin{remark}\label{R:geometric-iwip}
The full irreducibility conclusion for $\fee_u$ of part (2) in
Theorem~\ref{T:new iwip main}  does not necessarily hold under the
weaker assumption in part (2) that $\fee$ only be fully irreducible
(i.e.~without the assumption that $\fee$ be hyperbolic). Recall that,
as proved in \cite{BH92}, a fully irreducible $\fee\in \Out(F_N)$ is
non-hyperbolic if and only if $\fee$ comes from a pseudo-Anosov
homeomorphism $f$ of a once-punctured connected surface $S$ without
boundary with $F_N=\pi_1(S)$. In this case the mapping torus $M_f$ of
$f$ is a hyperbolic $3$--manifold fibering over the circle with fiber
$S$ and $G_\fee=\pi_1(M)$.  If $M$ admits a different fibration over
$\mathbb{S}^1$ with fiber $Y$, then $M_f$ also decomposes as the mapping torus of a homeomorphism $f'$ of $Y$ and this  $f'$ is again pseudo-Anosov. However, in general, there is no reason to 
expect that the new fiber $Y$ will have exactly one puncture, even for a fibration of $M$ coming from a class $u\in H^1(G_\fee,\R)$  which is projectively close to the class $u_0$ associated with the original fibration. If such $Y$ has at least two punctures, then the element of $\Out(\pi_1(Y))$ corresponding to $f'$ is not fully irreducible, since the punctures of $Y$ are permuted by $f'$ and the peripheral curves around the punctures are primitive elements of $\pi_1(Y)$.
\end{remark}

\section{Dynamics and entropy}

In this section, we establish Theorem~\ref{T:new entropy main} from the Introduction:
\begin{theorem:new entropy main}
Let $\fee \in \Out(F_N)$ (where $N\ge 2)$ be an element represented by an expanding irreducible train track map $f \colon \Gamma \to \Gamma$. Let $X_f$ be a folded mapping torus, $\flow$ the associated semiflow, and $\A_{X_f} \subset H^1(X_\f;\R)$ the positive cone of $X_\f$.  Then there exists a continuous, homogeneous of degree $-1$ function $\mathfrak H \colon \A_{X_f} \to \R$
such that for every primitive integral $u \in \A_{X_f}$
\[ \log(\lambda(\fee_u)) = \mathfrak H(u).\]
Moreover, $1/\mathfrak H$ is concave, and hence $\mathfrak H$ is convex.
\end{theorem:new entropy main}

Our proof closely follows Fried's proof of the analogous statement for mapping tori of pseudo-Anosov homeomorphisms (which applies in a more general setting); see \cite{FriedD}.  To carry out the arguments here, we work with the natural extension, as we now explain.

\subsection{The natural extension of the first return map}

Fix any integral class $u \in \A$ represented by a positive cocycle $z$ and let $\fib_z\colon X = X_\f \to \mathbb{S}^1$ be an associated fibration.  Recall that $\Theta_z$ is the preimage of a point under $\fib_z$ which we assume (as we may) is $0 \in \mathbb{S}^1 = \R/\Z$, so that $\Theta_z = \fib_z^{-1}(0)$.  Additionally, $\f_z\colon \Theta_z \to \Theta_z$ denotes the first return map of $\psi$.
The failure of $\f_z\colon \Theta_z \to \Theta_z$ to be a homeomorphism can be rectified in a certain sense by the construction of the {\em natural extension}.  This is the space
\[ \Theta_z' = \{ \underline{x} = ( \ldots, x_{-1},x_0,x_1,\ldots ) \mid x_i \in \Theta_z \mbox{ such that } x_i  = \f_z(x_{i-1}) \mbox{ for all } i \}. \]
equipped with the product topology, together with the {\em shift map} $\f_z'\colon \Theta_z' \to \Theta_z'$ defined by
\[ (f_z'(\underline{x}))_i = x_{i+1} \]
for all $i \in \Z$.  That is, the $i^{th}$ coordinate of $\f_z'(\underline{x})$ is the $(i+1)^{st}$ coordinate of $\underline{x}$.  There is a continuous surjective semi-conjugacy $\pi_z\colon \Theta_z' \to \Theta_z$ given by
\[ \pi_z(\ldots , x_{-1}, x_0, x_1, \ldots) = x_0. \]
That is, $\pi_z \circ \f_z' = \f_z \circ \pi_z$ ($\pi_z$ is a semi-conjugacy of $\f_z'$ onto $\f_z$). This fact is immediate from the definition of $\Theta_z'$ and $\f_z'$.
Since $\pi_z$ is surjective, the topological entropies satisfies the inequality
\[ h(\f_z') \geq h(\f_z).\]
However, the system $\f_z'\colon \Theta_z' \to \Theta_z'$ is the ``smallest'' homeomorphism which admits a semiconjugacy to $\f_z\colon \Theta_z \to \Theta_z$, in the sense that any other homeomorphism which can be semiconjugated onto $\f_z \colon \Theta_z \to \Theta_z$ factors through $\pi_z$ via a surjection to $\Theta_z'$.  The key fact for us is that topological entropy is actually preserved in this case, see \cite[Section 6.8]{Down}:

\begin{lemma} \label{L:entropy preserved}
The topological entropy of the extension satisfies
\[ h(\f_z') = h(\f_z).\]
\end{lemma} 

\subsection{A natural extension of the semiflow}

We can make a similar construction for the semiflow $\flow_s\colon X \to X$.
Say that a map $\gamma\colon \R \to X$ is a {\em biinfinite flow line} for $\flow$ if 
\[ \flow_s(\gamma(t)) = \gamma(t+s) \]
for all $t \in \R$ and $s \geq 0$.  Let $X'$ denote the space of all biinfinite flow lines, equipped with the topology of uniform convergence on compact sets:
\[ X' = \{ \gamma \mid \gamma \mbox{ is a biinfinite flow line } \}. \]
This comes equipped with a flow $\Flow$ defined by $(\Flow_s(\gamma))(t) = \gamma(t + s)$ for all $t, s \in \R$.
There is a map $\Pi\colon X' \to X$ given by $\Pi(\gamma) = \gamma(0)$ which is a continuous surjection (since there is a biinfinite flow line through every point).
Again from the definitions, we see that for all $s \geq 0$ we have
\[ \flow_s \circ \Pi = \Pi \circ \Flow_s\]
which is to say, $\Pi$ semiconjugates $\Flow$ to $\flow$.  This is also clearly the minimal flow with this property.

By composing $\fib_z$ with $\Pi$ we obtain a map
\[ \fib_z' = \fib_z \circ \Pi\colon X' \to \mathbb{S}^1.\]
We put $\Theta_z'':= (\fib_z')^{-1}(0)$ and denote the first return map of $\Flow$ to this section by $\f_z''\colon \Theta_z'' \to \Theta_z''$.  As the next lemma shows, $\f_z''$  is nothing but the natural extension of $\f_z\colon \Theta_z \to \Theta_z$ already constructed above.
\begin{lemma}
There is a homeomorphism $F\colon \Theta_z'' \to \Theta_z'$ that conjugates $\f_z''$ to $\f_z'$.
\end{lemma}
\begin{proof}
First observe that $\gamma \in X'$ is in $\Theta''_z$ if and only if $\gamma(0) \in \Theta_z$.  Such a biinfinite flow line $\gamma\colon \R \to X$ meets $\Theta_z$ at infinitely many points in both forward and backward time.  More precisely, there is an increasing, biinfinite, proper sequence of times $\{t_j\}_{j \in \Z}$ with $t_0 = 0$ so that
\begin{equation} \label{E:flowlinesequence} \gamma(t) \in \Theta_z \Leftrightarrow t = t_j \mbox{ for some } j.
\end{equation}
Since $\f_z$ is the first return map of $\flow$ to $\Theta_z$, it follows that $\f_z(\gamma(t_j)) = \gamma(t_{j+1})$ for all $j$. 
Therefore, for any $\gamma \in \Theta_z''$, we may define
\[ F(\gamma) = (\ldots,\gamma(t_{-1}),\gamma(t_0),\gamma(t_1),\ldots)\]
(that is, the $i^{th}$ coordinate of $F(\gamma)$ is $\gamma(t_i)$), which is an element of $\Theta_z'$.

To see that $F$ is injective, we claim that the sequence
\[ (\ldots,\gamma(t_{-1}),\gamma(t_0),\gamma(t_1),\ldots)\]
determines $\gamma$.  This is because, for any $t > s$, $\gamma(s)$ determines $\gamma(t)$ by the formula $\gamma(t) = \psi_{t-s}(\gamma(s))$.  As $j \to -\infty$, $t_j \to - \infty$, and so for any $t \in \R$, $\gamma(t)$ is determined by some $\gamma(t_j)$ with $t_j < t$.  Therefore, $F$ is injective.  Similarly, given any sequence $\underline{x} \in \Theta_z'$, we can use this reasoning to construct $\gamma \in \Theta_z''$ with $F(\gamma) = \underline{x}$.  

Suppose now that $\gamma_n \to \gamma$ is a sequence of flow lines converging uniformly on compact sets with $\gamma(0),\gamma_n(0) \in \Theta_z$ for all $n$ and $\gamma_n(0) \to \gamma(0)$.  Let $\{t_j\}_{j \in \Z}$ denote the sequence for $\gamma$ as above, and similarly, $\{t_{n,j} \}_{j \in \Z}$ the sequence for $\gamma_n$.  The uniform convergence on compact sets implies that the points of intersection of $\gamma_n$ with $\Theta_z$ are converging uniformly on compact sets to the intersection points of $\gamma$ with $\Theta_z$, which is to say, $\gamma_n(t_{n,j}) \to \gamma(t_j)$ for all $j$ as $n \to \infty$.  Therefore $F$ is continuous.  Because $X$ is compact, one can show that $X'$ is also compact, hence so is $\Theta_z''$.  Since $\Theta_z'$ is Hausdorff, the continuous map $F\colon \Theta_z'' \to \Theta_z'$ is a homeomorphism.

For any $\gamma \in \Theta_z''$ and any $j$, we have $\f_z(\gamma(t_j)) = \gamma(t_{j+1})$, and therefore $F$ conjugates $\f_z''$ to the shift map $\f_z'$.
\end{proof}

Because of this lemma, we will ignore the distinction between $\f_z''\colon \Theta_z'' \to \Theta_z''$ and $\f_z'\colon \Theta_z' \to \Theta_z'$, and just refer to $\Theta_z'$ as the section of $\Flow$ and $\f_z' \colon \Theta_z' \to \Theta_z'$ the first return map of $\Flow$.

\subsection{The variational principal and Abramov's Theorem}

The topological entropy of $\f_z'$ is approximated from below by metric entropies $h_\nu(\f_z')$ for invariant probability measures $\nu$ on $\Theta_z'$, according to the {\em variational principle}; see \cite[Section 6.8]{Down}.

\begin{theorem} \label{T:variational principle} For any integral class $u \in \A$ represented by a positive cocycle $z$ we have
\[ h(\f_z') = \sup_\nu h_\nu(\f_z') \]
where $\nu$ ranges over all $\f_z'$--invariant Borel probability measures on $\Theta_z'$.
\end{theorem}

The key to proving continuity (and the other properties) of the entropy is a result of Abramov which relates entropy of $\f_z'$ to that of the flow $\Flow$.  To make sense of this, we have to relate $\f_z$--invariant measures on $\Theta_z'$ to $\Flow$--invariant measures on $X'$.  In fact, there is a bijection between these two spaces of measures defined as follows.

Given a $\Flow$--invariant probability measure $\mu$ on $X'$, we can construct a measure $\bar \mu_0$ on any section $\Theta_z'$ by the formula
\begin{equation} \label{E:hat mu from mu} \bar \mu_0(A) = \frac{1}{t}
  \mu( \{ \flow_s(x) \mid x \in A, 0\le s\le t \})
\end{equation}
for any Borel measurable set $A \subset \Theta_z'$.  This may not be a probability measure, and indeed the total measure $\bar \mu_0(\Theta_z')$ is an important quantity.  The $\f_z'$--invariant probability measure on $\Theta_z$ is the normalization
\[ \bar \mu = \frac{1}{\bar \mu_0(\Theta_z')} \bar \mu_0.\]

There is a flow-equivariant local homeomorphism
\[ \Theta_z' \times \R \to X' \]
given by $(\underline{x},s) \mapsto \Flow_s(\underline{x})$ (where the flow on the domain is just the natural translation action on the second coordinate).
The product measure $\bar \mu_0 \times m$, where $m$ is Lebesgue measure on $\R$, is flow-invariant and pushes forward to a flow invariant measure.  From the construction of $\bar \mu_0$, it follows that $\mu$ is the push forward of $\bar \mu_0 \times m$.  As an abuse of notation, we will just write $\mu = \bar \mu_0 \times m$.
This can be used to describe the inverse construction.

Abramov's Theorem relating the entropy of a flow to the first return map of a section (see \cite{Abramov,DGS}) takes the following form in our situation (see also \cite{FriedD}). 
\begin{theorem} \label{T:Abramov}
For any integral class $u \in \A$ represented by a positive cocycle $z$ and any $\Flow$--invariant probability measure $\mu$ on $X$, let $\bar \mu_0$ and $\bar \mu$ be defined as above.
Then
\[ h_{\bar \mu}(\f_z') = \frac{h_\mu(\Flow)}{\bar \mu_0(\Theta_z')}.\]
\end{theorem}

As an immediate corollary of these two results, we have the following useful expression for the reciprocal of the topological entropy.
\begin{corollary} \label{C:variation Abramov}
For any integral class $u \in \A$ represented by a positive cocycle $z$, we have
\[ \frac{1}{h(\f_z')} = \inf_{\mu} \frac{\bar \mu_0(\Theta_z')}{h_\mu(\Flow)}\]
where the infimum is over all $\Flow$--invariant Borel probability measures $\mu$ on $X'$, and $\bar \mu_0$ is constructed from $\mu$ as above.
\end{corollary}

\subsection{The measure $\bar \mu_0$ and linearity}

Given a $\Flow$--invariant probability measure $\mu$ on $X'$, we now have a function defined on the integral classes $u \in \A$ which we denote
\[ \mu_*(u) = \bar \mu_0(\Theta_{z_u}')\]
where $z_u$ is a positive cocycle representing $u$.   If $z_u'$ is another positive cocycle representing $u$ with fiber $\Theta_{z_u'}'$, then we can flow $\Theta_{z_u}'$ homeomorphically onto $\Theta_{z_u'}'$, and hence $\bar \mu_0(\Theta_{z_u'}') = \bar \mu_0(\Theta_{z_u}')$ is independent of the choice of $z_u$ (and fibration $\fib_{z_u}$).

Our goal is now to prove the following.
\begin{lemma} \label{L:mu* linear}
If $\mu$ is a $\Flow$--invariant probability measure on $X'$, then $\mu_*$ is the restriction of a linear function (of the same name):
\[ \mu_*\colon H^1(X; \mathbb R) \to \mathbb R.\]
\end{lemma}
\begin{proof}
Given an integral class $u \in \A$ represented by a positive cocycle $z= z_u$, we can reparameterize the flow as $\Flow^z$, so that for the projection of the flow lines to $\mathbb{S}^1 = \R/\Z$ by $\fib_z'$ are local isometries.  The bijections
\[  \begin{array}{ccc} \{ \Flow\mbox{--invariant measures on } X'  \} & \leftrightarrow & \{ \f_z'\mbox{--invariant measures on } \Theta_z' \} \\
 & \leftrightarrow  & \{ \Flow^z\mbox{--invariant measures on } X'  \}   \end{array} \]
will allow us to turn the $\Flow$--invariant measure $\mu$ on $X'$ into a $\Flow^z$--invariant measure $\mu^z$ on $X'$, given by
\[ \mu_z = \bar \mu_0 \times m_z \]
where $m_z$ is the Lebesgue measure on $\R$, and the subscript $z$ is only there to remind us that it is induced by the reparameterized flow $\Flow^z$.  

The measure $\mu_z$ is {\em not} a probability measure.  Indeed, the parameterization is chosen so that the map
\[ \Theta_z' \times [0,1) \to X'\]
given by
\[ (\underline{x},s) \mapsto \Flow^z_s(\underline{x}) \]
is 1-1 and onto.  It follows that
\[ \mu_z(X') = \bar \mu_0(\Theta_z').\]
What is important about this formula is that the left hand side does not involve the section $\Theta_z'$.  Moreover, we could have used a different section $\Theta_w'$, and as long as we use the flow $\Flow^z$, we would produce the same measure $\mu_z$.

Now suppose that $u,v \in \A$ are two integral classes represented by the positive cocycles $z$ and $w$, respectively.  Then $u+v$ is represented by $z+w$, which is also positive.  A useful observation is that because  $\fib_z$ and $\fib_w$ are maps to the {\em additive group} $\mathbb S^1 = \R/\Z$, we can add them to produce a map $\fib_z + \fib_w$.  Each of $\fib_z$ and $\fib_w$ satisfy Lemma~\ref{L:constructing f_z}, and without loss of generality, the finite cover $\mathbb M$ by local models in the statement is the same for both $\fib_z$ and $\fib_w$.  Indeed, we may always pass to a refinement of the cover without affecting the validity of the lemma, and we will freely do so without further mention.  In addition, by adding a constant in $\mathbb S^1$ to each of $\fib_z$ and $\fib_w$, we may assume that the point $x_0 \in \mathbb S^1$ is the same in both, and is the identity $0 \in \mathbb S^1$.

Setting $\fib_{z+w} = \fib_z + \fib_w$, it is straightforward to see that this map satisfies properties (1), (2), and (3) of Lemma~\ref{L:constructing f_z} for the cocycle $z+w$ representing the class $u+v$.  It may not satisfy property (4), however.  To address this issue, we proceed as follows.  Note that $\Xi = \fib_{z+w}^{-1}(0) \cap X_\f^{(1)}$ is the finite set of points $\xi \in X_\f^{(1)}$ with the property that $\fib_z(\xi) = -\fib_w(\xi)$.  Then $\Xi = \Xi_0 \cup \Xi_1$ where $\fib_z(\xi) = 0 = \fib_w(\xi)$ for $\xi \in \Xi_0$ and $\fib_z(\xi) \neq 0 \neq \fib_w(\xi)$ for $\xi \in \Xi_1$.  Consequently, $\Xi_0 \subset \Theta_z \cap \Theta_w$ and $\Xi_1 \cap \Theta_z= \emptyset$. After modifying $\fib_z$ by an arbitrarily small homotopy if necessary (which is the identity on $\Theta_z$), we may assume that $\Xi_1 \subset X_\f^{(1)} \setminus X_\f^{(0)}$ and that this set flows into the union of vertical $1$--cells.  Since $\Xi_0$ already has this property (being contained in $\Theta_z \cap \Theta_w$), it follows that $\fib_{z+w}$ satisfies all properties of Lemma~\ref{L:constructing f_z}.

The map $\fib_{z+w} = \fib_z + \fib_w$ determines a corresponding map
\[ \fib_{z+w}' = \fib_z' + \fib_w' \colon X' \to \mathbb S^1.\]
The measure $\mu_z$ and $\mu_w$ are defined by taking the product $\bar \mu_0$ with the Lebesgue measures $m_z$ and $m_w$, induced by the reparameterized flows $\Flow^z$ and $\Flow^w$, respectively.  Since these flows are parameterized so that $\fib_z'$ and $\fib_w'$ restricted to flow lines are local isometries, it follows that $\mu_{z+w}$ is defined by taking the product $\bar \mu_0$ with $m_z + m_w$.  Consequently, we have
\[ \mu_{z+w} = \mu_z + \mu_w.\]
Hence
\[ \mu_*(u+v) = \mu_{z+w}(X') = \mu_z(X') + \mu_w(X') = \mu_*(u) + \mu_*(v)\]
and $\mu_*$ is linear on the integral points of $\A$. It follows that $\mu_*$ uniquely extends to a linear function on $H^1(X;\R)$.
\end{proof}

\subsection{Completing the proof}

Combining the previous facts, we give the
\begin{proof}[Proof of Theorem \ref{T:new entropy main}]

If $u\in \A$ is a primitive integral element and $z_u$ is a positive 1-cocycle,
then, by Corollary \ref{C:all h.e.} the map $f_{z_u}\colon \Theta_{z_u}\to\Theta_{z_u}$ is a homotopy equivalence taking vertices to vertices. By
Corollary \ref{C:perturbed semidirect product} we also know that $\Theta_{z_u}$ is connected
and furthermore, it has no valence-1 vertices (see the proof of Theorem \ref{T:folded basics}).  By Lemma~\ref{L:1returntt} the map
$f_{z_u}\colon \Theta_{z_u}\to\Theta_{z_u}$ is a regular topological graph-map.
Corollary~\ref{C:perturbed semidirect product} now implies that
$f_{z_u}$ is a regular topological representative of the monodromy
$\fee_u\in \Out(\ker(u))$. By Proposition~\ref{prop:h(f)} we have
$h(\f_{z_u})=\log\lambda(\f_{z_u})$ where $\lambda(\f_{z_u})$ is the
spectral radius of the transition matrix $A(\f_{z_u})$. We have also
seen in the proof of Theorem~\ref{T:new iwip main} that for an appropriate choice of a linear
structure on $\Theta_{z_u}$ the map
$f_{z_u}\colon \Theta_{z_u}\to\Theta_{z_u}$ is an expanding irreducible train-track
map representing $\fee_u$. Therefore
$\lambda(\f_{z_u})=\lambda(\fee_u)$ is the growth rate of $\fee_u$ and
thus $h(\f_{z_u})=\log\lambda(\f_{z_u})=\log
\lambda(\fee_u)=h(\fee_u)$, the algebraic entropy of $\fee_u$. Since
$\fee_u$ is fully irreducible, we have $h(\f_{z_u})=h(\fee_u)>0$.

Define a function $W\colon \mathcal A \to \mathbb R$ by
\[ W(u):= \inf_{\mu} \frac{\mu_*(z_u)}{h_\mu(\Flow)}\]
where $u\in \A$, where $z_u$ is a positive cocycle representing $u$, and the infimum is over all Borel $\Flow$--invariant probability measures $\mu$ on $X'$. 
By Lemma \ref{L:entropy preserved} we have $W(u) = 1/h(\f_{z_u}') =
1/h(\f_{z_u})$ for any primitive integral class $u \in \mathcal A$.
The linearity of $\mu_*$ implies that the function $W$ is homogeneous
of degree $1$. By Corollary \ref{C:variation Abramov}, we see that $W$ is the infimum of a collection of linear functions, and is therefore concave and hence continuous.
For any primitive integral $u\in A$ we have $0<h(\f_{z_u})<\infty$ and
hence $0< W(u)=1/h(\f_{z_u})<\infty$. Since the primitive integral
classes are projectively dense in $\A$, concavity and homogeneity of
$W$ imply that $0<W(u)<\infty$ for all $u\in \A$. Therefore setting
$\mathfrak H(u):=1/W(u)$ for all $u\in \A$ gives us a continuous
homogeneous of degree $-1$
function $\mathfrak H\colon \A\to (0,\infty)$. By
construction, for every primitive integral $u\in \A$ we have
$\mathfrak H(u)=\log\lambda(\f_{z_u})=\log
\lambda(\fee_u)=h(\fee_u)$, as required.
\end{proof}

As mentioned in the introduction, we have the following analogue from the $3$--manifold setting of McMullen's~\cite[Section 10]{Mc} which complements the results of Algom-Kfir and Rafi \cite{AR}.
\begin{corollary} \label{C:McMullen analogue}
Let $\fee \in \Out(F_N)$ (where $N\ge 2)$ be an element represented by an expanding irreducible train track map $f \colon \Gamma \to \Gamma$, and let $X_f$, $\A_{X_f}$, and $\mathfrak H$ be as in Theorem \ref{T:new entropy main}. The function
\[ D \colon \A \to \R_+\]
given by $D(u) = \mathfrak H(u) u(\epsilon)$ is continuous and constant on rays.  In particular, for any compact subset $K \subset \A$, $D$ is bounded on the cone over $K$, 
\[ C_K = \{ c \tau \mid c \in \R_+, \, \tau \in K \}.\]

Consequently, if $dim(H^1(G;\R)) \geq 2$, then $G \cong F_{N_k} \rtimes_{\fee_k} \Z$ for some sequence of automorphisms $\fee_k \in \Out(F_{N_k})$, with $N_k\to \infty$, represented by expanding irreducible train track maps, and there are constants $0 < c_1 < c_2 < \infty$ so that
\[ c_1 \leq N_k \log \lambda(\fee_k) \leq c_2.\]
If $\fee$ is fully irreducible and hyperbolic so are all $\fee_k$.
\end{corollary}
\begin{proof}
Continuity follows from linearity of $u \mapsto u(\epsilon)$ and continuity of $\mathfrak H$.  The fact that $D$ is constant on rays comes from the canceling homogeneity degrees: $+1$ for $\epsilon$ and $-1$ for $\mathfrak H$.  The second statement follows from the fact that $D(K) = D(C_K)$, and $K$ is compact.  The existence of $\fee_k$ and constants $0 < c_1 < c_2 < \infty$ follows from the first statement and Theorems \ref{T:folded basics} and \ref{T:new iwip main} by taking any sequence $\{u_k\} \subset \A$ of primitive integral elements for which the corresponding rays $\R_+ u_k$ converge to a point in $\A$.  The last statement is a consequence of Theorem C, part (3).
\end{proof}

\begin{remark}\label{R:alternative}
It is possible to avoid the use of Theorem~\ref{T:topological graph to graph} in the proof of Theorem~\ref{T:new entropy main} and in Theorem~\ref{T:new iwip main} part (2). 

Namely, suppose $f\colon \Gamma\to\Gamma$ is a regular topological graph-map of a finite connected topological graph $\Gamma$ with an irreducible transition matrix $A(f)$.
It is straightforward to find a metric structure $\mathcal L$ on $\Gamma$ and a regular linear graph map $f'\colon  (\Gamma,\mathcal L)\to (\Gamma, \mathcal L)$ such that $f'$ is isotopic to $f$ and has ``the same combinatorics'' as $f$ and such that $f'$ is a local $\lambda(f)$--homothety on edges with respect to $d_\mathcal L$.
By ``the same combinatorics'' we mean that for any $e\in E\Gamma$ and $n\ge 1$ we have $f^n(e)=(f')^n(e)$ as combinatorial edge-paths. Thus $A(f)=A(f')$, $\lambda(f)=\lambda(f')$,  and a turn in $\Gamma$ is taken by $f$ if and only if it is taken by $f'$. Hence $\wh_\Gamma(f,v)=\wh_\Gamma(f',v)$ for any $v\in V\Gamma$.  We have $h(f)=h(f')=\log\lambda(f)$ by Proposition~\ref{prop:h(f)}.
Also, if $f^n$ is locally injective on edges for all $n\ge 1$ then $(f')^n$ is locally injective on edges for all $n\ge 1$; that is, $f'$ is a train-track map. 

In the proof of Theorem~\ref{T:new iwip main}, we can first verify that $f_z^n\colon \Theta_z\to\Theta_z$ is locally injective on edges for all $n\ge 1$ and that $A(f_z)$ is irreducible with $\lambda(f_z)>1$, and that all Whitehead graphs of all the vertices of $\Theta_z$ with respect to $f_z$ are connected.  We can then pass to the modified isotopic linear graph-map $f_z'\colon \Theta_z\to\Theta_z$ constructed as above. Then it will follow that $f_z'$ is a (weakly) clean train-track map representing the same outer automorphism $\fee_u$ of $\pi_1(\Theta_z)=\ker(u)$ as $f_z$. Hence $\fee_u$ is fully irreducible by Theorem~\ref{T:clean}, yielding Theorem~\ref{T:new iwip main}. Moreover, since $h(f)=h(f')=\log\lambda(f)=h(\fee_u)$, the proof of Theorem~\ref{T:new entropy main} would go through without appealing to  Theorem~\ref{T:topological graph to graph} as well. We included the proof of Theorem~\ref{T:topological graph to graph} since it appears to be of independent interest and may be useful in the further analysis of the properties of the monodromies $\fee_u$.  Moreover, part (1) of Theorem~\ref{T:new iwip main}, which does require the use of Theorem~\ref{T:topological graph to graph}, provides a stronger analogy with the motivating results of Fried for $3$--manifolds.
\end{remark}




\appendix

\section{Entropy and Expansion} \label{App:Entropy_expansion}

\subsection{Entropy}

The following statement about computing topological entropy of
topological graph-maps appears to be a known ``folklore'' fact for the
setting of linear graph-maps with an irreducible transition matrix,
although we were unable to find an explicit reference for it in the
literature. In that context, this fact is a straightforward corollary
of Theorem~7.1 in the paper~\cite{Bowen70} of Bowen (using the
``canonical eigenmetric'' on the graph provided by Corolary~\ref{cor:eigenmetric} below) which allows one to compute the topological entropy by counting $n$--periodic points of $f$ as $n\to\infty$. The version we provide below is more general and applies to topological graph-maps that may be the identity map on some small non-degenerate subintervals of some edges (thus $f$ could have continuumly many fixed points, in which cases Bowen's result is not applicable) and also which do not necessarily have irreducible transition matrix (so that a ``uniformly expanding'' metric on the graph as required by Bowen's theorem is not necessarily available even in the linear-graph context).

\begin{proposition}\label{prop:h(f)}
Let $\Gamma$ be a finite connected topological graph with at least one edge and let $f\colon\Gamma\to\Gamma$ be regular topological graph-map.  Let $h(f)$ be the topological entropy of $f$. 

Then $h(f)=\log \lambda(f)$ (recall that $\lambda(f)$ is the spectral radius of $A(f)$).
\end{proposition}

\begin{proof}

This statement follows, for example, by a direct application of Lemma~3.2 of~\cite{AMM00}.  In the notations of \cite{AMM00}, for any $x\in \Gamma$ the set $f^{-1}(x)$ is finite and we have $K_f(x)=\{x\}$.  With $A=V\Gamma$, the triple $(\Gamma,A,f)$  satisfies the condition from \cite{AMM00} of being a ``monotone representative of an action''. In the notations of \cite{AMM00}, the partition $\mathcal A_f$ of $\Gamma$ consists of all the vertices and open $1$-cells of $\Gamma$ and the partition $\mathcal A_f^n$ of $\Gamma$ consists of all the $f^n$-preimages of vertices and of all the connected components of the $f^n$-preimages of open 1-cells. Thus, up to an additive constant independent of $n$, the cardinality $\#(\mathcal A_f^n)$ of the partition $\mathcal A_f^n$ is equal to $\sum_{e\in E_+\Gamma} |f^n(e)|$ (where $E_+\Gamma$ comes from any choice  an orientation on $\Gamma$). 
In other words, $\#(\mathcal A_f^n)$ is equal, up to a finite additive
error independent of $n$, to the sum of all the entries of the matrix $(A(f))^n$. 

By  Lemma~3.2 of~\cite{AMM00} we have
\[
h(f)=\lim_{n\to\infty} \frac{1}{n} \log  \#\mathcal (A_f^n). \tag{$\ast$}
\]

The assumptions on $f$ imply that $\lambda(f)\ge 1$.  If $\lambda(f)=1$ then the sum of all the entries of $(A(f))^n$ grows polynomially in $n$ and equation $(\ast)$ yields $h(f)=0$. Thus we have $h(f)=0=\log 1 =\log \lambda(f)$, as required.

If $\lambda>1$, then there is a constant $c\ge 1$ independent of $n$ such that for every $n\ge 1$ we have $\frac{1}{c} \lambda(f)^n \le     \#\mathcal (A_f^n)\le c\lambda(f)^n$. Then by  $(\ast)$ we again have $h(f)=\log\lambda(f)$, as required.
\end{proof}

\subsection{Expansion and promotion of topological graph-maps}\label{sect:expansion}

Given a topological graph-map $\f\colon\Gamma \to \Gamma$, we say that $\f$ is {\em expanding on all scales} if for every $x \in \Gamma$, every neighborhood $U$ of $x$, and every edge $e$ of $\Gamma$, there exists a positive integer $m$ and an open interval $W \subset U$ so that $\f^m$ maps $W$ homeomorphically onto $e$.  Observe that as an immediate consequence, $\f(e)$ is always a nondegenerate combinatorial edge-path (it cannot be constant on any open set).

Let us first observe that for graph maps, expanding on all scales is a rather mild condition that can be deduced from the combinatorial data.

\begin{lemma}\label{L:linear maps expanding on all scales}
Let $f\colon \Gamma\to \Gamma$ be a graph map of a finite (linear) graph $\Gamma$. Suppose that $f$ is expanding and that $A(f)$ is irreducible. Then $f$ is expanding on all scales.
\end{lemma}
\begin{proof}
The graph $\Gamma$ comes equipped with a linear structure $\Lambda$. For each edge $e$ and any chart $(e_\alpha,j_\alpha,[a_\alpha,b_\alpha])$ with $e_\alpha=e$, the composition $f\circ j_\alpha\colon [a_\alpha,b_\alpha]\to\Gamma$ is a PL edge-path. Thus if $e_1\dots e_n$ is the combinatorial edge-path $f(e)$ then there is a partition $a_\alpha=c_0 < c_1\dots < c_n = b_\alpha$ of $[a_\alpha,b_\alpha]$ so that $f\circ j_\alpha$ restricts to a homeomorphism of $(c_{i-1},c_i)$ onto $e_i$ for each $i=1,\dotsc,n$. We then define the \emph{subdivision ratios} of the edge $e$ to be the numbers $\rho^e_i = \frac{c_i-c_{i-1}}{b_\alpha-a_\alpha}$ for $i=1,\dotsc,n$. These numbers all satisfy $0 < \rho^e_i < 1$ unless $n=1$. Notice that since the transition functions between charts are affine, these numbers do not depend on the choice of the chart $j_\alpha\colon[a_\alpha,b_\alpha]\to \Gamma$. Since $\Gamma$ has finitely many edges, there are only finitely many subdivision ratios, and we let $\rho$ denote the maximal ratio that is strictly less than $1$.

For an edge $e_0\in E\Gamma$ we will now analyze what happens to $e_0$ under iteration by $f$. By the maximality condition for linear structures, $\Lambda$ contains a chart of the form $(e_0,j,[0,1])$. For $k\geq 0$ the set $j^{-1}\circ f^{-k}(V\Gamma)$ partitions $[0,1]$ into subintervals, each of which is mapped homeomorphically onto an edge of $\Gamma$ by $f^k$. Let $\mathcal{S}_k$ denote the set of intervals in this subdivision. Suppose that an interval $I\in \mathcal{S}_k$ is mapped homeomorphically onto $e\in E\Gamma$ by $f^k$. Then $I$ becomes further subdivided in $\mathcal{S}_{k+1}$ into subintervals of lengths $\len(I)\rho^e_{1},\dotsc,\len(I)\rho^e_{|f(e)|}$. In particular, if  $|f(e)|\geq 2$, then each of these subintervals has length at most $\len(I)\rho$.

Let $l_k$ denote the maximal length of any subinterval in $\mathcal{S}_k$. Notice that, since intervals are only ever subdivided (i.e., $f^{-k}(V\Gamma)\subset f^{-k-1}(V\Gamma)$), $l_k$ is nonincreasing. We claim that $l_k$ in fact tends to $0$ as $k\to \infty$. Indeed, since $f$ is expanding, there exits $M\geq 1$ so that $|f^M(e)|\geq 2$ for every $e\in E\Gamma$ (notice that $|f^n(e)|$ is nondecreasing function of $n$). It follows that every interval in $\mathcal{S}_k$ becomes subdivided in $\mathcal{S}_{k+M}$. Therefore $l_{k+M}\leq l_{k}\rho$ for all $k\geq 0$ and we conclude $l_k\to 0$.

We now prove that $f$ is expanding on all scales. Given $x\in \Gamma$ and a neighborhood $U$ of $x$, we may find an open interval $U'\subset U$ that is contained in some edge $e_0$ of $\Gamma$. The above discussion shows that for some sufficiently large $k$, $U'$ contains an interval $W'\subset U'$ that is mapped homeomorphically over some edge $e_1$. Let $e\in E\Gamma$ be any edge. Since $A(f)$ is irreducible, there exists a power $n$ so that the combinatorial edge-path $f^n(e_1)$ crosses $e$ or $e^{-1}$. Therefore, there is a further subinterval $W\subset W'$ so that $f^{k+n}(W)$ is mapped homeomorphically onto $e$.
\end{proof}

While a topological graph map is always homotopic to a graph map with
respect to a linear structure, the next theorem says that if $f$ is expanding on all scales, then $f$ is already linear with respect to some linear structure.

Recall that if $\Gamma$ is a topological graph and $e\in E\Gamma$ then $\overline e$ denotes the closure of $e$ in $\Gamma$.

\begin{theorem} \label{T:topological graph to graph}
Suppose $\f\colon\Gamma \to \Gamma$ is a regular topological graph-map that is expanding on all scales.  Choose an orientation on $\Gamma$ and an ordering $E_+\Gamma = \{e_1,\ldots,e_n\}$. 

Then the topological graph-map $\f$ is expanding, the matrix $A(\f)$ is irreducible with the Perron-Frobenius eigenvalue  $\lambda=\lambda(f)>1$, and  there exists a unique volume-1 metric structure $\mathcal L$ on $\Gamma$ with the following properties:

\begin{enumerate}
\item $\f \colon \Gamma \to \Gamma$ is linear with respect to $\mathcal L$;
\item $\f$ is a local $\lambda$--homothety, meaning that the norm of the derivative of $\f$ at every point of $\Gamma \setminus \f^{-1} (V\Gamma)$ is $\lambda$; and
\item For every $i=1,\dots, n$ the $d_\mathcal L$--length of $e_i$ is equal to $v_i$ where ${\bf v}=(v_1,\dots, v_n)^T$ is the left Perron-Frobenius eigenvector (i.e.~for the left action of $A(\f)$ on column vectors) with $\sum_{i=1}^n v_i =1$.
\end{enumerate}
\end{theorem}
\begin{proof}
We assume for simplicity that $\Gamma$ has no loop-edges.  It will be clear from the proof how to adjust the argument to cover the case where some loop-edges are present.

We will first establish the existence of a metric $d_\mathcal L$ satisfying conditions (1) and (2) of the theorem.
The assumptions on $\f$ imply that $A(\f)$ is irreducible.  Indeed, given any two edges $e_i,e_j\in E_+\Gamma$, the edge $e_j$ is a neighborhood of any point $x \in e_j$, and as such there exists $m$ and an open interval $W$ in $e_j$ so that $\f^m$ maps $W$ homeomorphically onto $e_i$.   Therefore, there is at least one occurrence of $e_i$ or $e_i^{-1}$ in $\f^m(e_j)$, and hence $(A(\f)^m)_{ij} \neq 0$.  

Clearly, $A(\f)$ is not a permutation matrix and thus the topological graph-map $f$ is expanding.

Now let $\Gamma_0 = \Gamma$ and for $k \geq 1$ define $\Gamma_k$ to be the graph obtained by subdividing $\Gamma_{k-1}$ along $\f^{-1}(V\Gamma_{k-1})$.  By construction $V\Gamma_{k-1} \subset V\Gamma_k$ for all $k \geq 0$.  Furthermore, since $\f(V\Gamma) \subset V\Gamma$, it follows that $V\Gamma_k = (\f^k)^{-1}(V\Gamma)$.  Given any edge $e_{k,0} \in E\Gamma_k$ for any $k \geq 0$, the subdivision makes $e_{k,0}$ into a combinatorial edge-path $\gamma_{e_{k,0}} = e_{k+1,1} e_{k+1,2} \ldots e_{k+1,r}$ in $\Gamma_{k+1}$.  Also, for every $k \geq 1$, we have $\f(e_{k,0})$ is an edge $e_{k-1,0} \in E\Gamma_{k-1}$.

\begin{claim}
The set
\[ \mathcal V =  \bigcup_{k =0}^\infty V\Gamma_k \]
is dense in $\Gamma$.
\end{claim}
\begin{proof}[Proof of claim.]
Given any $x \in \Gamma$ and any neighborhood $U$, we need only prove that $V\Gamma_k \cap U \neq \emptyset$ for $k$ sufficiently large.  Let $e$ be any edge of $\Gamma$ and let $m$ and $W \subset U$ be as in the definition of $\f$ being expanding on all scales.  Without loss of generality, we can assume that $\overline W \subset U$ (by replacing $U$ with a strictly smaller neighborhood before finding $m$ and $W$).  Then $\f^m(\overline W) = \overline e$, and hence $(\f^m)^{-1}(o(e) \cup t(e)) \subset U$.  Therefore $(\f^m)^{-1}(V\Gamma) = V\Gamma_m$ nontrivially intersects $U$, and we set $k = m$.
\end{proof}

\begin{claim}
There exists a sequence of length functions $\{L_k\colon E \Gamma_k \to \R_+\}_{k =0}^\infty$ with the following two properties.
\begin{enumerate}
\item For every edge $e_{k,0} \in E\Gamma_k$, with $k \geq 0$, if $\gamma_{e_{k,0}} = e_{k+1,1} e_{k+1,2} \ldots e_{k+1,r}$ is the combinatorial edge-path representing $e_{k,0}$ in $E\Gamma_{k+1}$, then
\[ L_k(e_{k,0}) = \sum_{\ell = 1}^r L_{k+1}(e_{k+1,\ell}).\]
\item For every edge $e_{k,0} \in E\Gamma_k$,  with $k \geq 1$, if we write $\f(e_{k,0}) = e_{k-1,0}$, then
\[ \lambda L_k(e_{k,0}) = L_{k-1}(e_{k-1,0}).\]
\end{enumerate}
\end{claim}
\begin{proof}[Proof of Claim.]
To construct this sequence of length functions, we define $L_0$ first to be given by $L(e_i) = L(e_i^{-1}) = v_i$, where ${\bf v} = (v_1,\ldots,v_n)$ is the Perron-Frobenius eigenvector for $\lambda$.  Then for $k \geq 1$ define $L_k$ in terms of $L_{k-1}$ inductively by the second condition.  Equivalently, we note that for every $k \geq 1$ and every edge $e_{k,0} \in E\Gamma_k$, we have that $\f^k$ maps $e_{k,0}$ homeomorphically onto some edge $e_i \in E\Gamma$ and so its $L_k$--length is $\lambda^{-k} v_i$.

It remains to verify that the first property holds.  Fix $e_{k,0}$ and we let $i$ be as above so that $\f^k$ maps $e_{k,0}$ homeomorphically onto $e_i$.  The next iteration of $\f$ maps $e_i$ over each edge $e_j$ exactly $a_{ij}$ times and we have
\begin{equation} \label{E:PF2length} \lambda v_i = \sum_{i=1}^n a_{ij} v_j.
\end{equation}
On the other hand, writing
\[ e_{k,0} = e_{k+1,1} e_{k+1,2} \ldots e_{k+1,r}\]
and letting $j(\ell)$ be such that $\f^{k+1}$ maps $e_{k+1,\ell}$ homeomorphically onto $e_{j(\ell)}$, it follows that
\[ \f(e_i) = e_{j(1)} e_{j(2)} \ldots e_{j(r)}.\]
and so comparing with \eqref{E:PF2length} we have
\[ \lambda v_i = \sum_{\ell = 1}^r v_{j(\ell)}.\]
The $L_{k+1}$--length of $e_{k+1,\ell}$ is $\lambda^{-(k+1)} v_{j(\ell)}$ and therefore we have
\[ L_k(e_{k,0}) = \lambda^{-k} v_i = \lambda^{-k} \left( \lambda^{-1} \sum_{\ell = 1}^r v_{j(\ell)} \right) = \lambda^{-(k+1)} \sum_{\ell = 1}^r v_{j(\ell)} = \sum_{\ell = 1}^r L_{k+1}(e_{k+1,\ell}).\qedhere\]
\end{proof}
We now prove that there exists a metric structure with the required properties.  As a consequence of the second condition, we note that the $L_k$--length of any edge of $E\Gamma_k$ is at most $\lambda^{-k}$ times the maximum $L_0$--length of the edges in $E\Gamma_0 = E\Gamma$.  In particular, the maximum $L_k$--length of any edge in $E\Gamma_k$ tends to zero as $k \to \infty$.

Now suppose that $e \in E_+\Gamma$ and that $x,y \in \mathcal V \cap \overline{e}$ are any two points.  We define a ``distance in $e$''  $d_{\mathcal V,e}(x,y)$ between $x$ and $y$ to be the sum of the $L_k$--length of the edges between $x$ and $y$ inside $e$, provided $x,y \in V\Gamma_k$ (which is valid for all sufficiently large $k$, depending on $x$ and $y$).  By the first condition, this definition of $d_{\mathcal V,e}(x,y)$ is independent of $k$ and hence we get  a metric $d_{\mathcal V,e}$ on $\mathcal V \cap \overline{e}$.  Moreover, by construction, with respect to the ordering on $\mathcal V\cap \overline{e}$ induced by the orientation on $e$, if $x < y < z$ are three points in $\mathcal V \cap \overline{e}$, then
\[ d_{\mathcal V,e}(x,z) = d_{\mathcal V,e}(x,y) + d_{\mathcal V,e}(y,z). \]
Because the $L_k$--lengths of edges in $E\Gamma_k$ tend to zero and because $\mathcal V \cap e$ is dense in $\overline{e}$, we can extend $d_{\mathcal V,e}$ to a geodesic metric $d_e$ on $\overline{e}$ (inducing the given topology): for an arbitrary pair of points $x',y'$ in $\overline e$, we can approximate $d_e(x',y')$ from above (respectively, below) by $d_{\mathcal V,e}$--distances between points in $\mathcal V$ limiting to $x',y'$ from outside (respectively, inside) the arc between $x'$ and $y'$.   For $x,y\in  \mathcal V \cap \overline{e}$ we will have $d_e(x,y)=d_{\mathcal V,e}(x,y)$.

The collection of metrics $d_e$ on edges of $\Gamma$ determines a geodesic metric $d$ on $\Gamma$ consistent with the original topology on $\Gamma$.
By construction of $d_e$ and using local injectivity of  $f$ on edges,  it follows that for any edge $e\in E\Gamma$ and any sufficiently close points $x,y\in \mathcal V$  we have $d(f(x),f(y))=\lambda d(x,y)$.  Since $\mathcal V$ is dense in $e$,  it follows that $f$ is locally a $\lambda$--homothety on edges with respect to $d$, as required. This completes the proof of the existence of a metric structure $\mathcal L$ satisfying conditions (1) and (2) of the theorem.

To see that such $\mathcal L$ is unique note that condition (2) of the theorem uniquely determines the $d_\mathcal L$-lengths of the edges of $\Gamma$.
Proceeding inductively on $k$, we see that conditions (1) and (2) then uniquely determines the $d_\mathcal L$-lengths of all edges of $\Gamma_k$ for all $k\ge 0$.  For $k=0$ we have $\Gamma_0=\Gamma$ and the statement holds.
If $k\ge 1$ and $e\in E\Gamma_k$ then $e'=f(e)$ is an edge of $\Gamma_{k-1}$, so that, by the inductive hypothesis,  the $d_\mathcal L$ length of $e'$ is already determined. Then by condition (1) it follows that the $d_\mathcal L$-length of $e$ is equal to $1/\lambda$ times the $d_\mathcal L$ length of $e'$, which completes the inductive step.

It thus follows that conditions (1) and (2) of the theorem uniquely define the $d_\mathcal L$ length along an edge of $\Gamma$ between any two points of $\mathcal V$ belonging to that edge. Since the set $\mathcal V$ is dense in $\Gamma$, the uniqueness of $\mathcal L$ follows.
\end{proof}

\begin{remark}
This proof follows closely the strategy for constructing the stable foliation for a pseudo-Anosov in \cite[Expos\'e 9.5]{FLP}.
\end{remark}

Lemma~\ref{L:linear maps expanding on all scales} and Theorem~\ref{T:topological graph to graph} immediately imply:

\begin{corollary}\label{cor:eigenmetric}
Let $\Gamma$ be a finite (linear) graph and let $f\colon\Gamma\to\Gamma$ be a regular graph-map such that $f$ is expanding and $A(f)$ is irreducible.

Then there exists a (possibly different) linear structure $\Lambda$ and compatible volume-$1$ metric structure $\mathcal L \subset \Lambda$ on the underlying topological graph $\Gamma$ such that, with respect to $\mathcal L$, the map $f\colon \Gamma\to \Gamma$ is still linear and the restriction of $f$ to any edge is affine with constant derivative $\lambda(f)$. (Thus \emph{locally} on every edge, $f$ is a $\lambda(f)$--homothety with respect to $d_\mathcal L$).
\end{corollary}

We respectively call the objects $\Lambda$, $\mathcal L$, and $d_{\mathcal L}$ provided by Corollary~\ref{cor:eigenmetric} (or Theorem \ref{T:topological graph to graph}) the \emph{canonical linear structure}, \emph{canonical metric structure}, and \emph{canonical eigenmetric} associated to the given map $f\colon \Gamma\to\Gamma$.
These canonical structures lead to the following rigidity statement.  First we say that two regular topological graph maps $\f_1,\f_2 \colon \Gamma \to \Gamma$ are {\em combinatorially equivalent} if for every $e \in E\Gamma$, we have $\f_1(e) = \f_2(e)$ as combinatorial edge-paths.

\begin{theorem}\label{T:comb_equiv_are_conj}
Suppose $\Lambda_1,\Lambda_2$ are linear structures on a graph $\Gamma$ and suppose that for each $i=1,2$, $\f_i \colon \Gamma \to \Gamma$ is an expanding $\Lambda_i$--linear graph map.  If $\f_1$ and $\f_2$ are combinatorially equivalent and $A(f_1)=A(f_2)$ is irreducible, then $\f_1$ and $\f_2$ are conjugate by a homeomorphism $h \colon \Gamma \to \Gamma$ homotopic to the identity: 
\[ h \f_1 h^{-1} = \f_2.\]
\end{theorem}
\begin{proof}[Proof sketch] Let $d_i$ be the canonical eigenmetric associated to $f_i$, for $i=1,2$. If we write $d_i^e$ for the induced path metric on the interior of $e\in E\Gamma$, then the homeomorphism $h\colon \Gamma\to \Gamma$ is the unique map that fixes $V\Gamma$ and restricts to an isometry $(e,d_1^e)\to (e,d_2^e)$ for each edge $e\in E\Gamma$.
\end{proof}

\begin{corollary}\label{cor:find_combinatorial_metric}
Let $f\colon \Gamma\to\Gamma$ be a regular topological graph-map that is expanding on all scales. Then there exits a metric structure $\mathcal L$ on $\Gamma$ with respect to which $f$ is a combinatorial graph-map. 
\end{corollary}
\begin{proof}
First equip $\Gamma$ with the linear structure $\Lambda_1$ provided by Theorem~\ref{T:topological graph to graph}, so that $f_1 = f\colon \Gamma\to\Gamma$ is an expanding irreducible (linear) graph-map. Next choose any metric structure $\mathcal L_2$ (with associated linear structure $\Lambda_2$) for which each edge has length $1$. As indicated in Remark~\ref{rem:make_combinatorial}, we may then build a combinatorially equivalent (linear) graph-map $f_2\colon \Gamma\to\Gamma$ that is homotopic to $f$ rel $V\Gamma$ and such that the subdivision of each edge along $f_2^{-1}(V\Gamma)$ is into equal length parts with respect to $\mathcal L_2$. Hence $f_2\colon \Gamma\to\Gamma$ is a combinatorial graph map with respect to $\mathcal L_2$.

Applying Theorem~\ref{T:comb_equiv_are_conj} to the maps $f_1,f_2\colon \Gamma\to\Gamma$ with associated linear structures $\Lambda_1,\Lambda_2$, we obtain a homeomorphism $h\colon \Gamma\to\Gamma$ such that $hf_1 = f_2 h$. Pulling back $\mathcal L_2$ via $h$ now gives a metric structure $\mathcal L$ on $\Gamma$ with respect to which $f = f_1$ is a combinatorial graph map.
\end{proof}

\section{Clean train-tracks and full irreducibility}\label{sect:clean}

In this appendix  we obtain a strengthened and easier-to-use version of a criterion from \cite{K12} for checking the full irreducibility of hyperbolic automorphisms of $\Out(F_N)$.

Recall that according to our Definition~\ref{D:train-track}, for a train-track map $\f\colon\Gamma\to\Gamma$ we allow the graph $\Gamma$ to have valence-2 vertices but no valence-1 vertices. However we do require that for any vertex $v$ of valence $2$ and any $n\ge 1$ the map $\f^n$ be locally injective on some small neighborhood of $v$ in $\Gamma$. This condition is automatically satisfied if we only require $\f^n$ to be an immersion on all edges for all $n$ and  also require the Whitehead graphs of all valence-2 vertices to be connected.
Recall from Definition~\ref{D:reg-exp-graph-map} that a train-track map $\f\colon\Gamma\to\Gamma$ is \emph{expanding} if for every edge $e$ of $\Gamma$ the simplicial length $|\f^n(e)|$ tends to $\infty$ as $n\to\infty$. 

Recall also from Definition~\ref{defn:clean_track} that an expanding train-track map $f\colon \Gamma\to \Gamma$ whose Whitehead graphs $\wh(\f,v)$ are connected for all $v\in V\Gamma$ is \emph{clean} if there exits $m>0$ such that $A(f^m)>0$ and is \emph{weakly clean} if $A(f)$ is irreducible. By definition, a clean train-track is weakly clean; here we show that the converse is also true.

\begin{remark}\label{rem:b1}
Note that if $\Delta$ is a finite connected subgraph of a graph $\Gamma$ then the inclusion $\iota\colon \Delta\to \Gamma$ induces an injective homomorphism $\iota_\ast\colon H_1(\Delta; \mathbb R)\to H_1(\Gamma; \mathbb R)$. For this reason if, in this situation, $\f\colon\Gamma\to \Gamma$ is a graph-map which is homotopy equivalence, then $b_1(\Delta)\le b_1(f(\Delta))$. 
\end{remark}

\begin{proposition}\label{prop:clean}
Let $\f\colon\Gamma\to \Gamma$ be a  weakly clean train-track. Then $\f$ is clean.
\end{proposition}

\begin{proof}
Since $\f$ is expanding and $\Gamma$ is finite, for every oriented edge $e$ of $\Gamma$ there exists $n>0$ and an oriented edge $e'$ such that the path $\f^n(e)$ contains at least three occurrences of the edge $e'$ (with positive orientation). From irreducibility of $A(\f)$ it follows that there exists $j>0$ such that $\f^j(e')$ contains either $e$ or $e^{-1}$. Then $\f^{n+j}(e)$ contains at least 3 occurrences of $e$ or at least 3 occurrences of $e^{-1}$. Hence $\f^{2(n+j)}(e)$ contains at least 3 occurrences of $e$.  Put $k(e):=2(n+j)$ and then put $m$ to be the product of $k(e)$ taken over all oriented edges $e$ of $\Gamma$. Then for any oriented edge $e$ the path $\f^m(e)$ contains at least three positive occurrences of $e$ and the same is true for $\f^{mq}(e)$ for any integer $q\ge 1$.

Now let $e$ be an oriented edge of $\Gamma$ and for $n\ge 1$ put $\Delta_0=\{e\}$ and let $\Delta_n$ be the subgraph of $\Gamma$ given by the union of all the edges in the path $\f^{mn}(e)$. Then, by construction, $\Delta_n\subseteq \Delta_{n+1}$ for all $n\ge 0$ and we put $\Delta(e):=\cup_{n\ge 0} \Delta_n$.
Since $\Gamma$ is finite, and the graphs $\Delta_n$ are connected, $\Delta(e)$ is a finite connected subgraph of $\Gamma$ containing $e$ and such that $\Delta(e)=\Delta(e^{-1})$.
Also, by construction, we have $\f^m(\Delta(e))=\Delta(e)$.

{\bf Claim 1.} For any edge $e$ of $\Gamma$ and any $j>0$ the graphs $\Delta(e)$ and $\f^j(\Delta(e))$ are homotopy equivalent.
Since $\f\colon\Gamma\to\Gamma$ is a homotopy equivalence,  Remark~\ref{rem:b1} implies that $b_1(\Delta(e))\le b_1(\Delta')$ where $\Delta'=\f^j(\Gamma(e))$. We can find $k\ge 1$ such that $j<mk$. Then $\f^{mk-j}(\Delta')=\f^{mk}(\Delta(e))=\Delta(e)$ and hence, again by  Remark~\ref{rem:b1}, we have $b_1(\Delta')\le b_1(\Delta(e))$. Hence $b_1(\Delta')=b_1(\Delta(e))$, as required, and Claim~1 is established.

{\bf Claim 2.} For any two edges $e, e'$ of $\Gamma$ we have $b_1(\Delta(e))=b_1(\Delta(e'))$ (and hence the graphs $\Delta(e)$ and $\Delta(e')$ are homotopy equivalent).

Indeed, let $e,e'$ be arbitrary oriented edges. By irreducibility of $A(\f)$ there exists $j>0$ such that $\f^j(e)$ contains $e'$ or $(e')^{-1}$. 
Then for any $k\ge 0$ $\f^{mk} (\f^j \Delta(e)) = \f^{j+mk}(\Delta(e))=\f^j \big( \f^{mk}(\Delta(e)) \big)=\f^j(\Delta(e))$.  Hence $\f^{mk}(e')$ is contained in $\f^j(\Delta(e))$ for every $k\ge 1$ and therefore, by definition, $\Delta(e')\subseteq \f^j(\Delta(e))$. Hence $b_1(\Delta(e'))\le b_1(\f^j(\Delta(e)))=b_1(\Delta(e))$. A symmetric argument shows that $b_1(\Delta(e))\le b_1(\Delta(e'))$ and hence $b_1(\Delta(e))= b_1(\Delta(e'))$. Thus Claim~2 is verified.

{\bf Claim 3.} For any edge $e$ of $\Gamma$ we have $\Delta(e)=\Gamma$.

Indeed, suppose not, and let $e$ be such that $\Delta(e)\ne \Gamma$. Then there exists a vertex $v$ of $\Delta(e)$ such that $v$ is incident to an edge of $\Delta(e)$ and to an edge of $\Gamma - \Delta(e)$.

Since the Whitehead graph $\wh_\Gamma(\f,v)$ is connected, there exist oriented edges $e_1$ and $e_2$ with origin $v$ such that $e_1\subseteq \Delta(e)$, $e_2\not\subseteq \Delta(e)$ and such that the turn $e_1,e_2$ is taken by $\f$. Hence there exist an edge $e_3$ and $i>0$ such that $e_1^{-1}e_2$ is a subpath of $\f^i(e_3)$. Since $\f\colon\Gamma\to\Gamma$ is surjective, there exist $k\ge 1$ and an edge $e_4$ such that $km>i$ and such that $\f^{km-i}(e_4)$ contains $e_3$. Hence $\f^m(e_4)$ contains $e_1^{-1}e_2$ as a subpath. Therefore, by definition, we have $e_1 \subseteq \Delta(e_4)$ and $e_2\subseteq \Delta(e_4)$ and hence $\Delta(e_1)\subseteq \Delta(e_4)$ and $\Delta(e_2)\subseteq \Delta(e_4)$. 

Also, since $\f^m(\Delta(e))=\Delta(e)$ and since $e_1$ is an edge of $\Delta(e)$, we have $\Delta(e_1)\subseteq \Delta(e)$. In particular, $e_2$ is not contained in $\Delta(e_1)$.

Let $v'$ be the terminal vertex of $e_2$. Then $v'\not\in \Delta(e_1)$ since otherwise $\Xi:=e_2\cup \Delta(e_1)$ has the first Betti number bigger than that of $\Delta(e_1)$ which is impossible since $\Xi\subseteq \Delta(e_4)$ and $b_1(\Delta(e_1))=b_1(\Delta(e_4))$ by Claim~2. 

By the choice of $m$ the path $\f^m(e_2)$ contains at least 3 positive occurrences of $e_2$. Hence there exists a fixed point $x$ of $\f^m$ in the interior of $e_2$ coming from a positive occurrence of $e_2$ in $\f^m(e_2)$ which is neither initial nor terminal edge of $\f^m(e_2)$. Thus $\f^m$ preserves the orientation in a small neighborhood of $x$. Let $[x,v']$ be the segment of $e_2$ from $x$ to $v'$. Then $f^m([x,v'])=[x,v']\cup \beta$ where $\beta$ is a tight combinatorial edge-path of positive simplicial length with initial vertex $v'$. Thus for all $n\ge 0$ the path $\f^{mn}([x,v'])=[x,v']\cup \beta_n$ is a proper initial segment of the path $\f^{m(n+1)}([x,v'])=[x,v']\cup \beta_{n+1}$, where $\beta_n$ is a tight combinatorial edge-path starting at $v'$ with $|\beta_n|\to\infty$ as $n\to\infty$. Let $\rho$ be the path in $\Gamma$  obtained by taking the ``union'' of all $\f^{mn}([x,v'])$, that is, $\rho$ is the path starting at $x$ such that for all $n\ge 0$ $\f^{mn}([x,v'])$  is an initial segment of $\rho$. This $\rho$ is sometimes called a ``combinatorial eigenray'' of $\f$. By construction we have $\rho\subseteq \Delta(e_2)$.
The path $\rho$ has the form $\rho=[x,v']\rho'$ where $\rho'$ is a semi-infinite edge-path in $\Gamma$ starting at $v'$.

Then either $\rho'$ has a nonempty intersection with $\Delta(e_1)$ (in which case we look at the first time when $\rho'$ hits the graph $\Delta(e_1)$), or, if $\rho'$ and $\Delta(e_1)$ are disjoint, then there is an oriented edge that occurs in $\rho'$ at least twice and hence $\rho'$ contains a nontrivial immersed loop. In both cases we see that $\Xi_1=\Delta(e_1)\cup e_2 \cup \rho'$ has $b_1(\Xi_1)>b_1(\Delta(e_1))$. However, $\Xi_1\subseteq \Delta(e_4)$ and hence $b_1(\Delta(e_4))> b_1(\Delta(e_1))$, which contradicts Claim~2.

Thus Claim~3 is established and we know that $\Delta(e)=\Gamma$ for every edge $e$ of $\Gamma$. By definition of $\Delta(e)$ it follows that for every edge $e$ of $\Gamma$ there exists $n_e\ge 1$ such that $\f^{n_em}(e)$ passes through every topological edge of $\Gamma$. Since $\f\colon\Gamma\to\Gamma$ is surjective, it follows that for every $i\ge n_em$ the path $\f^i(e)$ also passes through every topological edge of $\Gamma$. Now put $M$ to be the maximum of $n_em$ taken over all edges $e$ of $\Gamma$. Then for any two oriented edges $e,e'$ of $\Gamma$ the path $\f^M(e)$ contains an occurrence of $e'$ or $(e')^{-1}$. This means that $A(\f^M)>0$, as required.
\end{proof}

Note that if $\fee\in \Out(F_N)$ is hyperbolic (i.e., atoroidal) and $\f\colon\Gamma\to\Gamma$ is a train-track representative of $\fee$ with $A(f)$ irreducible, then $\f$ is expanding (since otherwise $A(\f)$ is a permutation matrix and $\fee$ has finite order, which contradicts the hyperbolicity assumption on $\fee$). Moreover, it is proved in \cite{K12} that for $\fee\in \Out(F_N)$ hyperbolic, the element $\fee$ is fully irreducible if and only if some (equivalently, any) train-track representative of $\fee$ is clean. Thus combining Proposition~4.4 of \cite{K12} with the equivalence between the clean and weakly clean conditions (Proposition~\ref{prop:clean}),  we immediately obtain:

\begin{theorem}\label{T:clean} 
Let $N\ge 3$ and let $\fee\in \Out(F_N)$ be a hyperbolic element. Then the following conditions are equivalent:
\begin{enumerate}
\item The automorphism $\fee$ is fully irreducible.

\item If $\f\colon\Gamma\to\Gamma$ is a train-track representative of $\fee$ then $\f$ is clean.

\item There exists a clean train-track representative of $\fee$.

\item There exists a weakly clean train-track representative of $\fee$.

\end{enumerate}

\end{theorem}

Another useful corollary of Proposition~\ref{prop:clean} is that for hyperbolic elements of $\Out(F_N)$ the notions of being irreducible and fully irreducible coincide:

\begin{corollary}\label{C:hyp_iwip}
Let $N\ge 3$ and let $\fee\in \Out(F_N)$ be a hyperbolic element. Then $\fee$ is irreducible if and only if $\fee$ is fully irreducible.
\end{corollary}
\begin{proof}
Obviously, if $\fee$ is fully irreducible then $\fee$ is irreducible.
Thus let $\fee\in \Out(F_N)$ be hyperbolic and irreducible. Assume that $\fee$ is not fully irreducible. 
Since $\fee$ is irreducible, by a result of Bestvina--Handel~\cite{BH92} there exists a train-track representative $f\colon\Gamma\to\Gamma$ of $\fee$ with an irreducible transition matrix $A(f)$. Since $\fee$ is hyperbolic and thus has infinite order in $\Out(F_N)$, we know that $A(f)$ is not a permutation matrix. Hence $\lambda(f)>1$, so that $f$ is expanding.
By assumption $f$ is not fully irreducible and therefore, by Theorem~\ref{T:clean} above, the train-track $f$ is not weakly clean. This means that for some vertex $v$ of $\Gamma$ the Whitehead graph $\wh(f,v)$ is not connected.  Proposition~4.1 of \cite{K12} (which first appears, in implicit form, in the proof of Proposition~4.5 in \cite{BH92}) now implies that $\fee$ is reducible, contrary to our assumption that $\fee$ is irreducible.
Thus $\fee$ is fully irreducible, as required.
\end{proof}

Note that the conclusion of Corollary~\ref{C:hyp_iwip} fails without the assumption that $\fee$ is hyperbolic, as already observed by Bestvina and Handel in \cite{BH92}. For example, let $S$ be a compact connected hyperbolic surface with $m\ge 3$ boundary components, let $F_N=\pi_1(S)$ and let $\fee\in \Out(F_N)$ be induced by a pseudo-Anosov homeomorphism $g$ of $S$. The boundary components of $S$ represent primitive elements of $F_N$ and therefore such $\fee$ is not fully irreducible. However, if $g$ transitively permutes the boundary components of $S$, then $\fee$ is irreducible.

\bibliography{freebycyclic}{}
\bibliographystyle{alphanum}

\end{document}